\documentclass[12pt]{amsart}
\usepackage[top=100pt,bottom=60pt,left=96pt,right=92pt]{geometry}
\newcommand{\vungoc}{V\~u Ng\k{o}c}

\usepackage{stmaryrd}
\usepackage[utf8]{inputenc} % in modern latex versions automatically loaded/included
\usepackage[export]{adjustbox}
\usepackage[mathscr]{eucal}
\usepackage[svgnames]{xcolor}
\usepackage{tikz-cd}
\usepackage{amsmath,amsthm,amssymb,amsfonts,enumerate,layout,tikz,enumitem,verbatim} 
\usepackage[normalem]{ulem} % to strike out text with \sout{...}; attention: if [normalem] is removed then \emph{} underlines instead of making italic

\usetikzlibrary{shapes.geometric, arrows}
\usepackage[style=alphabetic,maxalphanames=4,minalphanames=4]{biblatex}
\usepackage{chngcntr}
\counterwithin{figure}{section}
\numberwithin{equation}{section}

\usepackage{array}

\usepackage{tikz-cd} % provides diagram features of tikz
\usepackage[colorlinks]{hyperref}
\usepackage[T1]{fontenc} % necessary when characters/symbols with dots, accents etc appear
\usepackage{mathtools}

\usepackage{graphicx}
\graphicspath{{figures/}}

\usepackage[font=small,labelfont=bf]{caption} % let's see if we like this layout

\usepackage{subcaption} %to fix the labels on subfigures from (A) to (a) (when using amsart)
\usepackage{fancyhdr}

\definecolor{linkblue}{rgb}{0,0,.6}
\definecolor{citered}{rgb}{.7,0,0}
\hypersetup{colorlinks=true, linkcolor=linkblue, citecolor=citered, urlcolor=linkblue}

\allowdisplaybreaks

\newtheorem{theorem}{Theorem}[section]
\newtheorem{proposition}[theorem]{Proposition}
\newtheorem{corollary}[theorem]{Corollary}
\newtheorem{lemma}[theorem]{Lemma}

\theoremstyle{definition}
\newtheorem{definition}[theorem]{Definition}
\newtheorem{remark}[theorem]{Remark}
\newtheorem{example}[theorem]{Example}

\theoremstyle{plain}

\def\R{{\mathbb R}}
\newcommand{\AGL}{\mathrm{AGL}}

\title{On the affine invariant of simple hypersemitoric systems}
\author{Konstantinos Efstathiou \& Sonja Hohloch \& Pedro Santos}

\begin{document}
\begin{abstract}
Hypersemitoric systems are a class of integrable systems on $4$-dimen{\-}sional symplectic manifolds which only have mildly degenerate singularities and where one of the integrals induces an effective Hamiltonian $S^1$-action and is proper.
We introduce the affine invariant of simple hypersemitoric systems, which is a generalization of the Delzant polytope of toric systems and the polytope invariant of semitoric systems. Along the way, we compute and plot this invariant for meaningful and more and more complicated examples.

\end{abstract}

\maketitle
\tableofcontents

\section{Introduction}
Hamiltonian integrable systems are an important and interesting class of dynamical systems, which describe many physical phenomena. Furthermore, they obey certain conservation laws and have interesting rigidity properties. Examples are common in mathematics, physics and other natural sciences (biology, chemistry, etc.), such as the coupled angular momenta, the spherical pendulum, and the Lagrange, Euler and Kovalevskaya spinning tops. Hamiltonian integrable systems are a field with a long tradition at the intersection of dynamical systems, ODEs, PDEs, Lie theory, symplectic geometry, classical mechanics, and so on, for references see \cite{Perelomov,factorizationproblems,papachristou,delzant1988hamiltoniens,arnol2013mathematical}.

Roughly speaking, a Hamiltonian integrable system, from now on briefly an \textit{integrable system}, contains the maximal number of independent symmetries and
conserved quantities. More precisely, an integrable system is given by a triple $(M,\omega,F=(f_1,...,f_n):M\rightarrow \mathbb{R}^n)$ where $(M,\omega)$ is a $2n$-dimensional symplectic manifold, the functions $f_1,...,f_n$ Poisson commute, and the Hamiltonian vector fields of the functions are linearly independent at almost every point. A point where this linear independence fails to hold is called a singularity. It is at the singularities that the most interesting dynamical behavior of the system occurs. 

The classification of integrable systems has been a central point of study over the past decades. There is a long list of mathematicians that worked on many types of classifications, to name just a few, ranging from topological (Bolsinov \& Fomenko \cite{bolsinov2004integrable} and their school) to symplectic (Delzant \cite{delzant1988hamiltoniens}, Karshon \cite{karshon1999periodic}, Pelayo \& \vungoc \ \cite{pelayo2009semitoric},...), from local (Eliason \cite{eliasson1990normal}, Zung \& Miranda \cite{miranda2004equivariant}, Dullin \& \vungoc \ \cite{dullin2007symplectic}, \vungoc \  \& Wacheux \cite{vu2013smooth},...) to global (Delzant \cite{delzant1988hamiltoniens}, Karshon \cite{karshon1999periodic}, Pelayo \& \vungoc\ \cite{pelayo2009semitoric},...), etc.

When classifying integrable systems one has to make a choice: either go for a very large set of integrable systems which leads to quite involved invariants or restrict to a smaller set of integrable systems which give us more specific invariants. 

In this paper we are interested in integrable systems on a $4$-dimensional symplectic manifold with two conserved quantities one of which generates an $S^1$-symmetry, and with singularities that are either nondegenerate or mildly degenerate (in a manner that will be made more precise later). The ultimate goal is to come up with a classification of these systems. This paper aims at introducing one of the naturally appearing invariants, which is the image of the local actions of the system glued together appropriately. In what follows, we motivate the idea behind this invariant.

From the global symplectic point of view, the first classification result was for toric systems, i.e., integrable systems where the flows of the Hamiltonians are periodic of minimal period $2\pi$, i.e., the induced toric action is effective. Delzant \cite{delzant1988hamiltoniens}, in 1988, showed that if $(M,\omega,F)$ is a toric system then its image $F(M)$ is a rational convex polytope of special type, called a Delzant polytope, which completely classifies the toric system. Furthermore, given any Delzant polytope $\Delta$ one can construct an integrable system $(M',\omega',F')$ such that $F'(M')=\Delta$. We note that toric systems are the ``simplest'' kind of integrable system, where the only type of singularities that appear are  of elliptic type.

In 1999, Karshon \cite{karshon1999periodic} symplectically classified Hamiltonian $S^1$-spaces, i.e., triples $(M,\omega,J)$ where $M$ is a $4$-dimensional compact symplectic manifold and $J$ generates an effective $S^1$-action. The classification is in terms of a labeled graph that encodes information about the fixed points and isotropy groups of the $S^1$-action. Using this classification, Hohloch \& Palmer \cite{hohloch2021extending} were able to prove that for each Hamiltonian $S^1$-space $(M,\omega,J)$ there exists a Hamiltonian $H:M\rightarrow \mathbb{R}$ such that $(M,\omega,(J,H))$ is a hypersemitoric integrable system (see Definition \ref{d.hypersmitoric}). In particular, every Hamiltonian $S^1$-space extends to an integrable system. Hypersemitoric systems have not yet been classified, and to establish one invariant for these systems is the motivation for the present paper.  

Then, in 2009-2011, Pelayo \& \vungoc\ 
\cite{pelayo2009semitoric,pelayo2011constructing} advanced the question of global symplectic classification of integrable systems to systems of simple semitoric type (see Definition \ref{d.semitoric} and what follows), and later on Palmer \& Pelayo \& Tang \cite{semitoricnonsimple} extended this classification to semitoric systems. Semitoric systems are a generalization of toric systems, since all but one of the Hamiltonians generate periodic flows, and the systems allow for more general singularities than in the toric case, called focus-focus singularities.  The classification of semitoric systems is in term of five invariants, with one of them being a generalization of the Delzant polytope of toric systems. We refer to this invariant as the semitoric polytope invariant. 

Hypersemitoric systems are the next natural class of integrable systems when generalizing semitoric systems.  A hypersemitoric system is an integrable system $(M,\omega,(J,H))$ on a $4$-dimensional symplectic manifold such that $J$ is proper and generates an effective $S^1$-action and all singularities are either nondegenerate or parabolic (where parabolic singularities are the simplest type of degenerate singularities). In this paper we are interested in a specific type of hypersemitoric systems called simple. A hypersemitoric system $(M,\omega,F=(J,H))$ is called simple if each connected component of a fiber $F^{-1}(f)$ contains at most one $S^1$ orbit of singular points. The goal of this paper is to contribute to the classification of hypersemitoric systems, by generalizing the semitoric polytope invariant to simple hypersemitoric systems. We call this invariant the affine invariant of a simple hypersemitoric system.
In contrast to toric and semitoric systems, hypersemitoric systems may have fibers that have several components. The presence of disconnected fibers makes defining a generalization of a polytope invariant a more challenging task. 

Compared to a semitoric system, there may be two new main features in a simple hypersemitoric system: a flap (see Section \ref{s.flaps/pleatsdefinition} and Figure \ref{p.flap}) and/or a pleat/swallowtail (see Section \ref{s.flaps/pleatsdefinition} and Figure \ref{p.pleat}). In order to define the affine invariant for a simple hypersemitoric system we focus first on studying systems just admitting one of these structures and then generalize our procedure to general systems. 

When studying examples with hyperbolic-regular points, we noticed that systems with so called curled tori do not fall into the class of hypersemitoric systems due to appearance of degenerate nonparabolic points. To obtain a more general classification of integrable systems, it is important to also understand what happens in these examples. For this reason, in this paper we also compute the affine invariant of two nonhypersemitoric examples. 

\subsection{Our results}The main result of this paper is the following:
\begin{theorem}
\label{t.ourresult}
To each simple hypersemitoric system $(M,\omega,F:=(J,H))$ on a $4$-dimensional compact symplectic manifold $(M,\omega)$ one can associate an affine invariant, which is the image of a map defined by the action variables obtained after decomposing $F(M)$ into leaves and introducing certain vertical cuts. In fact, this affine invariant is a symplectic invariant of the system.
\end{theorem}
Furthermore, we compute certain representatives of the affine invariant for $3$ specific systems, see Figure \ref{f.fullinvariantffinsideflap}, Figure \ref{f.fullinvarianffoutsideflap} and Figure \ref{f.fullinvariantflaponflap} in Section \ref{s.generalhypersemitoric}:
\begin{itemize}
  \item Figure \ref{f.fullinvariantffinsideflap} shows $4$ representatives of the affine invariant of a hypersemitoric system containing two focus-focus values inside of a flap.
  \item Figure \ref{f.fullinvarianffoutsideflap} shows $4$ representatives of the affine invariant of a hypersemitoric system containing two focus-focus values outside of a flap.
  \item Figure \ref{f.fullinvariantflaponflap} shows $4$ representatives of the affine invariant of a hypersemitoric system containing a flap with two elliptic-elliptic values inside of another flap.
\end{itemize}
Theorem \ref{t.ourresult} is proven in Section \ref{s.generalhypersemitoric}. The rough idea behind the proof is as follows:
\begin{itemize}
\item Work layer wise, where each layer corresponds to a connected component of the associated fibers.
\item Introduce a certain choice of cuts, as \vungoc \ \cite{vu2007moment}, for certain critical values of rank $0$ of the system.
\item Make a suitable choice of action coordinates in the resulting simply connected set.
\end{itemize}
In more detail, to obtain Theorem \ref{t.ourresult} we first analyze the situation of a hypersemitoric system exhibiting only a flap or a pleat/swallowtail. In particular, we note that in the case of a flap different approaches can be taken in order to arrive at an affine invariant: one can make a cut for each flap or one can make a cut for each elliptic-elliptic value present in the image of the flap. The first approach is ``smoother'' while the second approach gives more continuity at the expense of introducing more corners on the invariant, see Figure \ref{f.polytopeinvariantsflapwith2elliptic}. This leads to the following statement:
\begin{theorem}
\label{t.resultflaps}
     Let $(M,\omega)$ be a closed $4$-dimensional symplectic manifold:
     \begin{enumerate}
   \item Let $(M,\omega,F=(J,H))$ be a hypersemitoric system exhibiting only standard flaps. Then making a choice of cut for each flap and a suitable choice of action coordinates one can associate with the system an affine invariant.
    \item Let $(M,\omega,F=(J,H))$ be a hypersemitoric system exhibiting only standard flaps. Then making a choice of cut for each elliptic-elliptic value on the image of the flap and a suitable choice of action coordinates one can associate with the system an affine invariant.
    \end{enumerate}
    If the number of elliptic-elliptic values present on the image of a flap is bigger than one, then the two approaches yield a different affine invariant.
\end{theorem}
Theorem \ref{t.resultflaps} is proven in Section \ref{s.flap}. Furthermore, in Section \ref{s.flap} we compute representatives of these different choices of affine invariant, see Figure \ref{p.iafsflap} and Figure \ref{f.polytopeinvariantsflapwith2elliptic}.
\begin{itemize}
    \item Figure \ref{p.iafsflap} shows a representative of the affine invariant of a hypersemitoric system containing a standard flap with a single elliptic-elliptic value. Both approaches of the affine invariant yield the same result in this example.
    \item Figure \ref{f.polytopeinvariantsflapwith2elliptic} shows representatives of the two types of affine invariant for a system containing a standard flap with two different elliptic-elliptic values. The two approaches yield different results.
\end{itemize}
\noindent
Next we need to study hypersemitoric systems exhibiting only a pleat/swallowtail:
\begin{theorem}
    Let $(M,\omega,F=(J,H))$ be a hypersemitoric system exhibiting only a pleat/swallowtail. Then making suitable choices of action coordinates one can associate with $(M,\omega,F=(J,H))$ an affine invariant. In particular no choice of cut is necessary. 
\end{theorem}
This is proven in Section \ref{s.pleat}. In particular:
\begin{itemize}
    \item Figure \ref{f.panelswallowtail} shows the affine invariant for a hypersemitoric system containing a single pleat/swallowtail.
\end{itemize}

On the way towards defining an affine invariant for a simple hypersemitoric system we came across interesting, formally nonhypersemitoric examples. Here "formally" means that there exist degenerate nonparabolic singular points as the only obstruction for being hypersemitoric. If one wants to further extend the affine invariant to more classes of integrable systems, it is important to understand these examples. Let us call these examples a \textit{system exhibiting a line of curled tori}, see Definition \ref{d.lineofcurledtori}.
%\begin{proposition}
 %   With a system $(M,\omega,F=(J,H))$ exhibiting a line of curled tori, by making a suitable choice of action coordinates one can associate an affine invariant. Furthermore, no choice of cuts is needed.
%\end{proposition}
\begin{proposition}
    Let $(M,\omega,F=(J,H))$ be a system exhibiting a line of curled tori. Let $C_0$ denote the set of critical values of rank $0$ of the system. If $F(M)\backslash C_0$ is simply connected then by making a suitable choice of action coordinates one can associate an affine invariant with the system. Furthermore, no choice of cuts is necessary.
\end{proposition}
This is the content of Proposition \ref{p.generalizedcurledtori}.
Furthermore, we compute the affine invariant of a system exhibiting a line of curled tori for two examples, see Figure \ref{f.polytopecurledtori} and Figure \ref{f.gridmicroflap}.
\begin{itemize}
    \item Figure \ref{f.polytopecurledtori} shows the affine invariant for a system exhibiting a line of curled tori connecting to a single degenerate value. 
    \item Figure \ref{f.gridmicroflap} shows the affine invariant for a system exhibiting a line of curled tori connecting to the image of a generalized flap.
\end{itemize}

\subsection{Structure of the paper}
\begin{itemize}
    \item In Section \ref{s.premilinaries} we recall and state the necessary definitions and conventions. 
    \item In Section \ref{s.structureflap} we give an example of a system exhibiting a flap and describe the basic structure of the image of a flap.
    \item In Section \ref{s.introquantization} we recall the method of quantization. In particular we apply it to two examples, one of which is the Hirzebruch surface. Later in the paper we will use the quantization of certain systems to obtain the (representatives of the) affine invariants of said systems.
    \item In Section \ref{s.flap} we show that with a hypersemitoric system exhibiting only flaps one can associate an affine invariant. We explore two ideas that allow us to obtain an affine invariant: making a cut for each flap, see Subsection \ref{s.polytopenflaps}, and making a cut for each elliptic-elliptic value present in the image of a flap, see Subsection \ref{s.polytopeflap}. Moreover, in Subsection \ref{s.grouporbitnflaps} and Subsection \ref{s.orbitpolytope} we study the effect the choice of cut direction has on the affine invariant. Furthermore, in Subsection \ref{s.examplespolytopeflap} we compute the affine invariant of specific systems using both approaches. 
    \item In Section \ref{s.pleat} we show that with a hypersemitoric system exhibiting only a pleat/swallowtail one can associate an affine invariant. Furthermore, in Subsection \ref{s.pleatexample} we compute the affine invariant for a specific system. 
    \item In Section \ref{s.curledtori} we study systems exhibiting a line of curled tori and associate an affine invariant with them. Furthermore, we compute this invariant for two examples.
    \item In Section \ref{s.generalhypersemitoric} we associate with a simple hypersemitoric system an affine invariant. Furthermore, in Subsection \ref{s.examplesgeneralhypersemitoric} we compute representatives of this affine invariant in three examples. 
    \item In Appendix \ref{s.classicalactions} we compute the classical actions of $4$ systems, which is necessary to obtain (representatives of) their affine invariants. 
    \item In Appendix \ref{s.quantizationofthesystems} we show how to quantize $3$ systems defined on a certain Hirzebruch surface which is then used to obtain the affine invariant of specific examples. 
\end{itemize}
\subsection*{Acknowledgments}
The second author was partially and third author fully supported by the FWO-EoS project {\em Beyond symplectic geometry} with UA Antigoon number 45816. Moreover, the second author was also partially supported by the grant "Francqui Research Professor 2023-2026" of the Francqui Foundation with UA Antigoon number 49741.

\section{Preliminaries}
\label{s.premilinaries}
In this section we introduce the background, concepts, notations and results necessary for this paper, i.e., experts may proceed directly to Section \ref{s.structureflap}.

\subsection{Integrable systems}
\begin{quotation}
{\em Throughout this paper we assume all manifolds $M$ to be connected. }
\end{quotation}
Let $(M,\omega)$ be a symplectic manifold. Since the symplectic form $\omega$ is nondegenerate, we can associate to each $f\in C^{\infty}(M,\mathbb{R})$ a vector field $X_{f}$ using the relation $\omega(X_f,\cdot )=-df(\cdot)$. We say that $f$ is a Hamiltonian function and $X_f$ is its Hamiltonian vector field. The flow of $X_f$ is called the Hamiltonian flow generated by $f$. The Poisson bracket of the Hamiltonian functions $f,g\in C^{\infty}(M,\mathbb{R})$ induced by $\omega$ is defined by $\{f,g\}:=\omega(X_f,X_g)$.

\begin{definition}
Let $(M,\omega)$ be a $2n$-dimensional symplectic manifold. A $2n$-dimen{\-}sional (completely) integrable system is a triple $(M,\omega,F)$ where $F=(f_1,...,f_n):M\rightarrow \mathbb{R}^n$, called the momentum map, satisfies the following conditions: 
\begin{itemize}
\item $\{f_i,f_j\}=0$ for all $1\leq i,j \leq n$;
\item The Hamiltonian vector fields $X_{f_1},...,X_{f_n}$ are linearly independent almost everywhere in $M$.
\end{itemize}
\end{definition}

\begin{definition}
Let $(M,\omega,F)$ and $(M',\omega',F')$ be $2n$-dimensional integrable systems. We say that $(M,\omega,F)$ and $(M',\omega',F')$ are isomorphic if there exists a pair $(\phi,\rho)$ where $\phi:(M,\omega)\rightarrow (M',\omega')$ is a symplectomorphism and $\rho:F(M)\rightarrow F'(M')$ is a diffeomorphism, such that $\rho\circ F=F'\circ \phi$.
\end{definition}

The rank of $dF$ at a point $x\in M$ is defined by the dimension of the span of $X_{f_1},...,X_{f_n}$ at $x$ which is equal to the rank of $dF$ at $x$. A point $x\in M$ is called a \textbf{regular} point of $F$ if $dF(x)$ has maximal rank, and otherwise it is called \textbf{singular}. A value $c\in F(M)$ is called a \textbf{regular value} of $F$ if $F^{-1}(c)$ only contains regular points, in which case the fiber $F^{-1}(c)$ is also called \textbf{regular}. A value $c\in F(M)$ is called singular if there exists at least one singular point in $F^{-1}(c)$.

Note that the symplectic form vanishes on the fibers of $F$, and in particular, the regular fibers of $F$ are Lagrangian submanifolds of $M$. Therefore, the map $F$ induces a \textbf{singular Lagrangian fibration} on $M$. 

\subsection{Regular points} Let us recall an example of an integrable system on the cotangent bundle $T^*\mathbb{T}^n$ of the $n$-torus $\mathbb{T}^n$. Let $(q_1,\dots,q_n,p_1,\dots,p_n)$ be the standard local coordinates in $T^*\mathbb{T}^n \simeq \mathbb{T}^n \times \mathbb{R}^n$, then the symplectic form is locally given by $\omega_0:=-\sum_{j=1}^{n}dq_j\wedge dp_j$. Consider the momentum map given by $F:=(p_1,\dots,p_n)$. This example serves as a model for neighborhoods of regular fibers:

\begin{theorem}
\label{t.AL}
({Liouville-Arnold-Mineur Theorem, Arnold \cite{arnol2013mathematical}}) Let $(M,\omega,F)$ be an integrable system and let $c\in F(M)$ be a regular value. If $\Lambda_c:=F^{-1}(c)$ is a regular, compact and connected fiber, then there exist neighborhoods $U\subset F(M)$ of $c$ and $V\subset \mathbb{R}^n$ of the origin, such that for
\begin{equation*}
\mathcal{U}:= \bigsqcup_{r\in U}F^{-1}(r)\quad  \text{and} \quad \mathcal{V}:=\mathbb{T}^n\times V \subset T^*\mathbb{T}^n
\end{equation*}
we have that $(\mathcal{U},\omega|_{\mathcal{U}},F|_{\mathcal{U}})$ and $(\mathcal{V},\omega_0|_{\mathcal{V}},F|_{V})$ are isomorphic integrable systems. 
\end{theorem}

In particular this means that $\Lambda_c\simeq \mathbb{T}^n$ and that $F|_{\mathcal{U}}$ is a trivial torus bundle. The local coordinates $p_j$ on $T^*{\mathbb{T}^n}$ are called \textbf{action coordinates} and the $q_j$ are called \textbf{angle coordinates}. Moreover, if $\{\gamma_1(c),\dots,\gamma_n(c)\}$ is a basis of the homology group $H_1(\Lambda_c)$, varying smoothly with $c\in U$, then the action coordinates are given by the formula
\begin{equation*}
p_j(c)=\frac{1}{2\pi}\oint_{\gamma_j(c)}\overline{\omega},
\end{equation*}
where $\overline{\omega}$ is any $1$-form such that $d\overline{\omega}=\omega$ on $\mathcal{U}$.

\subsection{Integral affine structures}
\label{s.IAFS}
The group of integral affine maps of $\mathbb{R}^n$ is the semidirect product $\AGL(n,\mathbb{Z}) = GL(n,\mathbb{Z}) \ltimes \mathbb{R}^n$. An element $T = (A, b) \in \AGL(n,\mathbb{Z})$, where $A \in GL(n,\mathbb{Z})$ and $b\in \mathbb{R}^n$, acts on $\mathbb{R}^n$ as  $T(x) = A x + b$, $x \in \mathbb{R}^n$. An \textbf{integral affine structure} on an $n$-manifold $X$ is an atlas of charts on $X$ such that the transition functions between these charts are integral affine maps on $\mathbb{R}^n$. If $X$ and $Y$ are manifolds equipped with integral affine structures, we call a map $g:X\rightarrow Y$ an integral affine map if it sends the integral affine structure of $X$ to the integral affine structure of $Y$. Equivalently, one can define an \textbf{integral affine structure} $\mathcal{A}$ on an $n$-manifold $X$ as a lattice in its tangent bundle, see for example Symington \cite[Proposition $2.10$]{symington2002four}.

Suppose that $(M,\omega,F)$ is an integrable system such that all fibers are connected, and let $B=F(M)$ denote the momentum map image. Let $B_r\subset B$ denote the set of regular values of $F$. Given any $c\in B_r$, applying Theorem \ref{t.AL}, we obtain coordinates $p_1,\dots,p_n$ in a neighborhood of $c$. Since they arise from a choice of primitive of $\omega$ and a choice of basis of $H_1(\Lambda_c)$, the action coordinates are unique up to the action of $\AGL(n,\mathbb{Z})$. Therefore, the action coordinates induce an integral affine structure on $B_r$. In fact, this construction also works if the fibers of $F$ are disconnected, but then one needs to work on each connected component.

\begin{example}
The standard lattice $\Lambda_0$ generated by the unit vectors tangent to the coordinate axes in $\mathbb{R}^n$ defines the \textbf{standard integral affine structure} $\mathcal{A}_0$ in $\mathbb{R}^n$.
\end{example}

\subsection{Topological monodromy}
\label{s.tm}In the notation of Section \ref{s.IAFS}, a natural question to ask is whether the integral affine structure on the set $B_r$ can be chosen to be trivial, i.e., if $M_{r}\simeq B_r \times \mathbb{T}^n$, with $M_r:=F^{-1}(B_r)$. Note that Theorem \ref{t.AL} guarantees the existence of \textbf{local} action-angle coordinates on a semi-global neighborhood of each regular fiber. So the actual question is whether these coordinates can be extended globally. One of the possible obstructions to this extension is the so called \textit{topological monodromy}, see for example Duistermaat \cite{duistermaat1980global}, which we recall in the following. 

Let $(M,\omega,F)$ be an integrable system with compact and connected fibers. By Theorem \ref{t.AL}, all regular fibers are diffeomorphic to an $n$-torus. In particular, their first homology groups are isomorphic. Let $\gamma \subset B_r$ be a loop. For each $s\in \gamma$, let $\Lambda_s:=F^{-1}(s)\subset M$ be the corresponding regular fiber. Fix a value $s_0\in \gamma$ and consider the first homology group of the corresponding fiber $H_1(\Lambda_{s_0})$. Due to Theorem \ref{t.AL}, for any $s\in \gamma$ close to $s_0$, there exists an isomorphism $H_1(\Lambda_s)\simeq H_1(\Lambda_{s_0})$, obtained by considering a smoothly varying basis $\{\alpha_1(\gamma(s)),\cdots ,\alpha_n(\gamma(s))\}$ of $H_1(\Lambda_{\gamma(s)})\simeq \mathbb{Z}^n$. Iterating this argument, following $\gamma$ until we are back at $s_0$, we obtain an automorphism
\begin{equation*}
\mu_{\gamma,s_0}:H_1(\Lambda_{s_0})\rightarrow H_1(\Lambda_{s_0}),
\end{equation*}
called the \textbf{monodromy} transformation. This map does not depend on the representative $\gamma\in [\gamma]$ of the homotopy class nor on the point $s_0\in \gamma$, so we obtain in fact a transformation $\mu_{[\gamma]}\in GL(n,\mathbb{Z})$. The association of this matrix to each homotopy class of the fundamental group defines the \textbf{monodromy map}
\begin{equation*}
\mu:\pi_1(B_r)\rightarrow GL(n,\mathbb{Z}).
\end{equation*}
\subsection{Affine monodromy}Similar to the topological monodromy associated to a torus fibration (see section \ref{s.tm}), an integral affine structure $\mathcal{A}$ on a manifold $B$ has an \textbf{affine monodromy} map
\begin{equation*}
\Psi:\pi_1(B,b)\rightarrow \text{Aut}(\Lambda_b),
\end{equation*}
where $\Lambda_b$ is the restriction to $T_bB$ of the lattice $\Lambda$ that defines $\mathcal{A}$. Specifically, if we identify $(T_bB,\Lambda_b)$ with $(\mathbb{R}^n,\mathbb{Z}^n)$ and let $\gamma$ be a loop with base point $b$, $\Psi([\gamma])$
is the automorphism of $(\mathbb{R}^n,\mathbb{Z}^n)$ such that $\gamma^{*}(TB,\Lambda)$ is isomorphic to $[0,1]\times (\mathbb{R}^n,\mathbb{Z}^n)/((0,p)\sim (1,\Psi([\gamma])(p)))$.

Let $(M,\omega,F)$ be a (completely) integrable system and $b\in B_r$. Recall the following lemma:
\begin{lemma}
(Symington \cite{symington2002four})
\label{l.af}
Consider, with respect to local standard coordinates near the regular fiber $F^{-1}(b)$, the topological monodromy of $\gamma\in \pi_1(B,b)$ given by $A\in GL(n,\mathbb{Z})$. Then the affine monodromy is given by the inverse transpose $A^{-T}$.
\end{lemma} 

\subsection{Nondegenerate singularities}
\label{sec:locNormalForm}
A singular point $p\in M$ of rank zero of an integrable system $(M,\omega,F=(f_1,\dots,f_n))$ is \textbf{nondegenerate} if the corresponding Hessians $d^2f_1(p), \dots, d^2f_n(p)$ span a Cartan subalgebra of the Lie algebra of quadratic forms on $T_pM$. 
% In Section \ref{s.nondegeneratefour}, this definition is reformulated for $4$-dimensional symplectic manifolds in terms of linear algebra. 
We refer to Bolsinov \& Fomenko \cite{bolsinov2004integrable} for the general definition of nondegenerate points of higher rank, and to Section ~\ref{s.nondegeneratefour} for the special case $\dim M = 4$.

Nondegenerate singularities are in fact linearizable. More precisely, there exists a local normal form, developed in works by R\"ussmann \cite{russmann1964verhalten}, Vey \cite{vey1978certains}, Colin de Verdi\'ere \& Vey \cite{de1979lemme}, Eliasson \cite{eliasson1990normal, eliassonphdthesis}, Dufour and Molino \cite{dufour1991compactification}, Miranda \cite{mirandaphdthesis,miranda2014integrable}, Miranda \& Zung \cite{miranda2004equivariant}, Miranda \& \vungoc \ \cite{miranda2005singular}, \vungoc\ \& Wacheux \cite{vu2013smooth}, and Chaperon \cite{chaperon2013normalisation}.
\begin{theorem}
\label{t.localnormalform}
If $p \in M$ is a nondegenerate singular point of an integrable system $(M,\omega,F=(f_1,\dots,f_n))$, there exist local canonically symplectic coordinates $(x_1,\dots,x_n,\xi_1,\dots,\xi_n)$ near $p$, in which $p$ corresponds to the origin and the symplectic form becomes $\omega=\sum_{i=1}^n dx_i\wedge d\xi_i$, and functions $q_1,\dots,q_n$ of $(x_1,\dots,x_n,\xi_1,\dots,\xi_n)$ such that $\{f_i,q_j\}=0$, for all $1 \leq i,j \leq n$, and each $q_i$ has one of the following types:
\begin{enumerate}[label={(\roman*)}]
    \item \textbf{elliptic:} $q_i=\frac{1}{2} (x_i^2+\xi_i^2)$;
    \item \textbf{hyperbolic:} $q_i=x_i\xi_i$;
    \item \textbf{regular:} $q_i=\xi_i$;
    \item \textbf{focus-focus:} $q_i=x_i\xi_{i+1}-x_{i+1}\xi_i$ and $q_{i+1}=x_i\xi_1+x_{i+1}\xi_{i+1}$.
\end{enumerate}
If there are no components of hyperbolic type, then
\begin{equation*}
    F-F(p) = g\circ (q_1,\dots,q_n) \circ (x_1,\dots,x_n,\xi_1,\dots,\xi_n)
\end{equation*}
where $g$ is a diffeomorphism from a small neighborhood of $(0,\dots,0)\in \mathbb{R}^n$ onto another such neighborhood such that $g(0,\dots,0)=(0,\dots,0)$.
\end{theorem}

The number of regular type components $q_i$ appearing in Theorem~\ref{t.localnormalform} equals the rank of the nondegenerate singularity $p$. Additionally, the relations $\{f_i,q_j\}=0$, $1 \leq i,j \leq n$, imply that each $q_i$ factors through $F$.

% \subsection{Nondegenerate singularities in dimension four}

%%%%%%%%%%%%%%%%%%%%%%%%%%%%%%%%%%%%%%%%%%%%%%%%%%%%
%%%%%%%%%% new subsection  %%%%%%%%%%%%%%%%%%%%%%%%%%%

\subsection{Nondegeneracy in four dimensions}
\label{s.nondegeneratefour}

Theorem \ref{t.localnormalform} implies that in dimension four, the types of occurring nondegenerate singularities $p \in M$ are limited to the following ones, determined by the possible combinations of $q_1$ and $q_2$:
\begin{enumerate}[label={(\roman*)}]
\item $p$ is \textbf{elliptic-elliptic} (rank $0$): $q_1=\frac{1}{2}(x_1^2+\xi_1^2)$ and $q_2=\frac{1}{2}(x_2^2+\xi_2^2)$;
\item $p$ is \textbf{elliptic-hyperbolic} (rank $0$): $q_1=\frac{1}{2}(x_1^2+\xi_1^2)$ and $q_2=x_2\xi_2$;
\item $p$ is \textbf{hyperbolic-hyperbolic} (rank $0$): $q_1=x_1\xi_1$ and $q_2=x_2\xi_2$;
\item $p$ is \textbf{focus-focus} (rank $0$): $q_1=x_1\xi_2-x_2\xi_1$ and $q_2=x_1\xi_1+x_2\xi_2$;
\item $p$ is \textbf{elliptic-regular} (rank $1$): $q_1=\frac{1}{2}(x_1^2+\xi_1^2)$ and $q_2=\xi_2$;
\item $p$ is \textbf{hyperbolic-regular} (rank $1$): $q_1=x_1\xi_1$ and $q_2=\xi_2$.
\end{enumerate}
Note that elliptic-regular singularities are also often referred to as {\bf transversally-elliptic} and hyperbolic-regular ones as {\bf transversally-hyperbolic}.

Instead of working with Cartan subalgebras, nondegenerancy of rank $0$ points, i.e., fixed points, can be verified using the following linear algebra technique.
\begin{lemma}
    ({Bolsinov \& Fomenko \cite{bolsinov2004integrable}}) Let $(M,\omega,F=(f_1,f_2))$ be a $4$-dimen\-sional completely integrable system having a fixed point $p\in M$. Let $\omega_p$ be the matrix of the symplectic form with respect to a basis of $T_pM$ and let $d^2f_1(p)$ and $d^2f_2(p)$ be the matrices of the Hessians of $f_1$ and $f_2$ with respect to the same basis. Then, the fixed point $p$ is nondegenerate if and only if $d^2f_1(p)$ and $d^2f_2(p)$ are linearly independent and if there exists a linear combination of $\omega_p^{-1}d^2f_1(p)$ and $\omega_p^{-1}d^2f_2(p)$ which has four distinct eigenvalues.
\end{lemma}

\begin{remark}
The matrices $\omega_p^{-1}d^2f_1(p), \omega_p^{-1}d^2f_2(p) \in \mathrm{sp}(4,\mathbb{R})$ represent the linearization of the Hamiltonian vector fields $X_{f_1}$, $X_{f_2}$ at $p$.
\end{remark}

%Denote by $\lambda_1,\lambda_2,\lambda_3,\lambda_4$ the distinct eigenvalues of a nondegenerate fixed point $p$. 
Then one can classify the type of of a nondegenerate rank zero singular point $p$ in terms of the eigenvalues as follows, with $\alpha,\beta\in \mathbb{R}^{\neq 0}$:
\begin{enumerate}[label={(\roman*)}]
    \item \textbf{elliptic-elliptic:} eigenvalues $\pm i \alpha$, $\pm i \beta$, $\alpha \ne \beta$
    \item \textbf{elliptic-hyperbolic:} eigenvalues $\pm i\alpha$, $\pm \beta$;
    \item \textbf{hyperbolic-hyperbolic:} eigenvalues $\pm \alpha$, $\pm \beta$, $\alpha \ne \beta$;
    \item \textbf{focus-focus:} eigenvalues $\pm \alpha \pm i \beta$.
\end{enumerate}
% where $\alpha,\beta\in \mathbb{R}^{\neq 0}$ and $\alpha\neq \beta$ for the elliptic-elliptic and hyperbolic-hyperbolic cases.

Nondegenerancy of rank one points can be characterised on $4$-dimensional symplectic manifolds $(M,\omega)$ as follows (see Bolsinov \& Fomenko \cite[ Section $1.8.2$]{bolsinov2004integrable} for details). Let $p$ be a singular point of rank one of a $4$-dimensional integrable system $(M,\omega,F=(f_1,f_2))$. Then, there are $\mu,\lambda\in \mathbb{R}$ such that $\mu\, df_1(p)+\lambda\, df_2(p)=0$ and $L_p:=\text{Span}\{X_{f_1}(p),X_{f_2}(p)\}\subset T_pM$ is the tangent line of the orbit generated by the $\mathbb{R}^2$-action at $p$. Denote by $L_p^{\perp}$ the symplectic orthogonal complement of the space $L_p$ in $T_pM$. Note that we have $L_p\subset L_p^{\perp}$. The Poisson commutativity $\{f_1,f_2\}=0$ implies that $L_p$ and $L_p^{\perp}$ are invariant under the $\mathbb{R}^2$-action. Therefore, $\mu\, d^2f_1(p)+\lambda\, d^2f_2(p)$ descends to the quotient $L_p^{\perp}/L_p$. This allows us to define:

\begin{definition}
    ({Bolsinov \& Fomenko, \cite{bolsinov2004integrable}}) A rank one critical point $p$ of a $4$-dimensional completely integrable system $(M,\omega,F=(f_1,f_2))$ is \textbf{nondegenerate} if $\mu\,d^2f_1(p)+\lambda\, d^2f_2(p)$ is invertible on the quotient $L_p^{\perp}/L_p$.
\end{definition}
The possible types of nondegenerate rank one points on $4$-dimensional symplectic manifolds are classified in terms of the eigenvalues of $\omega_p^{-1}(\mu d^2f_1(p)+\lambda d^2f_2(p))$ on the quotient $L_p^{\perp}/L_p$ as follows, where $\alpha\in \mathbb{R}^{\neq 0}$:
\begin{enumerate}[label={(\roman*)},start=5]
\item \textbf{elliptic-regular:} eigenvalues $\pm i\alpha$;
\item \textbf{hyperbolic-regular:} eigenvalues $\pm \alpha$.
\end{enumerate}

\subsection{Toric systems}
\label{s.toric}

Very `easy' integrable systems are the following ones. 

\begin{definition}
An integrable system $F=(f_1,\dots,f_n):M\rightarrow \mathbb{R}^n$ on a symplectic $2n$-dimensional manifold $(M, \omega)$ is \textbf{toric} if the Hamiltonian vector fields $X_{f_1},\dots,X_{f_n}$ generate periodic flows of the same period (in our convention $2\pi$) and the action of $\mathbb{T}^n$ on $M$ induced by these flows is effective, i.e., only the identity acts trivially.
\end{definition}
Having $n$ periodic flows implies that the singularities of toric integrable systems cannot have focus-focus or hyperbolic components. Speaking in terms of the local normal form, see Theorem \ref{t.localnormalform}, if $m=(0,\dots,0)$ and $\omega=\sum_{i=1}^n dx_i\wedge dy_i$, then the integrable system is locally in a neighborhood of $m$ of the form
\begin{equation*}
F(x_1,\dots,x_n,\xi_1,\dots,\xi_n) = \left( \frac{x_1^2+\xi_1^2}{2}, \dots, \frac{x_k^2+\xi_k^2}{2}, 
\xi_{k+1}, \dots, \xi_n \right)
\end{equation*}
for some $0 \leq k \leq n$.
Toric integrable systems have connected fibers, see Atiyah \cite{atiyah1982convexity}.
%, a fact known as Atiyah's connectivity.
In particular, all fibers of a toric system $F$ are diffeomorphic to tori of varying dimensions $\mathbb{T}^k$ with $ k\in \{0,...,n\}$. This distinguishes toric systems among other integrable systems.

Let $(M,\omega,F)$ be a toric system. A fundamental theorem of equivariant symplectic geometry, due to Atiyah \cite{atiyah1982convexity} and Guillemin \& Sternberg \cite{guillemin1982convexity} says that, if $M$ is compact and connected, the image $F(M)$ is a convex polytope in $\mathbb{R}^n$, obtained as the convex hull of the images of the fixed points of the Hamiltonian $\mathbb{T}^n$ action.
% of the $n$-torus on $M$ induced by concatenating the flows of $f_i$. 
% If two toric integrable systems are isomorphic then their associated convex polytopes coincide up to an element in $\AGL(n,\mathbb{Z})$. Furthermore,
Delzant showed that the convex polytopes obtained as images of toric integrable systems are of a special type \cite{delzant1988hamiltoniens}. They are referred to as Delzant polytopes and they are distinguished by being \textbf{simple}, \textbf{rational}, and \textbf{smooth} which means, respectively, that there are precisely $n$ edges meeting at each vertex, that the slopes of the edges are rational, and that the normal vectors to the facets meeting at each vertex form a basis of the integral lattice. Furthermore, Delzant showed that there exists a one-to-one correspondence between Delzant polytopes, up to an element of $\AGL(n,\mathbb{Z})$, and isomorphic toric integrable systems.

\subsection{Semitoric systems}
\label{s.semitoric}

Semitoric integrable systems form a class of integrable systems which generalizes the class of toric integrable systems.

\begin{definition}
    \label{d.semitoric}
    Given $(M,\omega)$ an integrable system 
    \begin{equation*}
        F=(f_1,\dots,f_n):M\rightarrow \mathbb{R}^n
    \end{equation*}
    is \textbf{semitoric} if the Hamiltonian vectors fields $X_{f_1},\dots,X_{f_{n-1}}$ generate periodic flows of the same period (in our convention $2\pi$), if the action of $\mathbb{T}^{n-1}$ induced on $M$ by these flows is effective, and if the singularities of $F$ are nondegenerate and do not have hyperbolic components. If $M$ is not compact the integrals $f_1,\dots,f_{n-1}$ are required to be proper.
\end{definition}
The singular points of these systems may have regular, elliptic and/or focus-focus components. 

In dimension four, semitoric systems were classified, first under some conditions, by Pelayo \& \vungoc\ \cite{pelayo2009semitoric, pelayo2011constructing} and then in full generality by Palmer, Pelayo and Tang \cite{semitoricnonsimple}. \vungoc\ \cite{vu2007moment} showed that the fibers of these systems are connected. The singular fibers of these systems are either points, circles or tori with a finite number of pinches. This last type of fiber does not appear in toric integrable systems.

Two semitoric systems $(M,\omega,(f_1, f_2))$ and $(\tilde{M},\tilde{\omega},(\tilde{f}_1, \tilde{f}_2))$ are \textbf{isomorphic} (as semitoric systems) if there exists a symplectomorphism $\psi: M\rightarrow \tilde{M}$ and a smooth map $g: \mathbb{R}^2\rightarrow \mathbb{R}$ such that $(\tilde{f}_1, \tilde{f}_2) \circ \psi =(f_1,g(f_1, f_2))$ with $\partial_2 g >0$. 

A semitoric system is \textbf{simple} if each fiber contains at most one focus-focus singularity. Such systems are determined up to isomorphism, according to Pelayo \& \vungoc\ \cite{pelayo2009semitoric}, by the following five symplectic invariants:
\begin{enumerate}
    \item The \textbf{number of focus-focus singularities}, denoted by $n_{FF}$.
    \item The \textbf{Taylor series invariant}, consisting of $n_{FF}$ formal Taylor series in two variables describing the foliation around each focus-focus singular fiber.
    \item The \textbf{polytope invariant}, a family of weighted rational convex polytopes generalizing the Delzant polytope and which may be viewed as a bifurcation diagram.
    \item The \textbf{height invariant}, given by $n_{FF}$ numbers corresponding to the height of the focus-focus critical values in the representatives of the polytope invariant.
    \item The \textbf{twisting index invariant}, given by $n_{FF}$ integers measuring how twisted the system is around singularities from a `toric point of view'.
\end{enumerate}

The classification result for semitoric systems states that two semitoric systems are isomorphic if and only if they have the same list of symplectic invariants. Furthermore, given any admissible list of invariants, a semitoric system with these invariants can be constructed.

Since the aim of this paper is to generalize the semitoric polytope invariant to other types of integrable systems let us consider it now in more detail. Let $(M,\omega,F=(f_1,f_2))$ be a semitoric system in a $4$-dimensional symplectic manifold $(M,\omega)$. Furthermore, let
\[ \{c_i=(x_i,y_i) \mid i=1,\dots,m_f\}\in \mathbb{R}^2 \]
be the set of focus-focus critical values, ordered in such a way that $x_1\leq x_2\leq \dots \leq x_{m_f-1}\leq x_{m_f}$ and consider the set $B_r$ of regular values in $B=F(M)$. For $i \in \{1, \dots, m_f\}$ and $\epsilon \in \{-1,+1\}$, define $\mathcal{L}_i^{\epsilon}$ to be the vertical ray starting at $c_i$ and going to $ \pm {\infty}$ depending on the sign of $\epsilon$, i.e., $\mathcal{L}_i^{\epsilon}=\{(x_i,y) \mid \epsilon y\geq \epsilon y_i\}$.
Given $\vec{\epsilon}=(\epsilon_1,...,\epsilon_{m_f})\in \{-1,+1\}^{m_f}$, define the line segment $l_i:=F(M)\cap \mathcal{L}_i^{\epsilon_i}$, set
\begin{equation*}
l^{\vec{\epsilon}}:=\cup_{i}l_i,
\end{equation*}
and decorate each $l_i$ with the multiplicity $\epsilon_i k_i$, where $k_i$ is the number of focus-focus points in the fiber $F^{-1}(c_i)$. If several $c_i$'s have the same $x_i$ coordinate, $l_i$ is the union of all corresponding segments. Given $c\in l_i$, define $k(c):=\sum_{c_j}\epsilon_jk_j$, where the sum runs over all focus-focus values $c_j \in l_i$. 

% Recall $\AGL(2,\mathbb{Z})$, the group of integral affine transformations, i.e., maps $T:\mathbb{R}^2\rightarrow \mathbb{R}^2$ of the form $T(x)=A x+ b$, with $A\in GL(2,\mathbb{Z})$ and $b\in \mathbb{R}^2$. 
Denote by $\mathcal{T}$ the subgroup of the group of integral affine maps $\AGL(2,\mathbb{Z})$, introduced in Section~\ref{s.IAFS}, which leaves a vertical line, including its orientation, invariant. In other words, an element of $\mathcal{T}$ is a composition of a vertical translation and an element of the Abelian subgroup $\{T^k| \ k\in \mathbb{Z}\}$ of $GL(2,\mathbb{Z})$, where
\begin{equation*}
T^k:= \begin{bmatrix}
1 & 0 \\
k & 1 
\end{bmatrix}.
\end{equation*}

\begin{theorem}({\vungoc\ \cite[Theorem 3.8]{vu2007moment}})
\label{th.straigheningHomeo}
Using the notation from above for a semitoric system $(M,\omega,(f_1,f_2))$, then for all $\vec{\epsilon}\in \{-1,+1\}^{m_f}$, there exists a homeomorphism $f_{\vec{\epsilon}}=(f_{\vec{\epsilon}}^{(1)}, f_{\vec{\epsilon}}^{(2)})$ from $B$ to $f_{\vec{\epsilon}}(B) \subseteq \R^2$ such that:
\begin{itemize}
    \item $f_{\vec{\epsilon}}|_{(B \setminus l^{\vec{\epsilon}})}$ is a diffeomorphism onto its image;
    
    \item $f_{\vec{\epsilon}}|_{(B_r \setminus l^{\vec{\epsilon}})}$ is affine, sending the integral affine structure of $B_r$ to the standard integral affine structure of $\mathbb{R}^2$;
    
    \item $f_{\vec{\epsilon}}$ preserves $f_1$, i.e., $f_{\vec{\epsilon}}(x,y)=(x,f_{\vec{\epsilon}}^{(2)}(x,y))$;
    
    \item for each $i \in \{1,\dots,m_f\}$ and each $c\in \operatorname{int}(l_i)$, there is an open ball $D$ around $c$ such that the restrictions of
    $f_{\vec{\epsilon}}|_{(B \setminus l^{\vec{\epsilon}})}$ to the domains $\{(x,y)\in D| \ x\leq x_i\}$ and $\{(x,y)\in D| \ x\geq x_i\}$ are smooth maps. Furthermore,
    \begin{equation*}
        \lim_{\stackrel{(x,y)\rightarrow c}{x<x_i}} df_{\vec{\epsilon}}(x,y)=T^{k(c)}\lim_{\stackrel{(x,y)\rightarrow c}{ x>x_i}} df_{\vec{\epsilon}}(x,y),
    \end{equation*}
\end{itemize}
where
\begin{equation*}
    T^{k(c)}= \begin{bmatrix}
1 & 0\\
k(c) & 1
\end{bmatrix}.
\end{equation*}
The map $f_{\vec{\epsilon}}$ is unique modulo a left composition by a transformation in $\mathcal{T}$, and its image is a representative of the polytope invariant.
\end{theorem}

The map $f_{\vec{\epsilon}}$ is sometimes referred to as {\bf straightening homeomorphism}.
Intuitively, one obtains $f_{\vec{\epsilon}}(B)$ by first cutting the set $B$ along each of the vertical lines $\mathcal{L}_i^{\epsilon_i}$ to the focus-focus values. Then the resulting set is simply connected, and thus there exists a global $2$-torus action on the preimage of this set. A representative of the polytope invariant can thus be seen as the closure of the image of a toric momentum map.

Describing all possible choices (like for instance the signs $\epsilon_j$) by means of a group action allows to write the polytope invariant as an equivalence class or orbit of this group action. 

We now describe the dependence of the map $f_{\vec{\epsilon}}$ on the choice of $\epsilon$. Let $t_{c_i}:\mathbb{R}^2\rightarrow \mathbb{R}^2$ be the map given by
\begin{equation*}
t_{c_i}(x,y)=
\begin{cases}
	(x,y), & x\leq x_i, \\
	(x,y+x-x_i), & x>x_i,
\end{cases}
\end{equation*}
that is, intuitively the map $t_{c_i}$ leaves the half-plane to the left of $c_i$ invariant and applies $T$, relative to a choice of origin on $c_i$, to the half-plane on the right of $c_i$. For a vector $k=(k_1,...,k_{m_f})\in \mathbb{Z}^{m_f}$, define $t_{k}:=t^{k_1}_{c_1}\circ ... \circ t_{c_{m_f}}^{k_{m_f}}$, which is a piecewise integral affine map. 

We now define the \textbf{polytope invariant}. For a choice of $\vec{\epsilon}$ consider $\Delta:=f_{\vec{\epsilon}}(B)$. The triple $(\Delta,(c_{i})_{i=1}^{m_f},(\epsilon_i)_{i=1}^{m_f})$ is called the \textbf{weighted polytope} associated to $\vec{\epsilon}$. The freedom in the definition of $f_{\vec{\epsilon}}$ can be expressed as an action of the group $(\mathbb{Z}_2)^{m_f}\times \mathcal{T}$ on the space of weighted polytopes: letting $(\vec{\epsilon'},T^n)\in (\mathbb{Z}_2)^{m_f}\times \mathcal{T}$ and taking $u:=\frac{1}{2}(\vec{\epsilon}-\vec{\epsilon'}\vec{\epsilon})$ the action is given by
\begin{equation*}
(\vec{\epsilon'},T^n)\cdot (\Delta,(c_{i})_{i=1}^{m_f},(\epsilon_i)_{i=1}^{m_f})= (t_{u}(T^n(\Delta)),(c_i)_{i=1}^{m_f},(\epsilon_i'\epsilon_i)_{i=1}^{m_f}).
\end{equation*}

\subsection{Local normal form and isotropy weights of an \texorpdfstring{$S^1$}{S1}-action}
\begin{definition}
\label{d.s1space}
A Hamiltonian $S^1$-space is a triple  $(M,\omega,J)$ where $(M,\omega)$ is a compact four dimensional symplectic manifold and $J:M\rightarrow \mathbb{R}$ a Hamiltonian such that  the flow of its Hamiltonian vector field $X_J$ is periodic of minimal period $2\pi$. The Hamiltonian flow of such a $J$ generates an effective Hamiltonian action of $S^1$ on $M$. 
\end{definition}

Karshon \cite{karshon1999periodic} classified Hamiltonian $S^1$-spaces in terms of a labeled graph encoding information about the fixed points and isotropy groups of the $S^1$-action. 

For each subgroup $G\subset S^1$, denote by $M^{G}$ the set of points in $M$ whose stabilizer is $G$. Note that the connected components of $M^{S^1}$ are symplectic manifolds, thus, since the action is effective, either points or surfaces.

\begin{lemma}
\label{l.weights}
({Chaperon, \cite{chaperon1983quelques}})
For every $p\in M^{S^1}$ there exist neighborhoods $U\subset M$ of $p$ and $U_0\subset \mathbb{C}^2$ of $(0,0)$ and a symplectomorphism $\Psi:(U,\omega)\rightarrow (U_0,\omega_0)$ with $\omega_0=\frac{i}{2}(dz_1\wedge d\overline{z}_1+dz_2\wedge d\overline{z}_2)$ and $J_0(z_1,z_2)=J(p)+\frac{m}{2}|z_1|^2+\frac{n}{2}|z_2|^2$ 
with $n,m\in \mathbb{Z}$, such that the following diagram commutes:
\begin{center}
\begin{tikzcd}
{(U,\omega)} \arrow[rd, "J"] \arrow[rr, "\Psi"] &            & {(U_0,\omega_0)} \arrow[ld, "J_0"'] \\
                                                & \mathbb{R} &                                    
\end{tikzcd}
\end{center}
\end{lemma}

We refer to the integers $n,m$ in Lemma \ref{l.weights} as the \textbf{isotropy weights} of $J$ at $p$. If the action is effective the integers $n,m$ are coprime. 

\subsection{Duistermaat-Heckman measure for an \texorpdfstring{$S^1$}{S1}-action}
\label{s.duistermaatmeasuregeneral}
Let $(M,\omega)$ be a symplectic manifold of dimension $2n$. The Liouville measure on $(M,\omega)$ is induced by the volume form $\frac{1}{(2\pi)^n}\frac{\omega^n}{n!}$. Now suppose that $(M,\omega,J)$ is a Hamiltonian $S^1$- space. The Duistermaat-Heckman measure $\mu_{J}$ of the $S^1$-action is given by $\mu_J([a,b]):=\text{vol}(J^{-1}([a,b]))$, where $\text{vol}$ is the volume with respect to the Liouville measure of $(M,\omega).$
\begin{itemize}
    \item Denote by $B_{min}$ and $B_{max}$ the extremal sets of $J$. If $B_{min}$ is a two dimensional surface, denote its self intersection number by $e_{min}$. If $B_{min}$ is an isolated fixed point with isotropy weights $n$ and $m$, define $e_{min}:=\frac{1}{nm}$. Similarly define $e_{max}$ for $B_{max}$. 
    \item Denote the isotropy weights at an interior fixed point $p$ by $m_p$ and $n_p$. 
    \item Let $a_{min}:=\frac{1}{2\pi}\int_{B_{min}}\omega$ and $a_{max}:=\frac{1}{2\pi}\int_{B_{max}}\omega$.
    \item Denote the values of $J$ at the fixed points by $y_{min}=J(B_{min})$, $y_{max}=J(B_{max})$, and $y_{p}=J(p)$ for interior fixed points $p$. 
    \item Let 
    \begin{equation*}
      H(x):=\begin{cases}
            1,\ x\geq 0,\\
            0,\ x<0 
        \end{cases}
    ,\quad
        \theta(x):=\begin{cases}
            x,\ x\geq 0,\\
            0,\ x<0 .
        \end{cases}
    \end{equation*} 
\end{itemize}
\begin{lemma}
    \cite[Lemma $2.12$]{karshon1999periodic}
    The density function for the Duistermaat-Heckman measure is 
    \begin{align*}
         \rho_J(y)  =& \ a_{min}H(y-y_{min})-e_{min}\theta(y-y_{min}) \\
        & +\sum_{p}\frac{1}{m_pn_p}\theta(y-y_p)\\
        & -e_{max}\theta(y-y_{max})-a_{max}H(y-y_{max}).
        \end{align*}
\end{lemma}

\subsection{Monodromy in the presence of an \texorpdfstring{$S^1$}{S1}-action}
\label{s.monodromy}
Let $(M,\omega)$ be a $4$-dimen{\-}sional connected symplectic manifold, such that $F=(J,H):M\rightarrow \mathbb{R}^2$ is a completely integrable system and assume that:
\begin{itemize}
\item $J$ induces an effective $S^1$-action;
\item $F$ is proper;
\item the $S^1$-action is free on $F^{-1}(B_r)$, where $B_r\subset F(M)$ is the set of regular values.
\end{itemize}
We now want to recall the notation of monodromy in the presence of an $S^1$-action as in  Martynchuk \& Efstathiou \cite{efstathiou2017monodromy,martynchuk2017parallel}. Consider a closed curve $\gamma\subset B_r$ without self-intersections
and assume that the fibers $F^{-1}(b)$, with $b\in \gamma$, are connected. By the Arnold-Liouville Theorem \ref{t.AL} we have an $n$-torus bundle $(E_{\gamma}=F^{-1}(\gamma),\gamma,F)$ over $\gamma$. Consider a fiber $F^{-1}(b_0)$, $b_0\in \gamma$ and let $S^1$ be any orbit of the Hamiltonian $S^1$-action on $F^{-1}(b_0)$.
Choose a basis $(e_1,e_2)$ of the integer homology group $H_1(F^{-1}(b_0),\mathbb{Z})$, such that $e_1$ is a basis of $H_1(S^1,\mathbb{Z})$. Since the $S^1$-action is globally defined on $E_{\gamma}$, the generator $e_1$ is also globally defined. In this situation the monodromy matrix of the bundle $(E_{\gamma},\gamma,F)$ with respect to the basis $(e_1,e_2)$ has the form 
\begin{equation*}
 \begin{bmatrix}
  1 & m \\
  0 & 1
 \end{bmatrix},
\end{equation*}
where $m$ is be related to the $S^1$-action as explained below in Proposition \ref{p.monodromy}.
The formula in Proposition \ref{p.monodromy} is the same as in Theorem $6$ in Martynchuk \& Efstathiou \cite{martynchuk2017parallel}. But note that we are taking $\gamma$ to have the reverse orientation and our definition of isotropy weights has a sign difference from the one found in Martynchuk \& Efstathiou \cite{efstathiou2017monodromy,martynchuk2017parallel}.
\begin{proposition}
\label{p.monodromy}
Let $F:M\rightarrow \mathbb{R}^2$ be as above. Consider a curve $\gamma\subset B_r$ without self-intersections, oriented in the clockwise direction, with all fibers $F^{-1}(b)$, with $b\in \gamma$, connected and such that the $S^1$-action is free on $E_{\gamma}.$ Moreover, assume that :
\begin{itemize}
 \item There exists a $2$-disk $U$ in the image of $F$ such that $\partial U =\gamma$;
 \item The preimage $F^{-1}(\overline{U})$ is a closed submanifold with boundary of $M$;
 \item There are precisely $k$ fixed points $p_1,...,p_k\in F^{-1}(U)$ and the $S^1$-action is free on $F^{-1}\left(\overline{U}\right)\backslash\cup_{i=1}^{k}\{p_i\}$.
\end{itemize}
Then 
\begin{equation*}
 m=\sum_{i=1}^{k}\frac{1}{m_in_i},
\end{equation*}
\end{proposition}
where $m_i,n_i$ are the isotropy weights of the fixed point $p_i$.

%%\begin{remark}
 %%Using Proposition \ref{p.monodromy} we can compute the monodromy of a loop $\gamma$ around a flap just by looking at the isotropy weights of the $S^1$-action of the fixed points in the flap.
%%\end{remark}
\subsection{Fractional monodromy}
\label{s.fractionalmonodromy}
Fractional monodromy is an invariant that generalizes standard monodromy to singular torus fibrations. We refer to Efstathiou \& Martynchuck  \cite{martynchuk2017parallel} for more details, and we summarize their construction in the following. 

First we introduce the notion of a Seifert manifold and parallel transport along it. 
\begin{definition}
    A \textbf{Seifert} manifold $X$ is a compact orientable $3$-dimensional manifold which is invariant under an effective fixed point free $S^1$-action. In addition if the manifold has boundary $\partial X\neq \emptyset$ the action must be free on $\partial X$. 

    We call $\rho:X\rightarrow B:=X/S^1$ a \textbf{Seifert fibration}.
\end{definition}
Now consider a closed Seifert manifold $X$ and its Seifert fibration
\begin{equation*}
    \rho:X\rightarrow B=X/S^1.
\end{equation*}
Since $X$ is compact the number $N$ given by the least common multiple of the orders of nontrivial isotropy groups for the $S^1$-action is finite. Let $\mathbb{Z}_{N}:=\mathbb{Z}/N\mathbb{Z}$ denote the subgroup of order $N$ of the acting group $S^1$. Then $\mathbb{Z}_{N}$ acts on the Seifert manifold $X$. We thus have the so called reduction map $h:X\rightarrow X':=X/\mathbb{Z}_{N}$ and $\rho':X'\rightarrow B$ defined by $\rho=\rho'\circ h$. By construction, $\rho':X'\rightarrow B$ is a principal $S^1$ bundle over $B$. We denote its Euler number by $e(X')$.
\begin{definition}
    The \textbf{Euler number} of the Seifert fibration $\rho:X\rightarrow B$ is defined by $e(X):=e(X')/N$. 
\end{definition}
\begin{definition}
\label{d.paralleltransport}
    Let $X$ be a $3$-dimensional manifold with boundary $\partial X:=X_0\sqcup X_1$.
    The cycle $\alpha_1\in H_1(X_1)$ is a \textbf{parallel transport} of the cycle $\alpha_0\in H_1(X_0)$ along $X$ if 
    \begin{equation*}
        (\alpha_0,-\alpha_1)\in \partial_*(H_2(X,\partial X)),
    \end{equation*}
    where $\partial_*$ is the connecting homeomorphism of the exact sequence 
    \begin{equation*}
        \cdot \cdot \cdot \rightarrow H_2(X)\rightarrow H_2(X,\partial X)\xrightarrow[]{\partial_*} H_1(\partial X)\rightarrow H_1(X)\rightarrow \cdot \cdot \cdot
    \end{equation*}
\end{definition}
\begin{remark}
    Definition \ref{d.paralleltransport} can be reformulated as follows: $\alpha_1$ is a parallel transport of $\alpha_0$ along $X$ if there exists an oriented $2$-dimensional submanifold $S\subset X$ that "connects" $\alpha_0$ and $\alpha_1$, i.e., 
    \begin{equation*}
        \partial S = S_0\sqcup S_1, \quad [S_i]=(-1)^i\alpha_i\in H_1(X_i), \quad i=0,1.
    \end{equation*}
\end{remark}

Let $M$ be a $2n$-dimensional symplectic manifold. 
Consider a Lagrangian fibration $F:M\rightarrow P$ over an $n$-dimensional manifold $P$, given by a proper integral map $F$. Locally such a fibration gives rise to an integrable system. Let $\gamma:[0,1]\rightarrow F(M)$ be a continuous curve such that the set 
\begin{equation*}
    Y=\{(x,t)\in M\times [0,1]| \ F(x)=\gamma(t)\}
\end{equation*}
is connected and such that $\partial Y=Y_0\sqcup Y_1$ is a disjoint union of the two regular tori $Y_0:=F^{-1}(\gamma(0))$ and $Y_1:=F^{-1}(\gamma(1))$. Let 
\begin{equation*}
    H_1^{0}:=\{\alpha_0\in H_1(X_0)|\ \alpha_0\  \text{can be parallel transported along}\ Y\}.
\end{equation*}
\begin{definition}
    If the parallel transport along the above defined $Y$ defines an automorphism of the group $H_1^0$, then this automorphism is called \textbf{fractional monodromy} along $\gamma$. 
\end{definition}
Now suppose that $F:M\rightarrow P$ is invariant under an effective $S^1$-action. Take a simple closed curve $\gamma$ in $F(M)$, oriented in the clockwise direction, that satisfies the following regularity conditions:
\begin{itemize}
    \item the fiber $F^{-1}(\gamma(0))$ is regular and connected;
    \item the $S^1$-action is fixed point free on the preimage $E:=F^{-1}(\gamma)$;
    \item the preimage $E$ is a closed oriented connected submanifold of $M$.
\end{itemize}
Let $e(E)$ be the Euler number of $E$ and let $N$ denote the least common multiple of the orders of nontrivial isotropy groups for the $S^1$-action. Take a basis $(a,b)$ of the homology group $H_1(Y_0)\cong \mathbb{Z}^2$, where $b$ is given by an orbit of the $S^1$-action.
\begin{theorem}
\label{t.fractionalmonodromy}
\cite[Theorem $5$]{martynchuk2017parallel}
    Using the notation from above, fractional monodromy along $\gamma$ is well defined. Furthermore, $(Na,b)$ form a basis of the parallel transport group $H_1^0$ and the corresponding isomorphism has the form $b\mapsto b$ and $Na\mapsto Na+kb$ where $k=Ne(E)\in \mathbb{Z}$. 
\end{theorem}

Now let $i_0:Y_0\rightarrow Y$ and $i_1:Y_1\rightarrow Y$ denote the corresponding inclusions. The composition
\begin{equation*}
    i_1^{-1}\circ i_0:H_1(Y_0,\mathbb{Q})\rightarrow H_1(Y_0,\mathbb{Q})
\end{equation*}
gives an automorphism of $H_1(Y_0,\mathbb{Q})$. In a basis of $H_1(Y_0,\mathbb{Z})$ this isomorphism is written as a matrix in $GL(2,\mathbb{Q})$, called the \textbf{matrix of fractional monodromy}. Theorem \ref{t.fractionalmonodromy} shows that in a basis $(b,a)$ of $H_1(X_0)$, where $b$ corresponds to the $S^1$-action, the fractional monodromy matrix has the form
\begin{equation*}
    \begin{bmatrix}
        1 & e(E)=k/N \\
        0 & 1
    \end{bmatrix}.
\end{equation*}
In certain cases the parameter $e(E)$ is easily computed:
\begin{theorem}
\label{t.weightsfractionalmonodromy}
    \cite[Theorem $6$]{martynchuk2017parallel} Assume that $\gamma$ bounds a compact $2$-dimensional manifold $U\subset P$ such that $F^{-1}(U)$ has only finitely many fixed points $p_1,...,p_l$ of the $S^1$-action. Then 
    \begin{equation*}
        e(E)=\sum_{k=1}^{l}\frac{1}{m_kn_k}
    \end{equation*}
    where $m_k,n_k$ are the isotropy weights of the fixed points $p_k$. 
\end{theorem}

\section{Hypersemitoric systems}
\label{s.structureflap}

Hypersemitoric systems form a class of integrable Hamiltonian systems which generalizes that of semitoric systems by allowing for the existence of parabolic singularities. The latter are defined and characterized in Section~\ref{s.parabolicpoints}.

\begin{definition}
\label{d.hypersmitoric}
An integrable system $(M,\omega,F=(J,H))$ is called a \textbf{hypersemitoric system} if the following hold true:
\begin{enumerate}[label={(\roman*)}]
\item $J$ is proper and generates an effective $S^1$-action;
\item all degenerate singular points of $F$ (if any) are of parabolic type. 
\end{enumerate}
\end{definition}

The existence of the effective $S^1$-action with Hamiltonian function $J$ imposes restrictions to the types of nondegenerate singularities that can appear in hypersemitoric systems:
%, as it precludes the appearance of hyperbolic-hyperbolic singularities, while hyperbolic-elliptic singularities may only appear in fixed surfaces of the $S^1$-action and it is a result by Karshon that 
%fixed surfaces of the $S^1$-action can only be located at the maximum or minimum of $J$ \cite{karshon1999periodic}.
\begin{proposition}
\label{p.hyperbolicstuff}
(Hohloch $\&$ Palmer \cite[Proposition $4.1$]{hohloch2021extending})
Let $(M,\omega,F=(J,H))$ be an integrable system such that $J$ generates an effective $S^1$-action. Then :
\begin{itemize}
\item $(M,\omega,F)$ has no critical points of hyperbolic-hyperbolic type.
\item If $p \in M$ is a critical point of hyperbolic-elliptic type then it lies in a fixed surface of the $S^1$-action. 
\end{itemize}
\end{proposition}
\begin{remark}
 Karshon \cite{karshon1999periodic} showed that a fixed surface of the $S^1$-action can only be located at the maximum or minimum of $J$.
\end{remark}

Therefore, the only rank-$0$ nondegenerate singularities that can appear in a hypersemitoric system are elliptic-elliptic, focus-focus and hyperbolic-elliptic, while at rank $1$ both elliptic-regular and hyperbolic-regular nondegenerate singularities are possible. 
Notice that if a hypersemitoric system has parabolic singularities then it also has hyperbolic-regular and elliptic-regular singularities, see Section~\ref{s.parabolicpoints}. 

%\ke{Most of this paragraph fits more naturally in the introduction.}
Hypersemitoric systems form a significantly more general class than semitoric systems and---unlike semitoric systems---are not yet classified. 
Besides being a natural class of integrable Hamiltonian systems to symplectically classify, hypersemitoric systems are also important to understand because of their close relation to Hamiltonian $S^1$-spaces.
In particular, it was shown by Hohloch \& Palmer \cite{hohloch2021extending} that each Hamiltonian $S^1$-space $(M,\omega,J)$, see Definition~\ref{d.s1space}, lifts to a hypersemitoric system $(M,\omega,(J,H))$, i.e., there exists $H$ such that $\{H,J\}=0$ and $(M,\omega,F=(J,H))$ is hypersemitoric. 
%\ke{Given a Hamiltonian $S^1$ space can we always find a \textbf{simple} hypersemitoric lift? Corollary 5.10 in \cite{hohloch2021extending} is unclear concerning this question.} \ps{I have no idea, Sonja seems to think so.}
However, the class of general hypersemitoric systems defined above is very broad.
To simplify the discussion, in this work we consider the following subclass of hypersemitoric systems.

\begin{definition}
\label{d.simple-hypersemitoric}
An hypersemitoric integrable system $(M,\omega,F=(J,H))$ is called a \textbf{simple hypersemitoric system} if each connected component of a fiber $F^{-1}(f)$ contains at most one $S^1$ orbit of singular points.
\end{definition}

Singularities in simple hypersemitoric systems have typical arrangements that we call flaps and pleats, following \cite{efstathiou2012topology}.
We give the definition of flaps and pleats in Section~\ref{s.flaps/pleatsdefinition} and we discuss their differences.
Then we describe in detail the structure of flaps in Section~\ref{s.structureimageflap}.
% \ke{Should we have a corresponding section for pleats?}

% In this section we present the most basic properties of the image of a flap. First we start with an example of a system exhibiting such a structure. 

\subsection{Parabolic singularities} 
\label{s.parabolicpoints}

Parabolic singularities were referred to by Colin de Verdi\`ere \cite{colin2003singular} as ``the simplest non Morse example of singular points in integrable systems'', and moreover, we will see that they naturally occur in many systems with hyperbolic-regular points. 

\begin{definition}
Let $(M,\omega)$ be a $4$-dimensional symplectic manifold, $(M,\omega,G)$ an integrable system, and $p\in M$ a singular point of the map $G$ for which there exists a local diffeomorphism $\phi$ of $\mathbb{R}^2$ defined locally around $G(p)$ such that $df_1(p) \ne 0$, where $(f_1,f_2) := \phi\circ G$. Let 
\begin{equation*}
\tilde{f_2}:=\tilde{f}_{2,p}:=(f_2)|_{f_1^{-1}(f_1(p))}:f_1^{-1}(f_1(p))\rightarrow \mathbb{R}.
\end{equation*}
The point $p$ is a \textbf{parabolic degenerate singular point} (briefly, \textbf{parabolic point}, sometimes also called a \textbf{cuspidal point} or \textbf{cusp}) if the following hold true:
\begin{enumerate}[label={(\roman*)}]
\item The point $p$ is a critical point of the map $\tilde{f}_2$.
\item $\operatorname{rank}(d^2\tilde{f}_2(p))=1$.
\item There exists a vector $v\in \ker(d^2\tilde{f}_2(p))$ such that 
\begin{equation*}
v^3(\tilde{f_2}):=\frac{d^3}{dt^3}\bigg|_{t=0}\tilde{f}_2(\gamma(t)) \ne 0,
\end{equation*}
where the curve $\gamma(t):\ ]-\epsilon,\epsilon[\rightarrow f_1^{-1}(f_1(p))$ satisfies $\gamma(0)=p$ and $\dot{\gamma}(0)=v$; the definition of $v^3(\tilde{f_2})$ does not depend on the choice of $\gamma$, see \cite[Remark~2.1]{bolsinov2018symplectic}.
\item $\operatorname{rank}(d^2(f_2-kf_1)(p))=3$, where $k \in \mathbb{R}$ is determined by $df_2(p)=k\,df_1(p)$.
\end{enumerate}
The image of a parabolic singular point is called a $\textbf{parabolic singular value}$ of $F$, briefly, a \textbf{parabolic value}.
\end{definition}

Intuitively, a parabolic point is a singular point where the rank of all relevant maps is as maximal as possible without the point being nondegenerate.

% By Remark $2.1$ of Bolsinov, Guglielmi, and Kudryavtseva \cite{bolsinov2018symplectic}, the definition of $v^3(\tilde{f_2})$ does not depend on the choice of $\gamma$.

Parabolic points do not admit a symplectic normal form but they do admit a smooth normal form:

%\ke{Concerning the citation to EG12 below, I'm clearly biased. Please check if citing this makes sense here.}\ps{I'm okay with it.}

\begin{proposition}
\label{p.parabolicnormalform}
({Efstathiou $\&$ Giacobbe \cite[Proposition 2]{efstathiou2012topology}; Kudryavtseva \& Martynchuk \cite[Theorem 3.1]{kudryavtseva2021c}}) Let $p\in M$ be a parabolic singular point of an integrable system $(M,\omega,F=(f_1,f_2))$ for which $df_1(p)\neq 0$. Then there exists a neighborhood $U$ of $p$ equipped with coordinates $(x,y,t,\theta)$ centered at $p$, and a local diffeomorphism $g=(g_1,g_2)$ of $\mathbb{R}^2$ around the origin with $g_1(x_1,x_2)=\pm x_1 + const$ and $\partial g_2 / \partial x_2 \ne 0$ such that 
\begin{equation*}
g\circ F|_{U}=(t,x^3+tx+y^2).
\end{equation*}
\end{proposition}

\begin{example}
\label{ex.cusp}
The origin is a parabolic point for the integrable system given by the local normal form $F:\mathbb{R}^4\rightarrow \mathbb{R}^2, (x,y,t,\theta)\mapsto (t,x^3+tx+y^2)$ equipped with the symplectic form $\omega=dx\wedge dy+dt\wedge d\theta$. 
\end{example}

%Considering $\theta$ in Proposition \ref{p.parabolic} as a parameter, we see that parabolic points come in one parameter families in $M$, which project to a single point in $F(M)$.

%Furthermore, in a neighborhood of a parabolic point there are two surfaces of nondegenerate singular points which meet at the parabolic point, one of hyperbolic-regular points and one of elliptic-regular points. 

\begin{remark}
In the analytic case the complete set of symplectic invariants of parabolic points and parabolic orbits is described by Bolsinov, Guglielmi, and Kudryavtseva \cite{bolsinov2018symplectic}. This was extended to the smooth case by Kudryavtseva \& Martynchuk \cite{kudryavtseva2021c}.
\end{remark}

\begin{figure}[htb]
\centering
\includegraphics[width=0.3\textwidth]{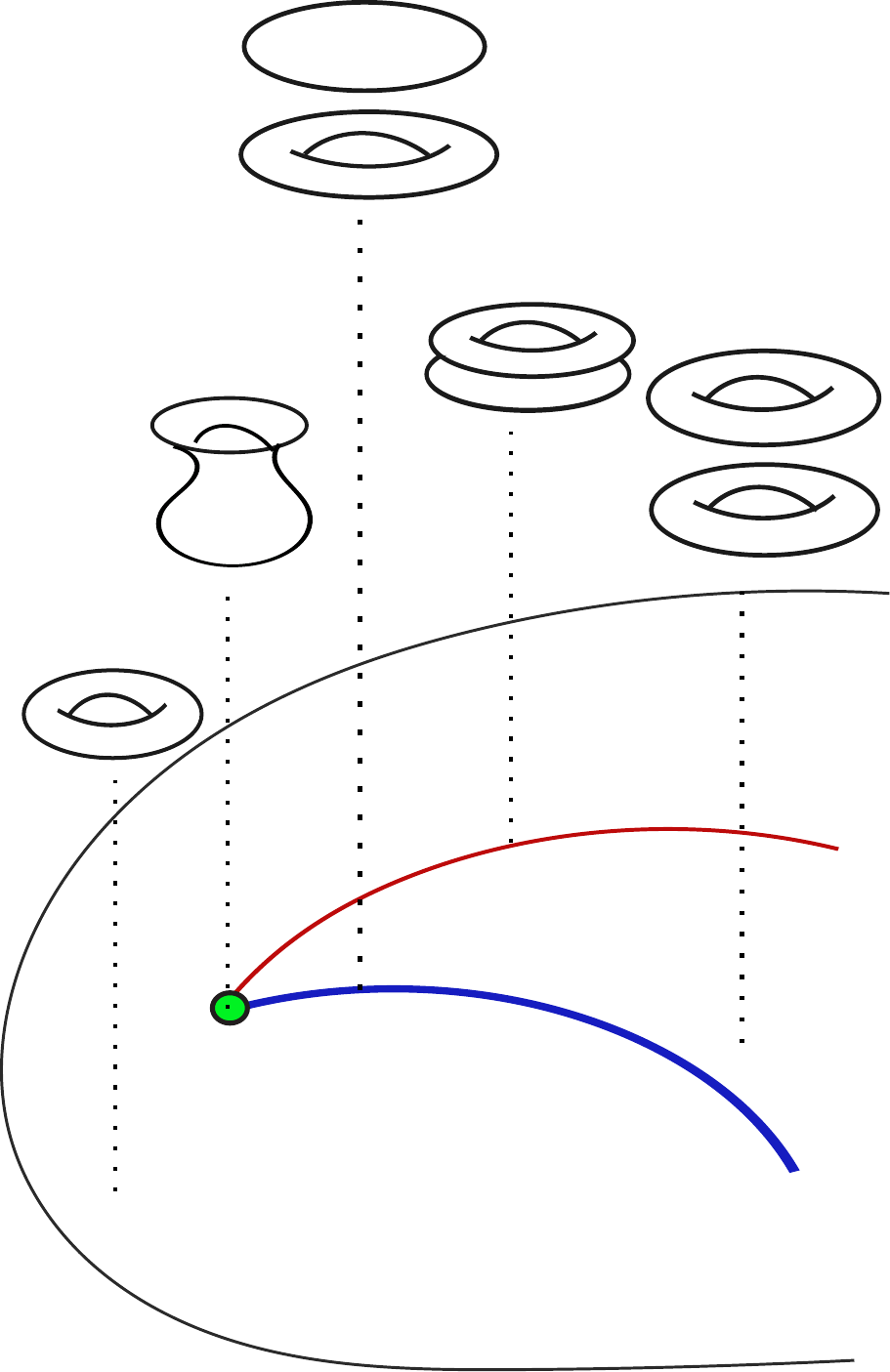}
\hspace{1cm}
\includegraphics[width=0.4\textwidth]{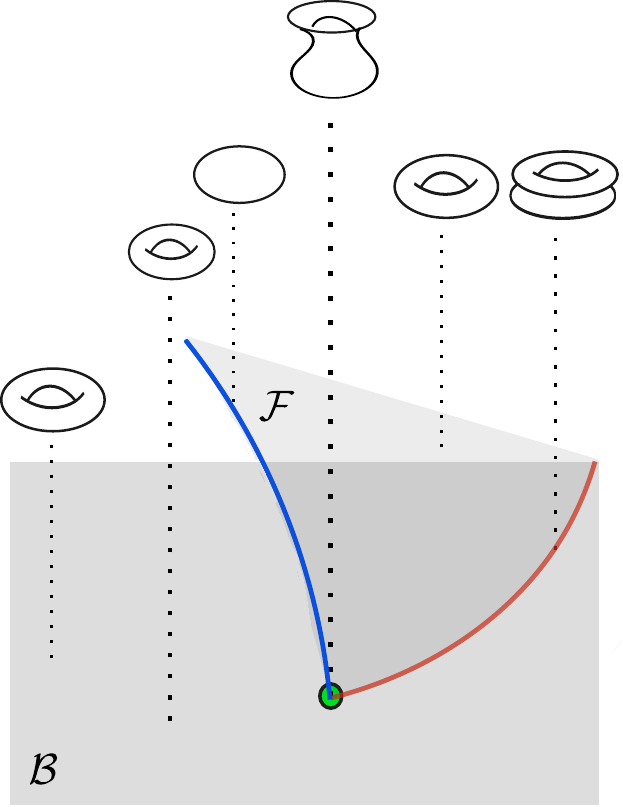}
\caption{On the left the local structure of the critical values of $F$ near a cuspidal value. The green points represent cuspidal values. On the right the local structure of the local flap $\mathcal{F}$ in the unfolded momentum diagram. The local base is represent by $\mathcal{B}$.
The red lines represent hyperbolic-regular values, and the blue lines represent elliptic-regular values. }
\label{p.localstructurecusp}
\end{figure}

\subsection{Parabolic singularities in simple hypersemitoric systems}
\label{s.parabolichypersemitoric}

Let $c$ be a parabolic value of $F$ of a simple hypersemitoric system.
The fiber $F^{-1}(c)$ is a so-called cuspidal torus, i.e., the product of $S^1$ with a `teardrop'.
% Denote by $\widehat{C}$ the $S^1$ orbit of cuspidal points on $F^{-1}(c)$
The normal form in Proposition~\ref{p.parabolicnormalform} implies that in a neighborhood of $c$ the set of critical values of $F$ consists of two curves emanating from $c$, see Fig.~\ref{p.localstructurecusp}.
One curve consists of hyperbolic-regular values.
For each such hyperbolic-regular value, the corresponding fiber is a so-called bitorus, i.e., the product of $S^1$ with a figure eight, and it contains an $S^1$ orbit of hyperbolic-regular points.
Approaching $c$, one of the loops of the figure eight shrinks and vanishes, i.e., the figure eight degenerates to the teardrop on $F^{-1}(c)$. 
The other curve consists of elliptic-regular values. 
For each such elliptic-regular value, the corresponding fiber is the (disjoint) union of an $S^1$ orbit of elliptic-regular points and a smooth torus $\mathbb{T}^2$.
% Each such orbit is a connected component of the corresponding fiber of $F$.
% The local structure of the critical values of $F$ near a cuspidal value is shown in Fig.~\ref{p.localstructurecusp}.
The two curves of critical values separate an open neighborhood $U$ of $c$ in two regions of regular values. 
In one region the fibers are connected and on the other region the fibers are disconnected. 
For more details see, for example, Efstathiou \& Giacobbe \cite{efstathiou2012topology}.

%Going into more detail, from a cusp critical value $c$ there originate two curves of critical values, one of elliptic-regular type, i.e., the only singularities in the preimage are of elliptic-regular type, and one of hyperbolic-regular type, i.e., the only singularities in the preimage are of hyperbolic-regular type.
%The preimage of each elliptic-regular critical value consists of two disjoint components, a smooth circle $S^1$ and a smooth torus $\mathbb{T}^2$. The two curves of critical values originating from the cusp critical value separate an open neighborhood of $U$ of $c$ in two regions of regular values. In one region the fibers are connected and on the other region the fibers are disconnected. 

For the integrable system $(M,\omega,F)$ one can define a branched covering of $F(M)$, i.e., there exists a stratified space $A$, usually referred to as \textbf{unfolded momentum domain}, a map $\widetilde{F} : M \rightarrow A$ and a projection $\pi : A \rightarrow F(M)$ such that the level sets of $\widetilde{F}$ are the connected components of the level sets of $F$ and $F = \pi \circ \tilde{F}$. 

The earlier discussion concerning the fibration in a neighborhood $U$ of a parabolic value $c$ implies that the set $\pi^{-1}(U) \subset A$ has two sheets which admit the following description, see Fig.~\ref{p.localstructurecusp}. One sheet, which we call the \text{local base} $\mathcal{B}$, contains regular values, the curve of hyperbolic-regular values and no elliptic-regular values. The other sheet, which we call the \textbf{local flap} $\mathcal{F}$, contains the curve of elliptic-regular values, the curve of hyperbolic-regular values as well as regular values ``in between'' these curves. The two sheets intersect along the curve of hyperbolic-regular values, and at the parabolic value. The curve of regular values in the topological boundary of $\mathcal{F}$ is called the \textbf{free boundary} of $\mathcal{F}$. 

For hypersemitoric systems it is important to consider ``normal forms'' near parabolic points that preserve the generator $J$ of the $S^1$ action. We have the following result.

\begin{lemma}\label{p.parabolicnormalformfixingj}
Let $(M,\omega,F=(J,H))$ be a hypersemitoric system and $p$ be a parabolic critical point. 
Then there exist local coordinates $(x,y,t,\theta)$ such that $(x,y,t,\theta)|_p = (0,0,J(p),0)$ and 
\[ H = \kappa ( x^3+y^2+f(t)x+g(t) ), \quad J = t, \]
where $f$, $g$ are real analytic functions with $f(J(p)) = 0$, $f'(J(p)) \ne 0$ and $\kappa \in \{-1,+1\}$.
\end{lemma}

\begin{proof}
The result directly follows from \cite[Lemma~2.2]{bolsinov2018symplectic}.
\end{proof}

A direct consequence of Lemma~\ref{p.parabolicnormalformfixingj} is that the set of critical values near the parabolic value $c = F(p)$ is given by
\[ H = \kappa \left( g(J) \pm \frac{2}{3\sqrt3} (-f(J))^{3/2} \right), \quad f(J) \le 0. \]
That is, the set of critical values consists locally of two curves such that each curve is a graph over $J$ with $J \le J(p)$ if $f'(J(p)) > 0$ and $J \ge J(p)$ if $f'(J(p)) < 0$. 
The curve with the `$+$' sign corresponds to hyperbolic-regular values while the curve with the `$-$' sign corresponds to elliptic-regular values.

\subsection{Flaps and pleats}
\label{s.flaps/pleatsdefinition}
Consider now the case where the integrable system $(M, \omega, F)$ contains a curve of critical values parameterized by $c: [0,1] \to \mathbb{R}^2$, $s \mapsto c_s$, so that the endpoints $c_0$ and $c_1$ are parabolic critical values while the interior points $c_s$ for $s\in\ ]0,1[$ are hyperbolic-regular critical values.
As described earlier, as $s$ approaches $0$ or $1$ one of the loops of the figure eight of the bitorus $C_s = F^{-1}(c_s)$ shrinks and vanishes, and the figure eight degenerates to a teardrop. 
Then there are two possible cases: either the loop that vanishes when approaching $c_0$ is the same that vanishes when approaching $c_1$, or not.
The former case typically gives rise to a \emph{flap}, Fig.~\ref{p.flap} (left), while the latter typically gives rise to a pleat, Fig.~\ref{p.pleat} (right).

\begin{figure}[ht]
\centering
\includegraphics[width=0.4\textwidth]{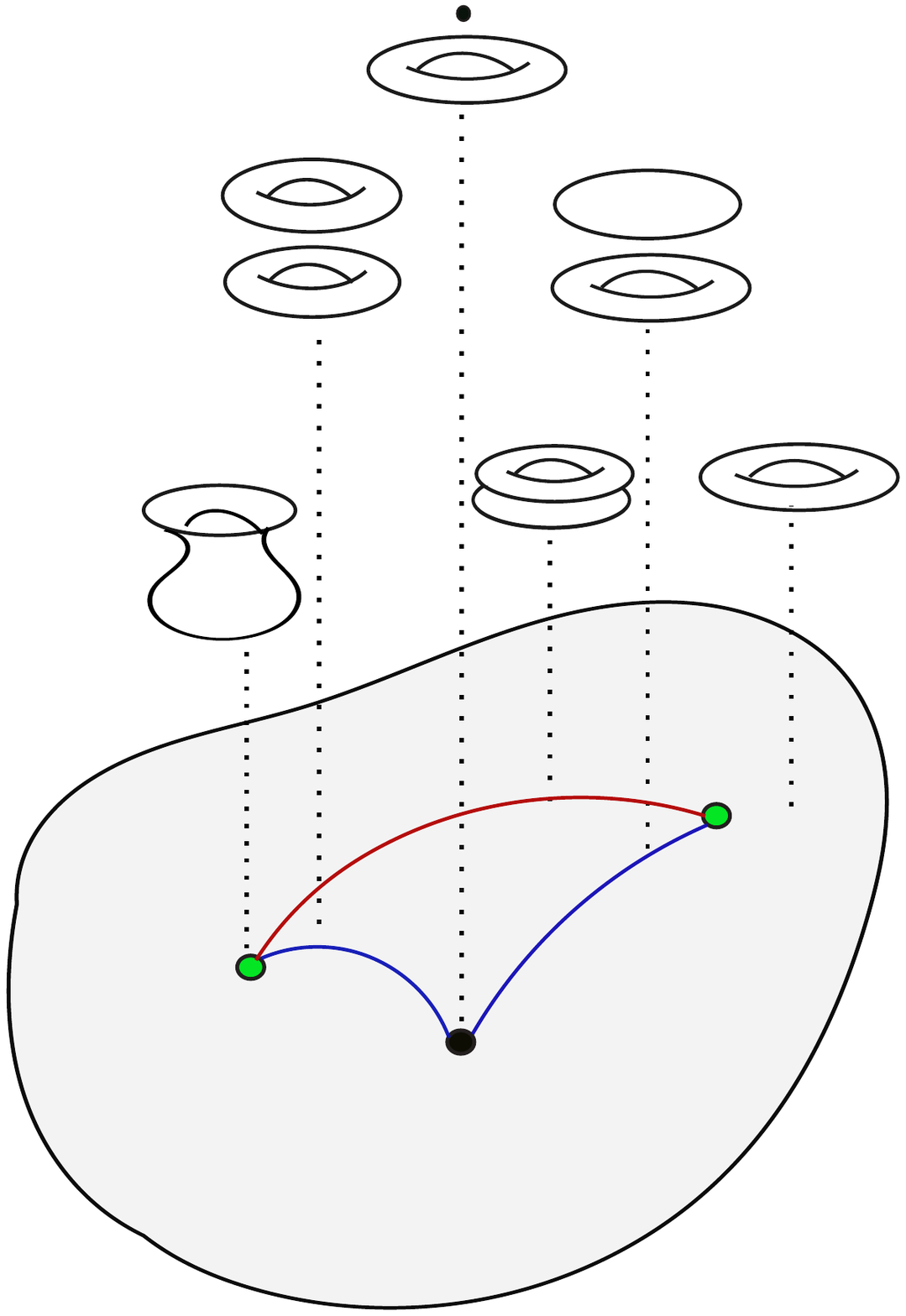}
\hspace{1cm}
\includegraphics[width=0.4\textwidth]{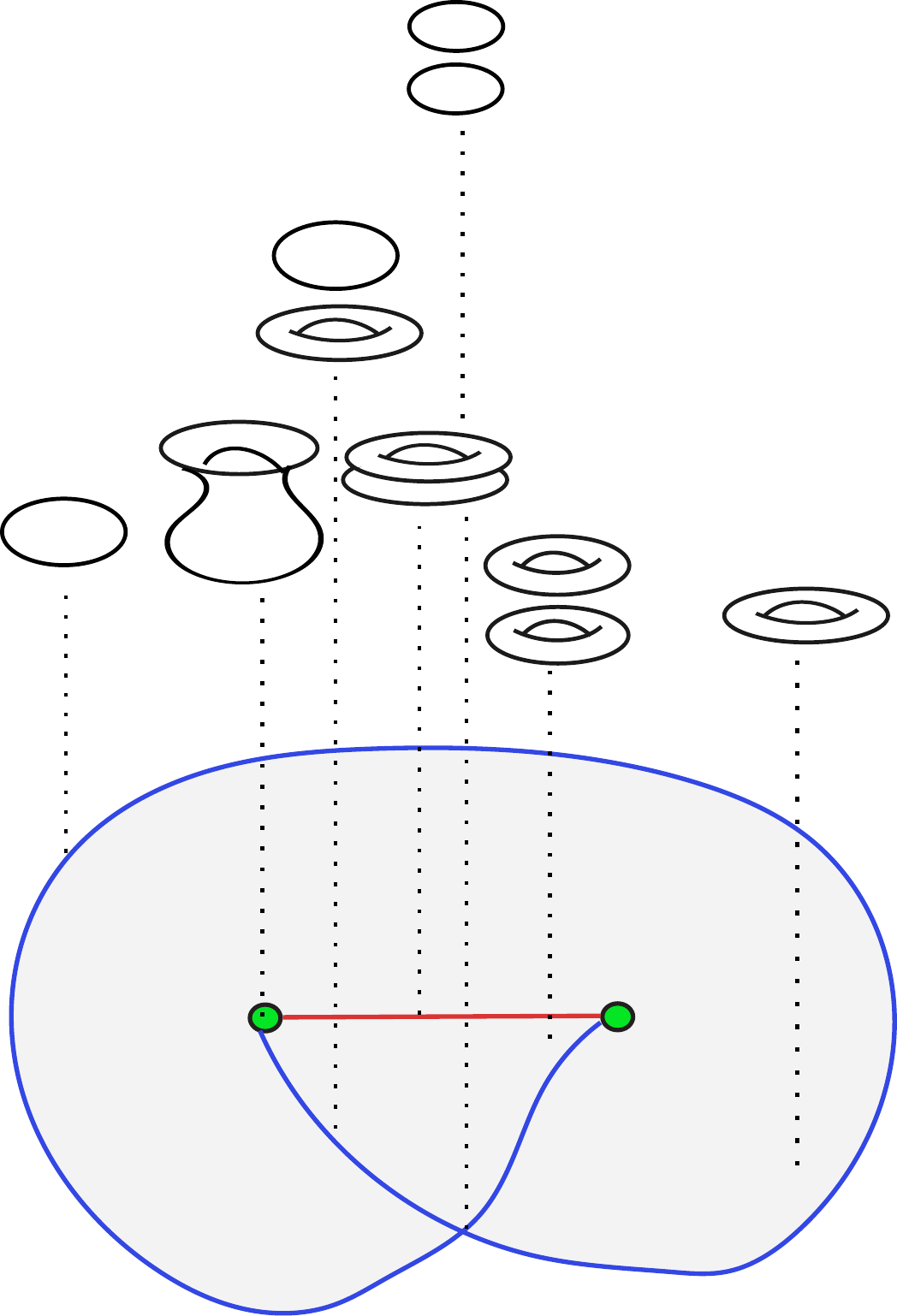}
\caption{Left: Standard flap. Right: Pleat. The black point at the right represents an elliptic-elliptic value. In both examples, the green points represent parabolic values, the red lines represent hyperbolic-regular values, and the blue lines represent elliptic-regular values. }
\label{p.flap}
\label{p.pleat}
\end{figure}

For $i\in \{0,1\}$ let $\mathcal{B}_i$ denote the local bases of $c_i$ and $\mathcal{F}_i$ denote the local flaps of $c_i$. Then, the \textbf{flap topology} is obtained if we glue a subset of the free boundary of $\mathcal{F}_1$ to a subset of the free boundary of $\mathcal{F}_2$, and glue a subset of the boundary of $\mathcal{B}_1$ to a subset of the boundary of $\mathcal{B}_2$. Similarly, the \textbf{pleat topology} is obtained if we glue a subset of the free boundary of $\mathcal{F}_1$ to a subset of the boundary of $\mathcal{B}_2$, and glue a subset of the free boundary of $\mathcal{F}_2$ to a subset of the boundary of $\mathcal{B}_1$.

\subsection{Structure of flaps}
\label{s.structureimageflap}
%\begin{definition}
 %   A flap $\mathcal{F}$ in a hypersemitoric system $(M,\omega,F=(J,H))$ is called \textbf{simple} if there are no critical points in its interior.
%end{definition}

%\begin{remark}
 %   The difference between a standard and a simple flap is that standard flaps have at least one elliptic-elliptic value. 
%\end{remark}

If $c : [0,1] \to \mathbb{R}^2$ has the flap topology, then the sheet in $A$ that contains the curves of elliptic-regular values emanating from each of the parabolic values $c_0$ and $c_1$ is called a \textbf{flap}, and we denote it again by $\mathcal{F}$. Notice that $\mathcal{F}$ contains both local flaps $\mathcal{F}_0$ and $\mathcal{F}_1$. The set $\pi(\mathcal{F})$ is called the \textbf{image of the flap} and we denote it by $\operatorname{Im}(\mathcal{F})$.
Additionally, we denote by $\gamma_{\mathcal{F}} = c : [0,1] \to F(M)$ the curve of critical values joining the two parabolic values.
The sheet that contains the local bases of $c_0$ and $c_1$ is called the \textbf{base}, and we denote it by $\mathcal{B}$. The set $\pi^{-1}(\pi(\mathcal{F}))\cap \mathcal{B} \subset A$ is called the \textbf{background of the flap} $\mathcal{F}$.

% Given a flap $\mathcal{F}$ in an integrable system $(M,\omega,F)$, the region bounded by the triangular shape shown in Figure \ref{p.flap} is called the \textbf{image of the flap} and we denote it by $\operatorname{Im}(\mathcal{F})$, i.e., $\operatorname{Im}(\mathcal{F}):=\pi(\mathcal{F})$. \ke{The definition of the ``image of the flap'' should be made more precise.} 
%
%Furthermore, recall that associated with a flap there exists a path of hyperbolic-regular and degenerate values. That is, there exists a path $\gamma_{\mathcal{F}}:[0,1]\rightarrow F(M)$ such that $\gamma_{\mathcal{F}}(t)$ for $t\in \ ]0,1[$ is a hyperbolic-regular value and $\gamma_{\mathcal{F}}(0)$, $\gamma_{\mathcal{F}}(1)$ are degenerate values. For a flap $\mathcal{F}$ we denote the image of the path $\gamma_{\mathcal{F}}:[0,1]\rightarrow F(M)$ also by $\gamma_{\mathcal{F}}$.
In order to define an affine invariant for a simple hypersemitoric system, we have to deal with the appearance of flaps. The appearance of critical values on the image of the flap may lead to extra difficulties. This is because the critical values will introduce monodromy in the underlying integrable system. Therefore, in order to study the simplest type of flaps for which defining an affine invariant is non-trivial, which is done in Section \ref{s.flap}, and gain some intuition for the general problem, we introduce the following definitions:
\begin{definition}
\label{d.standardflap}
A flap $\mathcal{F}$ is \textbf{standard} if there are no critical values in the interior of $\operatorname{Im}(\mathcal{F})$ and there exists at least one elliptic-elliptic value in the boundary of $\operatorname{Im}(\mathcal{F})$.
\end{definition}

\begin{definition}
    Let $(M,\omega,F)$ be an integrable system such that the only critical values outside of the boundary occur in the image of flaps. Then the \textbf{background of the system} is the sheet in its unfolded momentum domain $A$ that contains the background of the flaps. 
\end{definition}

\begin{definition}
\label{d.bounded}
    Let $\mathcal{F}$ be a flap in a hypersemitoric system $(M,\omega,F=(J,H))$, and $\mathcal{B}$ the background of the flap $\mathcal{F}$. Furthermore, let $\pi$ be the projection associated with the unfolded momentum domain of $(M,\omega,F)$. The flap $\mathcal{F}$ is said to be \textit{bounded by the background} if $\pi(\mathcal{F})\subseteq \pi(\mathcal{B})$. Otherwise, we say its \textit{unbounded by the background}. For sake of readability, we often drop {\em `by the background'} later in the paper.
\end{definition}

Our aim in this section is to prove the following result describing the boundary of the image of a flap.

\begin{proposition}
\label{p.flapstructure}
Let $(M,\omega,F=(J,H))$ be a hypersemitoric system and $\mathcal{F}$ a flap with parabolic values $c_0 = \gamma_{\mathcal{F}}(0)$, $c_1 = \gamma_{\mathcal{F}}(1)$. 
Then $J|_{F^{-1}(\operatorname{Im}(\mathcal{F}))} = \, [J(c_0),J(c_1)]$.
Moreover, the boundary of $\operatorname{Im}(\mathcal{F})$ consists of the two parabolic values $c_0$, $c_1$ joined by two curves $\gamma^h$ and $\gamma^e$, such that each curve is a graph over the $J$-interval $]J(c_0), J(c_1)[$. In particular:
\begin{itemize}
\item The curve $\gamma^h$ consists only of hyperbolic-regular values and $\operatorname{Im}(\gamma^h) = \operatorname{Im}(\gamma_{\mathcal{F}}) \setminus \{c_0,c_1\}$. 
\item The curve $\gamma^e$ consists of elliptic-regular values possibly interrupted by isolated elliptic-elliptic values. 
\end{itemize}
\end{proposition}

% First consider the following auxiliary results:

%\begin{lemma}
%\label{l.fixedsurfaceS1action}
%(Karshon \cite{karshon1999periodic})
% Let $(M,\omega,F=(J,H))$ be an integrable system such that $J$ generates an effective $S^1$-action. Then the fixed surfaces of the $S^1$-action can only be located at the maximum or minimum of $J$. 
%\end{lemma}

%When we introduced the notation of a flap in Subsection \ref{s.flaps/pleatsdefinition} it was via the topology of the fibers. In what follows we describe the consequences this has on the boundary of its image:  
%\begin{proposition}
%Let $(M,\omega,F=(J,H))$ be a hypersemitoric system with a flap $\mathcal{F}$. Then the boundary of $\operatorname{Im}(\mathcal{F})$ consists of two parabolic values joined by two types of curves in the following way: one curve consists only of hyperbolic-regular values. The other curve consists of elliptic-regular values possibly interrupted by isolated elliptic-elliptic values. 
%\end{proposition}

First consider the following auxiliary results:

\begin{lemma}\label{l.nonverticaltangencieshyperbolic}
(Hohloch $\&$ Palmer \cite[Lemma $4.2$]{hohloch2021extending}) 
Let $(M,\omega,F=(J,H))$ be an integrable system such that $J$ generates an effective $S^1$-action and let $C\subset M$ be a connected component of the set of hyperbolic-regular points of $M$. Then $F(C)$ does not have vertical tangencies. 
\end{lemma}

Analogously, we have:

\begin{lemma}
\label{l.nonverticaltangenciesellitpic}
Let $(M,\omega,F=(J,H))$ be an integrable system such that $J$ generates an effective $S^1$-action and let $C \subset M$ be a connected component of the set of elliptic-regular points of $M$ such that $F(C) \in \operatorname{int}({F(M)})$. Then $F(C)$ does not have vertical tangencies.
\end{lemma}

\begin{proof}
Let $p\in M$ be an elliptic-regular point such that $p\in C$. According to Bolsinov \& Fomenko \cite[Proposition $1.16$]{bolsinov2004integrable}, the set $F(C)$ is a one-dimensional immersed submanifold and there exist $a,b\in \mathbb{R}$ such that $bX_{J}(p)-aX_{H}(p)=0$ holds for the Hamiltonian vector fields of $J$ and $H$. Thus the tangent vector of the curve $F(C)$ at $F(p)$ is given by $(a,b)$. If $F(C)$ has a vertical tangency at $p$ then $a=0$, $b\neq 0$ and $dJ(p)=X_J(p)=0$. By Hohloch \& Palmer \cite[Lemma $2.43$]{hohloch2021extending} this implies that $F(p) \in \partial(F(M))$. Therefore, we obtain a contradiction.\end{proof}

We can now prove the main result of this section:

\begin{proof}[Proof of Proposition~\ref{p.flapstructure}]
%First recall that by definition of a flap there exist two parabolic critical values that are connected by a curve of hyperbolic-regular values in its image. %By definition this curve of hyperbolic-regular values does not have any other singularities. That this definition is well is due to the local normal form for nondegenerate points (Theorem \ref{t.localnormalform}), the local normal form for parabolic singularities (Proposition \ref{p.parabolicnormalform}), Proposition \ref{p.hyperbolicstuff} and Lemma \ref{l.fixedsurfaceS1action}. In more detail: if the curve of hyperbolic-regular values had other singularities, by Lemma \ref{l.fixedsurfaceS1action}, Proposition \ref{p.hyperbolicstuff} and Theorem \ref{t.localnormalform} it would have to be a parabolic value. However this is not possible due to the local normal form for a parabolic singularity, Proposition \ref{p.parabolicnormalform}.
%
By definition, a flap $\mathcal{F}$ has a curve $\gamma_{\mathcal{F}} : [0,1] \to \mathbb R^2$ of critical values so that the endpoints $c_0 = \gamma_{\mathcal{F}}(0)$, $c_1 = \gamma_{\mathcal{F}}(1)$ are parabolic values, while for $t \in \ ]0,1[$, $\gamma_{\mathcal{F}}(t)$ is a hyperbolic-regular value. Then Lemma~\ref{l.nonverticaltangencieshyperbolic} implies that the curve $\gamma_{\mathcal{F}} |_{]0,1[}$ is a graph over $J$ and thus it can be parameterized as $\gamma^h : \, ]j_0,j_1[ \to \mathbb R^2$ with $J \circ \gamma^h = \operatorname{id}$ and $j_0 = J(c_0)$, $j_1 = J(c_1)$.

By the local normal form for a parabolic singularity (Proposition~\ref{p.parabolicnormalform}), from a parabolic value there emanate two curves of critical values, one of hyperbolic-regular type, and the other of elliptic-regular type.
Suppose that the curve starting with elliptic-regular values later contains other types of singularities. Then by the local normal form for nondegenerate singularities (Theorem~\ref{t.localnormalform}) this may happen only at an elliptic-elliptic value or at a parabolic value. 
Apart from the two endpoints, this cannot occur at a parabolic value due to the local normal form for a parabolic singularity (Proposition~\ref{p.parabolicnormalform}), as we explain below.
Therefore, the only possibilities of other singularities in the curve are elliptic-elliptic values. In particular, the boundary of the image of a flap has only two parabolic values.

Assume that the statements above are not true. Then, by the local normal form for a parabolic singularity (Proposition \ref{p.parabolicnormalform} and Lemma \ref{p.parabolicnormalformfixingj}) there are two options:

\begin{itemize}

\item Either there would be a vertical tangency in a line of elliptic-regular values that is in the interior of $F(M)$, which is impossible due to Lemma \ref{l.nonverticaltangenciesellitpic},

\item or there would exist an elliptic-elliptic point $p$ such that $F(p) \in \text{int}(F(M))$ and $J(p)$ is a local extremum of $J$. By \cite[Lemma $5.1$]{guillemin1982convexity}, $J(p)$ is a global extremum of $J$. This contradicts the fact that $F(p)\in \text{int}(F(M))$.
\end{itemize}

Therefore, we obtain the desired result. 
\end{proof}

Right above we described the image of a flap of a simple hypersemitoric system. There is however another (quite similar) structure that can appear. Recall that for a parabolic value there is an associated curve of hyperbolic-regular values, see Lemma \ref{p.parabolicnormalformfixingj}. This curve of hyperbolic-regular values does not need to end on another parabolic value, i.e., the hyperbolic-regular points can converge to a hyperbolic-elliptic point, which lies in a fixed surface of the $S^1$-action, whose image lies in the boundary of the momentum map, more specifically in the image of an extremal level set of $J$,  see Hohloch \& Palmer \cite[Lemma $2.1$ \& Proposition $4.1$ \& Corollary $4.3$ ]{hohloch2021extending}. In this situation, we call the structure emanating by the local flap $\mathcal{F}$ a \textit{boundary flap} $\mathcal{F}$. Analogously to the case of a flap we define the background of a boundary flap, the curve $\gamma_{\mathcal{F}}:[0,1]\rightarrow F(M)$ such that $\gamma_{\mathcal{F}}:\ ]0,1[\rightarrow F(M)$ consists of hyperbolic-regular values, and \textit{bounded} and \textit{unbounded} boundary flaps. The image of a boundary flap is represented in Figure \ref{p.boundaryflap}, and the analogous of Proposition \ref{p.flapstructure} follows.
\begin{figure}[ht]
\centering
\includegraphics[width=0.4\textwidth]{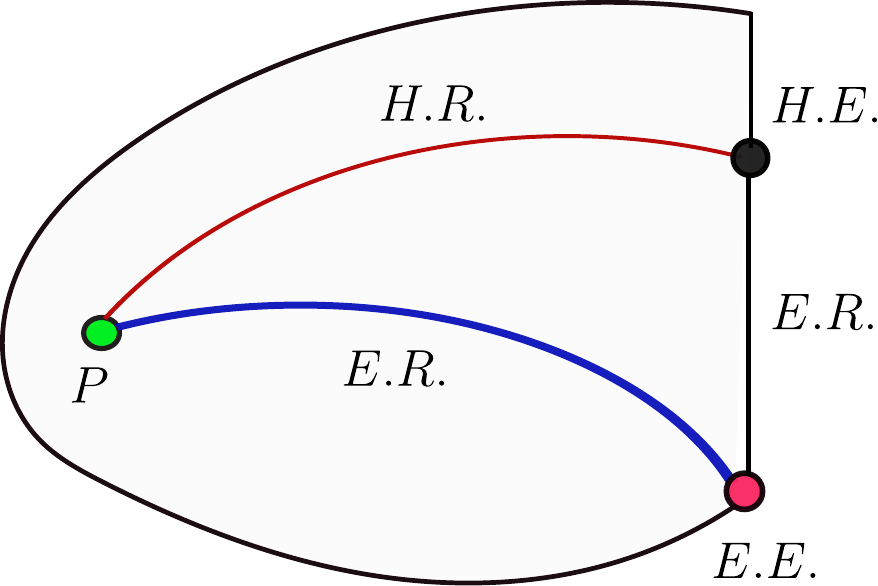}
\\
\caption{Example of the image of a boundary flap. The parabolic value is drawn in green, the hyperbolic-regular values in red and the elliptic-regular values in blue. The black dot represents the hyperbolic-elliptic regular value in $\partial(F(M))$, which lies the image of an extremal level set of $J$. The pink dot represents the elliptic-elliptic value in $\partial(F(M))$.}
\label{p.boundaryflap}
\end{figure}
Recalling the notation of Section \ref{s.flaps/pleatsdefinition}, if $c:[0,1]\rightarrow \mathbb{R}^2$ has the pleat topology, we call the region bounded by the curve $c$ and the curves emanating from the parabolic values $c(0),c(1)$ in the image of the momentum map a \textit{pleat}, see Figure \ref{p.pleat}. For an example of an integrable system exhibiting a pleat in its momentum map image, see Section \ref{s.pleatexample}.
Analogously to Proposition \ref{p.flapstructure} we can also describe the structure of pleats:
\begin{proposition}
\label{p.pleatstructure}
Let $(M,\omega,F=(J,H))$ be a hypersemitoric system and $P$ a pleat with parabolic values $c_0 = \gamma_{\mathcal{P}}(0)$, $c_1 = \gamma_{\mathcal{P}}(1)$. 
Then $J|_{F^{-1}(P)} = \, [J(c_0),J(c_1)]$.
Moreover, the boundary of $P$ consists of the two parabolic values $c_0$, $c_1$ and three  curves $\gamma^h$, $\gamma_1^e$ and $\gamma_2^e$, such that the curve $\gamma^h$ is a graph over the $J$-interval $]J(c_0), J(c_1)[$. In particular:
\begin{itemize}
\item The curve $\gamma^h$ consists only of hyperbolic-regular values and $\operatorname{Im}(\gamma^h) = \operatorname{Im}(\gamma_{\mathcal{F}}) \setminus \{c_0,c_1\}$. 
\item The curves $\gamma_1^e$ and $\gamma_2^e$ consist of elliptic-regular values possibly interrupted by isolated elliptic-elliptic values, and intersect transversally at a point in the boundary of $F(M)$. 
%\ke{If we go a bit further away from the parabolic points along the elliptic-regular lines, more complicated things could in principle appear. Can we be sure that the two elliptic-regular lines intersect transversally somewhere and that nothing more complicated will happen?}
\end{itemize}
\end{proposition}

\subsection{The modified Jaynes-Cummings model}
\label{s.jaynescummingsmodel}

%\ke{To me it would make more sense to move this to the next section.}
%\ps{Whatever you think is best, I don't really have a preference.}

We give here an example of a hypersemitoric system, originally introduced by Dullin \& Pelayo in \cite{dullin2016generating}.
Consider $M:=\mathbb{R}^2\times S^2$ with the symplectic form $\omega_{0}\oplus \omega_{S^2}$ where $\omega_0$ is the standard symplectic form in $\mathbb{R}^2$ and $\omega_{S^2}$ is the standard symplectic form in $\mathbb{S}^2$. For Cartesian variables $(u,v,x,y,z)$ where $(u,v)\in \mathbb{R}^2$ and $(x,y,z)\in \mathbb{R}^3$ such that $x^2+y^2+z^2=1$ we define $F=(J,H):M\rightarrow \mathbb{R}^2$ with 
\begin{equation*}
J(u,v,x,y,z) := \frac{1}{2} (u^2+v^2) + z , \quad 
H(u,v,x,y,z) := \frac{1}{2} (xu+yv).
\end{equation*}
The system $(M,\omega,F)$ is called the Jaynes-Cummings model; it is a semitoric system with one focus-focus singularity at the point $m=(0,0,0,0,1)$. 
Its polytope invariant and all the other invariants of the semitoric system can be found in Section~5 of Alonso et al. \cite{alonso2019symplectic} and references therein. A representative of the polytope invariant is shown in Figure~\ref{f.JCpolytope}.

\begin{figure}[ht]
\centering
\includegraphics[width=0.6\textwidth]{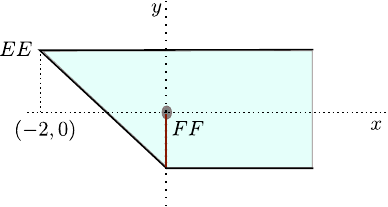}
\caption{A representative of the polytope invariant of the Jaynes-Cummings model for $\epsilon=-1$ (for the meaning of $\epsilon$ see the discussion in Section~\ref{s.semitoric}). The point labeled $FF$ is the focus-focus value and the point labeled $EE$ is the elliptic-elliptic value of the system. The red line represents the cut associated with the choice of $\epsilon$. }
\label{f.JCpolytope}
\end{figure}

Dullin \& Pelayo \cite{dullin2016generating} considered the system $(M,\omega,F=(J,H+G))$ where the function $G$ is given by $G(u,v,x,y,z)=\gamma z^2$ with $\gamma = 4/5$. In the system $(M,\omega,(J,H+G))$ the focus-focus singularity of the Jaynes-Cummings model is replaced by an elliptic-elliptic singularity that lies on a flap, see Figure~\ref{f.bifurcationdiagramJCmodification}. 

\begin{figure}[ht]
\centering
\includegraphics[]{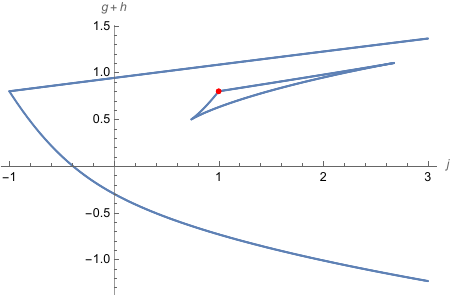}
\caption{The bifurcation diagram for the modified Jaynes-Cummings system $(M, \omega, (J,H+G))$. The red dot represents the elliptic-elliptic value on the image of the flap.}
\label{f.bifurcationdiagramJCmodification}
\end{figure}

\section{Introduction to quantization and joint spectrum}
\label{s.introquantization}

Briefly, quantization is a process that takes a classical phase space, here a symplectic manifold $M$, to a Hilbert space $\hat{M}_{\hbar}$, and a classical Hamiltonian $f\in C^{\infty}(M)$ to a self-adjoint operator $\hat{f}_{\hbar}$ acting on $\hat{M}_{\hbar}$, where $\hbar\in \mathbb{R}^{>0}$ is called the Planck constant. Furthermore, $\lim_{\hbar\rightarrow 0}\text{spec}(\hat{f}_{\hbar})=\operatorname{Im}(f)$. This procedure is also applied to tuples of Poisson commuting Hamiltonians.

In Subsection \ref{s.quantizationflap} we will show how to quantize the modification of the Jaynes-Cummings model, defined in Subsection \ref{s.jaynescummingsmodel}, and in Subsection \ref{s.quantizationhirzebruch} we will show how quantize the so called Hirzebruch surface. Experts may skip this section. In the upcoming sections we use the methods of quantization to help compute the (representatives of the) affine invariants of specific integrable systems.

\subsection{The joint spectrum of the modified Jaynes-Cummings model}
\label{s.quantizationflap}
This subsection serves as an introduction to the methods of quantization and our line of thought follows closely Pelayo \& \vungoc\ \cite{pelayo2012hamiltonian}. 

Recall the integrable system $(\mathbb{R}^2\times S^2,\omega_0\oplus \omega_{S^2},(J,H+G))$ defined in Section \ref{s.jaynescummingsmodel}. We now proceed to quantize this system according to Pelayo \& \vungoc\ \cite{pelayo2012hamiltonian}.

In order to quantize the Jaynes-Cummings model we first recall the quantization of the harmonic oscillator, and the quantization of the Cartesian coordinate functions on $S^2$.

We start by recalling the quantization of the harmonic oscillator: let $M=\mathbb{R}^2$ with coordinates $(u,v)$, canonical symplectic form $\omega_0$, and Hamiltonian $N(u,v)=\frac{u^2+v^2}{2}$. We call the system $(M,\omega_0,N)$ the \textbf{harmonic oscillator}. The quantization of $M$ is the Hilbert space $\hat{M}=L^2(\mathbb{R})$. Furthermore, the self-adjoint operator that quantizes $N$ is
\begin{equation*}
    \hat{N}=-\frac{\hbar^2}{2}\frac{\partial ^2}{\partial^2u}+\frac{u^2}{2}
\end{equation*}
where $\hbar$ is the Planck constant. Notice that the spectrum of $\hat{N}$ is the discrete set $\{\hbar(n+\frac{1}{2})| n\in \mathbb{N}\}$.

In order to obtain a quantization of $S^2$, we consider $S^2$ as a reduced space of $\mathbb{R}^4\cong \mathbb{C}^2$. The quantization of $\mathbb{R}^4$ is the Hilbert space $L^2(\mathbb{R}^2)$.  
Let $L(z_1,z_2):=\frac{|z_1|^2+|z_2|^2}{2}$ be the harmonic oscillator. The quantization of $L$ is the self-adjoint operator 
\begin{align*}
\hat{L}:=-\frac{\hbar^2}{2}\Bigl(\frac{\partial ^2}{\partial x_1^2}+\frac{\partial^2}{\partial x_2^2}\Bigr)+\frac{x_1^2+x_2^2}{2}.
\end{align*}
Furthermore, the spectrum of $\hat{L}$ is the discrete set $\{\hbar (n+1) | n\in \mathbb{N}\}$. 

Let $Y_{2}:=\{L=2\}$ be the Euclidean $3$-sphere of radius $2$ and notice that $S^2$ is the reduced space $Y_2/S^1$, where the quotient map is the following: let $(z_1,z_2)\in Y_2\subset \mathbb{C}^2$, then 
\begin{align*}
    x = \frac{\Re(z_1\overline{z_2})}{2}, \quad  y = \frac{\Im(z_1\overline{z_2})}{2}, \quad 
    z = \frac{|z_1|^2-|z_2|^2}{4}.
\end{align*}
Therefore, the quantization of $S^2$ is the Hilbert space $\mathcal{H}_{S^2}:=\ker(\hat{L}-2)$. Note that
\begin{equation*}
    \dim(\mathcal{H}_{S^2}) =
    \begin{cases}
        n+1, \quad 2=\hbar(n+1),\\
        0, \qquad \quad \text{otherwise}.
    \end{cases}
\end{equation*}
Now we quantize the Cartesian coordinate functions. First, the quantization of $\frac{z_j}{\sqrt{2\hbar}}$ is $a_j:=\frac{1}{\sqrt{2\hbar}}(\hbar \frac{\partial}{\partial x_j}+x_j)$, $j=1,2$. The operators $a_j$ are usually referred to as \textbf{annihilation operators}. Henceforth, an upper $*$ denotes the dual of an operator. The quantization of the Cartesian coordinate functions $x,y,z$ on $S^2$ are the restrictions to $\mathcal{H}_{S^2}$ of the following self-adjoint operators:
\begin{equation*}
\hat{x}:=\frac{\hbar}{2}(a_1a_2^*+a_2a_1^*), \quad \hat{y}:=\frac{\hbar}{2i}(a_1a_2^*-a_2a_1^*), \quad \hat{z}:=\frac{\hbar}{2}(a_1a_1^*-a_2a_2^*).
\end{equation*}
Furthermore, we define the quantization of $z^2$ as $\widehat{(z^2)}:=(\hat{z})^2$. We can summarize the previous discussion in the following definition:
\begin{definition}
\label{n.quantizedoperators}
The quantization of $\mathbb{R}^2\times S^2$ is the (infinite) dimensional Hilbert space $L^2(\mathbb{R})\otimes \mathcal{H}_{S^2}\subset L^2(\mathbb{R})\otimes L^2(\mathbb{R}^2)$. The quantization of $J$ is the operator $\hat{J}=(-\frac{\hbar^2}{2}\frac{\partial^2}{\partial^2u}+\frac{u^2}{2})\otimes Id + Id\otimes \hat{z}$. The quantization of $H$ is the operator $\hat{H}=\frac{1}{2}(u\otimes \hat{x}+(\frac{\hbar}{i}\frac{\partial}{\partial u})\otimes \hat{y})$. The quantization of $G$ is the operator $\hat{G}=Id\otimes \hat{z}^2$. 
\end{definition}
\begin{remark}
The operators $\hat{H}+\hat{G}$ and $\hat{J}$ commute, i.e., $[\hat{H}+\hat{G},\hat{J}]=0$. 
\end{remark}
Now we proceed to compute the joint spectrum of $\hat{H}+\hat{G}$ and of $\hat{J}$ in order to approximate the image of the momentum map of $(J,H+G)$, see Figure \ref{f.jointspectrummodificationJC}.

We need the following auxiliary lemmas:
\begin{lemma}(Pelayo $\&$ \vungoc\ \cite{pelayo2012hamiltonian})
\label{l.B1}
    The operators $A_j:=\frac{1}{\sqrt{2}}(\frac{\partial}{\partial x_j}+x_j)$, $j=1,2$, are unitary conjugated to the annihilation operators $a_j$.
\end{lemma}

\begin{lemma}(Bargmann \cite{bargmann1961hilbert})
\label{l.B2}
    Let $L^2_{\text{hol}}(\mathbb{C}^2,\pi^{-1}e^{-|z|^2})$ be the space of holomorphic functions on two variables with decay $\pi^{-1}e^{-|z|^2}$.
    Consider the operators $\frac{\partial}{\partial z_j}$ and $z_j$ on the Hilbert space $L^2_{\text{hol}}(\mathbb{C}^2,\pi^{-1}e^{-|z|^2})$. Then $A_j$ is unitary equivalent to $\frac{\partial}{\partial z_j}$ and $A_j^*$ is unitary equivalent to $z_j$.
\end{lemma}
From now on $z_j$ refer to functions in $L^2_{\text{hol}}(\mathbb{C}^2,\pi^{-1}e^{-|z|^2})$. We call the space $L^2_{\text{hol}}(\mathbb{C}^2,\pi^{-1}e^{-|z|^2})$ from Lemma \ref{l.B2} the Bargmann space. Furthermore, we note that  monomials $\{z^{\alpha}/\sqrt{\alpha!}\}_{\alpha \in \mathbb{N}^2}$
form a Hilbert basis for the Bargmann space. 

Lemma \ref{l.B1} and Lemma \ref{l.B2} allow us to obtain simpler expressions for the quantizations in Definition \ref{n.quantizedoperators}. For example, in the Bargmann space, $\hat{L}=\hbar (\frac{\partial}{\partial z_1}+\frac{\partial}{\partial z_2}+1)$. Using this, one obtains the following:

\begin{lemma}({Pelayo $\&$ \vungoc, \cite[Lemma 4.4]{pelayo2012hamiltonian}})
Let $\alpha_1,\alpha_2\in \mathbb{N}$. The function $\frac{z_1^{\alpha_1}z_2^{\alpha_2}}{\sqrt{\alpha_1!\alpha_2!}}=:\frac{z^{\alpha}}{\sqrt{\alpha!}}$ is an eigenfunction of $\hat{L}$ with norm $1$ and eigenvalue $\hbar(\alpha_1+\alpha_2+1)$.
\end{lemma}

\begin{lemma}
The space $\mathcal{H}_{S^2}=\ker(\hat{L}-\hbar(n+1))$ is given by 
\begin{equation*}
\mathcal{H}_{S^2}=\text{span}\left\{\left.\frac{z^{\alpha}}{\sqrt{\alpha!}}\right|\alpha_1+\alpha_2=n\right\},
\end{equation*} 
i.e., it is the space of homogeneous polynomials of degree $n$ of $\mathbb{C}^2$. 
\end{lemma}
\begin{proof}
    Recall that the monomials $\{z^{\alpha}/\sqrt{\alpha!}\}_{\alpha \in \mathbb{N}^2}$ form a Hilbert basis of the Bargmann space $L^2_{\text{hol}}(\mathbb{C}^2,\pi^{-1}e^{-|z|^2})$.
\end{proof}

Henceforth we use the following basis for $\mathcal{H}_{S^2}$:
\begin{equation*}
\{z_2^n,z_1z_2^{n-1},...,z_1^{n-1}z_2,z_1^n\}.
\end{equation*}

Let us first understand the operator $\hat{G}$: 
\begin{lemma}
The matrix of $\frac{1}{\gamma}\hat{G}$ relative to the basis $\{z_2^n,z_1z_2^{n-1},...,z_1^{n-1}z_2,z_1^n\}$ of $\mathcal{H}_{S^2}$ is 
\begin{equation*}
\frac{\hbar^2}{4}\begin{bmatrix}
n^2 & 0 & \dots & 0 \\
0 & (n-2)^2 & \dots & 0 \\
\vdots &\vdots &\ddots &\vdots\\
0 & 0 & \dots & n^2 
\end{bmatrix}.
\end{equation*}
\begin{proof}
    The restriction of the operator $\hat{z}=\frac{\hbar}{2}(a_1a_1^*-a_2a_2^*)$ is identified with $\frac{\hbar}{2}(\frac{\partial}{\partial z_1}z_1-\frac{\partial}{\partial z_2}z_2)$ in $\mathcal{H}_{S^2}$ and in terms of this basis we have $\hat{z}(z_1^kz_2^{n-k})=\frac{\hbar}{2}(k-(n-k))z_1^kz_2^{n-k}$. Therefore, $\widehat{(z^2)}(z_1^kz_2^{n-k})=\hat{z}^2(z_1^kz_2^{n-k})=\frac{\hbar^2}{4}(k-(n-k))^2z_1^kz_2^{n-k}$.
\end{proof}
\end{lemma}
Now apply Bargmann's approach to $\hat{N}=\frac{\hat{u}^2+\hat{v}^2}{2}$ acting on $L^2_{Hol}(\mathbb{C},\pi^{-1}e^{-|\tau|^2})$, where the variable in $\mathbb{C}$ is $\tau$. Therefore, $\hat{N}=\hbar(\tau \frac{\partial}{\partial \tau}+\frac{1}{2})$. Using these ideas, one obtains the following:

\begin{lemma}({Pelayo $\&$ \vungoc, \cite[Lemma $4.5$ \& Corollary $4.6$]{pelayo2012hamiltonian}})
\label{l.basisJCM}
Let $\hat{J}$ be the self-adjoint operator defined in \ref{n.quantizedoperators}. Then
\begin{equation*}
\text{spec}(\hat{J})=\hbar \left(\frac{1-n}{2}+\mathbb{N}\right).
\end{equation*}
In particular the spectrum of $\hat{J}$ is discrete.
Moreover, for $\lambda \in \hbar(\frac{1-n}{2}+\mathbb{N})$, let $\mathcal{E}_{\lambda}:=\ker(\hat{J}-\lambda)$. Then 
\begin{equation*}
\mathcal{E}_{\lambda}=\text{span}\left\{ \tau^{l}\otimes z_1^{k}z_2^{n-k}\left| \hbar \left(l+\frac{1}{2}+k-\frac{n}{2}\right)\right .=\lambda;\quad 0\leq k\leq n;\quad l\geq 0\right\}.
\end{equation*}
In particular $\mathcal{E}_{\lambda}$ has dimension $1+\min(n,\frac{\lambda}{\hbar}+\frac{n-1}{2})$. Furthermore,
given any $n\in \mathbb{N}$, and any $\lambda\in \hbar(\frac{1-n}{2}+\mathbb{N})$, the ordered set 
\begin{equation*}
B_{\lambda}:=\left\{\left.e_{l,k}:=\frac{\tau^{l}}{\sqrt{l!}}\otimes \frac{z_1^kz_2^{n-k}}{\sqrt{k!(n-k)!}}\right |k=0,1,....,\min(n,\frac{\lambda}{\hbar}+\frac{n}{2}-\frac{1}{2}), \quad l=\frac{\lambda}{\hbar}+\frac{n}{2}-\frac{1}{2}-k\right\}.
\end{equation*}
is an orthonormal basis of $\mathcal{E}_{\lambda}$.
\end{lemma}

Since $\hat{H}+\hat{G}$ commutes with $\hat{J}$ the spectral theory of $\hat{H}+\hat{G}$ is reduced to the study of the restriction of $\hat{H}+\hat{G}$ to $\mathcal{E}_{\lambda}$, which is understood in the following proposition:
\begin{proposition}
The matrix of the self-adjoint operator $\hat{H}+\hat{G}$ in the basis $B_{\lambda}$, defined in Lemma \ref{l.basisJCM}, is the symmetric matrix:
\begin{equation*}
\begin{bmatrix}
\gamma\frac{(n\hbar)^2}{4} & (\frac{\hbar}{2})^{3/2}\beta_1 & \dots & \dots & \dots & 0 \\
(\frac{\hbar}{2})^{3/2}\beta_1 & \gamma\frac{((n-2)\hbar)^2}{4} & (\frac{\hbar}{2})^{3/2}\beta_2 & \dots & \dots & 0 \\
0 & (\frac{\hbar}{2})^{3/2}\beta_2 & \gamma\frac{((n-4)\hbar)^2}{4} & (\frac{\hbar}{2})^{3/2}\beta_3 & \dots & 0 \\
\vdots &\vdots &\ddots & \ddots & \ddots  & \vdots \\
\vdots &\vdots &\ddots & \ddots & \ddots  &(\frac{\hbar}{2})^{3/2}\beta_{\mu}\\
0 & 0 & \dots & \dots &(\frac{\hbar}{2})^{3/2}\beta_{\mu} & \gamma\frac{(n\hbar)^2}{4} .
\end{bmatrix}
\end{equation*}
where $l_0:=\frac{l}{\hbar}+\frac{n}{2}-\frac{1}{2}$,  $\mu=\min(l_0,n)$ and 
\begin{equation*}
\beta_k:=\sqrt{(l_0+1-k)k(n-k+1)}.
\end{equation*}
\end{proposition}
\begin{proof}
This follows from our computation of $\hat{G}$ and the computation of $\hat{H}$ in Proposition $4.7$ in Pelayo \& \vungoc\ \cite{pelayo2012hamiltonian}. 
\end{proof}
 
\begin{figure}[ht]
\begin{center}
    \includegraphics[scale=1]{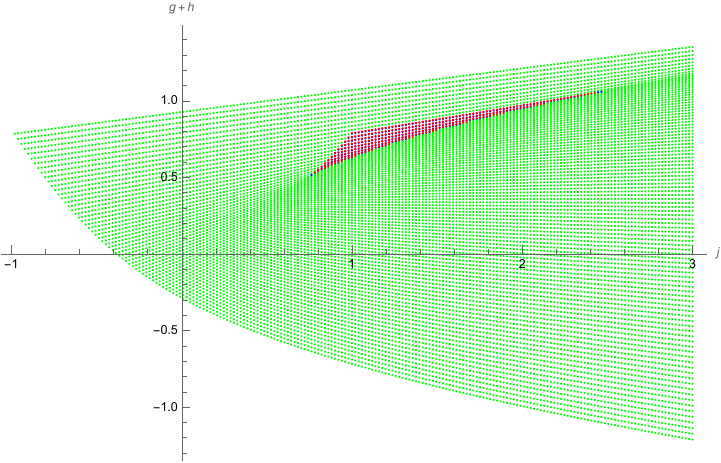}
    \caption{The joint spectrum of $(J,H+G)$ when $\hbar=\frac{2}{101}$. The green dots correspond to values outside of the image of the flap. The red dots correspond to values on the background of the flap. The blue dots correspond to values on the flap, but they are barely visible since they overlap mostly with the red dots.}
    \label{f.jointspectrummodificationJC}
\end{center}
\end{figure}

\subsection{Quantization of the Hirzebruch surface}
\label{s.quantizationhirzebruch} The so called Hirzebruch surface will be one of our leading examples throughout the paper.
Let $\alpha,\beta>0$ and $n\in \mathbb{N}$. Consider the Delzant polytope $\Delta_{\alpha,\beta,n}$ in $\mathbb{R}^2$ given by the vertices $(0,0),(0,\beta),(\alpha,\beta)$ and $(\alpha+n\beta,0)$. We call this polytope the Hirzebruch polytope. Now we associate to the Hirzebruch polytope the symplectic toric manifold $(W_{\alpha,\beta,n},\omega_{\alpha,\beta,
n})$, called the Hirzebruch surface, and a toric integrable system $(W_{\alpha,\beta,n},\omega_{\alpha,\beta,n},F)$ such that $F(W_{\alpha,\beta,n})=\Delta_{\alpha,\beta,n}$. Due to the Delzant construction \cite{delzant1988hamiltoniens}, for details see Cannas da Silva \cite{MR1853077}, $W_{\alpha,\beta,n}$ can be seen as the symplectic reduction of $\mathbb{C}^4$ with respect to the $\mathbb{T}^2$-action generated by 
\begin{equation*}
    N(z_1,z_2,z_3,z_4)=\frac{1}{2}(|z_1|^2+|z_2|^2+n|z_3|^2,|z_3|^2+|z_4|^2)
\end{equation*}
at the level $(\alpha+n\beta,\beta)$. In these coordinates the momentum map is given by $F=\frac{1}{2}(|z_2|^2,|z_3|^2)$. For more details see Le Floch \& Palmer \cite{lefloch:hal-01895250}.
Our goal in this subsection is to quantize this toric system. For this we make use of the description of $W_{\alpha,\beta,n}$ as the symplectic reduction of $\mathbb{C}^4$. The quantization of $\mathbb{C}^4$ is the Hilbert space $L^2(\mathbb{R}^4)$, analogous to Section \ref{s.quantizationflap}, which due to the Bargmann representation, can be identified with the Hilbert space of holomorphic functions on 4 variables $L^2_{hol}(\mathbb{C}^4,\pi^{-1}e^{-|z|^2})$, which is generated by monomials of the form $z_1^{\alpha_1}z_2^{\alpha_2}z_3^{\alpha_3}z_4^{\alpha_4}$ with $\alpha=(\alpha_1,\alpha_2,\alpha_3,\alpha_4)\in \mathbb{N}^{4}$. 

In order to understand the quantization of this toric system, we start by understanding the quantization of $N$. Recall that the harmonic oscillator $\frac{|z_i|^2}{2}$ is quantized as $\hbar(z_i\frac{\partial}{\partial z_i}+\frac{1}{2})$. Since $N$ is given by a linear combination of harmonic oscillators, its quantization, in the Bargmann representation, is given by
\begin{equation*}
    \hat{N}=\hbar\left(z_1\frac{\partial}{\partial z_1}+z_2\frac{\partial}{\partial z_2}+ nz_3\frac{\partial}{\partial z_3}+1+\frac{n}{2}, z_3\frac{\partial}{\partial z_3}+z_4\frac{\partial}{\partial z_4}+1\right).
\end{equation*}
Therefore, the quantization of $W_{\alpha,\beta,n}$ is defined to be the finite dimensional Hilbert space $\mathcal{H}_{\alpha,\beta,n}:=\ker(\hat{N}-(\alpha+n\beta,\beta))$. It is nonzero if $\alpha=\hbar(N_1+\frac{n}{2})$ and $\beta=\hbar(N_2+1)$ for $N_1,N_2\in \mathbb{N}$, which we assume from now on. Therefore, the space $\mathcal{H}_{\alpha,\beta,n}$ is generated by the monomials $z_1^{\alpha_1}z_2^{\alpha_2}z_3^{\alpha_3}z_4^{\alpha_4}$ such that 
\begin{equation}
\label{eq.torichirzebruch}
    \begin{cases}
        \hbar(\alpha_1+\alpha_2+n\alpha_3+\frac{n}{2}+1)=\alpha+n\beta,\\
        \hbar(\alpha_3+\alpha_4+1)=\beta.
    \end{cases}
\end{equation}
Now the quantization of $J(z_1,z_2,z_3,z_4)=\frac{1}{2}|z_2|^2$ is $\hat{J}=\hbar(z_2\frac{\partial}{\partial z_2}+\frac{1}{2})$ restricted to $\mathcal{H}_{\alpha,\beta,n}$ and the quantization of $R(z_1,z_2,z_3,z_4)=\frac{1}{2}|z_3|^2$ is $\hat{R}=\hbar(z_3\frac{\partial}{\partial z_3}+\frac{1}{2})$ restricted to $\mathcal{H}_{\alpha,\beta,n}$. Both operators commute. 
\begin{remark}
    The function $z_1^{\alpha_1}z_2^{\alpha_2}z_3^{\alpha_3}z_4^{\alpha_4}$ is an eigenfunction of $\hat{J}$ with eigenvalue $\hbar(\alpha_2+\frac{1}{2})$. 
\end{remark}

\begin{lemma}
\label{l.conditionspectrumJ}
    The spectrum of $\hat{J}$ is discrete and we have 
    \begin{equation*}
        \text{spec}(\hat{J})=\left\{\hbar \left. \left(k+\frac{1}{2}\right)\right | 0\leq k\leq \lfloor (\alpha+n\beta)/\hbar-1-n/2 \rfloor , k\in \mathbb{N} \right\}.
    \end{equation*}
    For a fixed value $a_2$ let $\mathcal{E}_{a_2}:=\ker (\hat{J}-a_2)$. Then
    \begin{itemize}
        \item If $n=0$ and $\alpha-a_2\geq \frac{\hbar}{2}$ then $\dim(\mathcal{E}_{a_2})=\lfloor \beta/\hbar\rfloor$.
        \item If $n\neq 0$ and $a_2\leq \alpha$ then $\dim(\mathcal{E}_{a_2})=\lfloor\beta/\hbar\rfloor$.
        \item Otherwise $\dim(\mathcal{E}_{a_2})=\lceil \frac{\beta}{\hbar}-1-\frac{1}{\hbar}\frac{1}{n}(a_2-\alpha)-\frac{(n-1)}{2n} \rceil +1.$
    \end{itemize}
\end{lemma}
\begin{proof}
    Recall that $\alpha_i\in \mathbb{N}$, in particular $k\in \mathbb{N}$.
   From Equation \eqref{eq.torichirzebruch} we obtain
   \begin{equation*}
       \alpha_1+\alpha_2+n\alpha_3=\frac{1}{\hbar}(\alpha+n\beta)-1-\frac{n}{2}
   \end{equation*}
   which give us an upper bound $\lfloor (\alpha+n\beta)/\hbar-1-n/2 \rfloor$ for $\alpha_2$.
   Now consider a fixed $a_2$. Recall that the basis of $\mathcal{E}_{a_2}$ is generated by the monomials $z_1^{\alpha_1}z_2^{\alpha_2}z_3^{\alpha_3}z_4^{\alpha_4}$ satisfying Equation \eqref{eq.torichirzebruch}. 
   %Note that for the basis to be orthonormal one needs to divide the monomials $z_1^{\alpha_1}z_2^{\alpha_2}z_3^{\alpha_3}z_4^{\alpha_4}$ by the corresponding coefficient  $\sqrt{\alpha!}:=\sqrt{(\alpha_1!)(\alpha_2!)(\alpha_3!)(\alpha_4!)}$.
   From Equation \eqref{eq.torichirzebruch} we obtain
   \begin{equation*}
       \alpha_1=\frac{1}{\hbar}\alpha-\alpha_2-1+\frac{n}{2}+n\alpha_4
   \end{equation*}
and using that $\alpha_2=\frac{a_2}{\hbar}-\frac{1}{2}$ we obtain
\begin{equation*}
    \alpha_1=\frac{1}{\hbar}(\alpha-a_2)+\frac{(n-1)}{2}+n\alpha_4.
\end{equation*}
Recall that $\alpha_4\in \mathbb{N}$ with values between $0$ and $\frac{\beta}{\hbar}-1$. Since $\alpha_1$ needs to be a nonnegative natural number we obtain the desired results. 
   
\end{proof}

\begin{remark}
    The function $z_1^{\alpha_1}z_2^{\alpha_2}z_3^{\alpha_3}z_4^{\alpha_4}$ is an eigenfunction of $\hat{R}$ with eigenvalue $\hbar(\alpha_3+\frac{1}{2})$.
\end{remark}
Therefore, the joint spectrum of $\hat{J}$ and $\hat{R}$ in $\mathcal{H}_{\alpha,\beta,n}$ is the pair $\hbar(\alpha_2+\frac{1}{2},\alpha_3+\frac{1}{2})$ such that 
\begin{equation*}
    \begin{cases}
        \hbar(\alpha_1+\alpha_2+n\alpha_3+\frac{n}{2}+1)=\alpha+n\beta,\\
        \hbar(\alpha_3+\alpha_4+1)=\beta.
    \end{cases}
\end{equation*}
where $\alpha_2$ ranges from $0$ to $(\alpha+n\beta)/\hbar-1-n/2$, and $\alpha_i\in \mathbb{N}$ for $i=1,...,4$. See Figure \ref{f.hirzebruchquantized} for an example.

%These properties of $\hat{J}$ and its eigenspaces are used in Appendixes \ref{s.quantizationflap2ellitptic}, \ref{s.quantizationpleat}, \ref{s.quantizationcurledtori}. 
\begin{figure}[ht]
\begin{center}
    \includegraphics[]{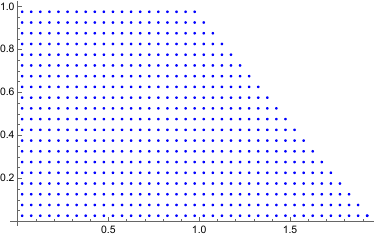}
    \caption{The joint spectrum of $(J=\frac{|z_2|^2}{2},R=\frac{|z_3|^2}{2})$ for $\hbar=0.05$ and $\alpha=\beta=n=1$.
}
\label{f.hirzebruchquantized}
\end{center}
\end{figure}

\section{Affine invariant in the presence of a standard flap}
\label{s.flap}
\subsection{Intuition}
%For toric systems, Delzant \cite{delzant1988hamiltoniens} showed that the image of the moment map is a classifying invariant modulo elements of $\AGL(n,\mathbb{Z})$. In \cite{vu2007moment} \vungoc\ generalized this invariant for semitoric systems where he had to deal with focus-focus points and the monodromy introduced by them. The solution was to preserve the $S^1$-action and to make the set of regular values simply connected. This was done by introducing certain vertical cuts above the focus-focus values. This leads to a family of polytopes, due to the choice of cut one has for each focus-focus value. For more details on this we refer the reader to Section \ref{s.semitoric}.

Recall that when a system admits a \textbf{standard} flap $\mathcal{F}$, see Definition~\ref{d.standardflap}, the set of regular values is not simply connected and the system displays monodromy due to the presence of elliptic-elliptic values in $\operatorname{Im}(\mathcal{F})$, see Proposition \ref{p.monodromy}. There are two ways to deal with this phenomenon. We can either make a cut for each flap, or we can make a cut for each elliptic-elliptic value present in the image of the flap. We explore the two ideas. In Section \ref{s.polytopenflaps} we make a cut for each flap. In Section \ref{s.polytopeflap} we make a cut for each elliptic-elliptic value present on the image of a flap. Each approach has its advantages and disadvantages. Doing a cut for each flap may create more discontinuities than the other approach, but on the other hand the resulting affine invariant is smoother, i.e., has fewer corners. 
For an example of the first approach see Figures \ref{p.iafsflap} and Figure \ref{f.polytopeinvariantsflapwith2elliptic}. For an example of the second approach see Figure \ref{f.polytopeinvariantsflapwith2elliptic}. 

If a hypersemitoric system admits a flap with no critical values of rank $0$ in its interior and no elliptic-elliptic values in its boundary, then there is no topological monodromy and hence no choice of cut necessary in order to obtain an affine invariant. We address this phenomenon in Section \ref{s.generalhypersemitoric}.

For simplicity of notation, from Section \ref{s.polytopenflaps} to Section \ref{s.polytopeflap}, and corresponding subsections, we assume that all flaps $\mathcal{F}$ are \textbf{bounded}, see Definition \ref{d.bounded}. We note that the adaption to the general case is straightforward, and is dealt with in Section \ref{s.subsectiongeneralinvariant}.

%\subsection{Affine monodromy}
%Recall from Section \ref{s.monodromy} that the topological monodromy associated with a flap: let $\gamma$ be a loop that circles around the flap and let ${e_1,e_2}$ be a basis of $H_1(F^{-1}(\gamma(0)))$ such that $e_1$ is an orbit of the $S^1$-action, then the topological monodromy of $\gamma$ is given by
%\begin{equation*}
%M=
%\begin{bmatrix}
 %1 & \sum_{i=1}^{k}\frac{1}{m_i n_i}\\
% 0 & 1 
%\end{bmatrix}
%\end{equation*}
%where $m_i,n_i$ are the isotropy weights of the fixed points $p_1,...,p_k$ in the flap. Therefore due to Lemma \ref{l.af} the affine monodromy of $\gamma$ is given by 
%\begin{equation}
%\label{eq.monodromymatrix}
 %M^{-T}= \begin{bmatrix}
%          1 & 0 \\
%          -\sum_{i=1}^{k} \frac{1}{m_i n_i} & 1
%         \end{bmatrix}.
%\end{equation}
%Note that, by Lemma \ref{l.weights}, $-\sum_{i=1}^{k}\frac{1}{m_in_i}>0$ and it equals the number of elliptic-elliptic points in the flap.

 \subsection{Affine invariant with one cut for each flap}
 \label{s.polytopenflaps}
Let $(M,\omega,F=(J,H))$ be a hypersemitoric system 
on a $4$-dimensional symplectic manifold such that the only critical values outside of the elliptic critical values occurring at the boundary of $F(M)$ are caused by \textbf{standard flaps} $\mathcal{F}_i,\dots,\mathcal{F}_n$. By Corollary $\ref{c.flapsfinitenflaps},$ it follows that $n<\infty$. 

Define projections $\pi_x:\mathbb{R}^2\rightarrow \mathbb{R}$, $(x,y)\mapsto x$ and $\pi_y:\mathbb{R}^2\rightarrow \mathbb{R}$, $(x,y)\mapsto y$. For each standard flap $\mathcal{F}_i$, with $n_i\in \mathbb{N}$, let $d_{i,j}, 1\leq j\leq n_i$, be the elliptic-elliptic values occurring on $\operatorname{Im}(\mathcal{F}_i)$ ordered as follows $\pi_x(d_{i,1})<\pi_x(d_{i,2})<\dots<\pi_x(d_{i,n_i})$.
Define $c_i$ as the hyperbolic-regular value in $\gamma_{\mathcal{F}_i}$ such that $\pi_x(c_i)=\pi_x(d_{i,1})$, $i=1,\dots,n$.
%If $\pi_x(\mathcal{F}_i)\cap \pi_x(\mathcal{F}_j)=\emptyset \ \forall i,j=1,...,n$ then we don't modify the $c_i$. If $\pi_x(\mathcal{F}_i)\subset \pi_x(\mathcal{F}_j)$ set $c_j:=c_i$. If $\pi_x(\mathcal{F}_i)\cap \pi_x(\mathcal{F}_j) \neq \emptyset$ but $\pi_x(\mathcal{F}_i)\not\subset \pi_x(\mathcal{F}_j)$ and $\pi_x(\mathcal{F}_j)\not\subset \pi_x(\mathcal{F}_i)$ then we define $c_i$ and $c_j$ as points such that the vertical line trough $c_i$ does not intersect $\mathcal{F}_j$ and the vertical trough $c_j$ does not intersect $c_i$.

Let $\vec{\epsilon}=(\epsilon_1,...,\epsilon_n)$ with $\epsilon_i\in \{\pm 1\}$ for $i=1,...,n$. For each $\epsilon\in \{-1,1\}$, let $\mathcal{L}^{\epsilon}_{c_i}$ be the vertical half line starting at $c_i$ and going to $\pm \infty$ according to the sign of $\epsilon$, i.e., $\mathcal{L}_{c_i}^{\epsilon}=\{(x_i,y)| \epsilon y \geq \epsilon y_i\}$ where $c_i=(x_i,y_i)$. Let $l_{c_i}^{\epsilon}:=F(M)\cap \mathcal{L}^{\epsilon}_{c_i}$. Denote by $\pi_j:\{-1,1\}^{n}\rightarrow \{-1,1\}$ the projection onto the $j$-th coordinate. Furthermore, let $\mathcal{T}$ be the subgroup of $\AGL(2,\mathbb{Z})$ which leaves a vertical line, with orientation, invariant. 

Then we have the following result which parallels Theorem~\ref{th.straigheningHomeo} on the classification of semitoric systems.

\begin{theorem}
\label{t.polytopenflaps}
Using the notation from above, let $(M,\omega,F)$ be a hypersemitoric system where the only critical values apart from the elliptic values in the boundary of $F(M)$ occur at the images of \textbf{standard} flaps $\mathcal{F}_1,...,\mathcal{F}_n$. Given any $\vec{\epsilon}\in \{-1,1\}^{n}$, there exists a continuous map $f_{\vec{\epsilon}}: F(M)\backslash(\cup_{i=1}^{n}( \gamma_{\mathcal{F}_i}\cup l^{\pi_i(\vec{\epsilon})}_{c_i}))\rightarrow \mathbb{R}^2$, such that:
\begin{enumerate}
\item $f_{\vec{\epsilon}}$ is a diffeomorphism onto its image; 
\item $f_{\vec{\epsilon}}$ is affine;
\item $f_{\vec{\epsilon}}$ preserves $J$, i.e., is of the form $f_{\vec{\epsilon}}(x,y)=(x,f_{\vec{\epsilon}}^{(2)}(x,y))$;
\item
For any $c\in \text{int}\left(l^{\pi_i(\vec{\epsilon})}_{c_i}\right)$
\begin{equation}
\label{eq.monodromyrelationnflaps}
\lim_{\begin{smallmatrix}(x,y)\rightarrow c \\ x<\pi_x(c_i) \end{smallmatrix}} df_{\vec{\epsilon}}(x,y) = M_{c_i} \lim_{\begin{smallmatrix}(x,y)\rightarrow c \\ x>\pi_x(c_i) \end{smallmatrix}} df_{\vec{\epsilon}}(x,y),
  \end{equation}
  where 
  \begin{equation*}
 M_{c_i}= \begin{bmatrix}
          1 & 0 \\
          k(c_i) & 1
         \end{bmatrix},
\end{equation*}
and $k(c_i)=\sum_{j}\pi_j(\vec{\epsilon})n_j$ where the sum runs over all the $c_j$ such that $c_i\in l_{c_j}^{\pi_j(\vec{\epsilon})}$. %We exclude points with $k(c_i)=0$. 
Furthermore, if $n_i=1$ and $\pi_x(\operatorname{Im}(\mathcal{F}_i))\cap \pi_x(\operatorname{Im}(\mathcal{F}_j))=\emptyset$ for all $j\neq i$, the map $f_{\vec{\epsilon}}$ extends to a continuous map over
\begin{equation*}
F(M)\backslash \left(\bigcup_{j=1}^{n_i-1}l_{c_j}^{\pi_j(\vec{\epsilon})}\cup \bigcup_{j=n_i+1}^{n}l_{c_j}^{\pi_j(\vec{\epsilon})}\cup \bigcup_{j=1}^{n}\gamma_{\mathcal{F}_j}\right).
\end{equation*}
\item The image of $f_{\vec{\epsilon}}$ is a rational polytope with ``holes'' in its interior.
\end{enumerate} 
Such an $f_{\vec{\epsilon}}$ is unique modulo left composition by $\mathcal{T}$. Furthermore, for each $i=1,...,n$ there exists a map $g_i:\operatorname{Im}(\mathcal{F}_i)\backslash \gamma_{\mathcal{F}_i}\rightarrow \mathbb{R}^2$ that is affine, $J$-preserving, and a diffeomorphism onto its image. The maps $g_i$ do not depend on $\vec{\epsilon}$ and are unique modulo left composition by $\mathcal{T}$.

%For examples see Figure \ref{p.iafsflap} and Figures \ref{f.actionangle2points1cut}, \ref{f.flappypart}.

\end{theorem}
\begin{proof}
    Let $\tilde{F}$ be the map associated with the unfolded momentum domain $A$ of $F$, i.e., there exists $A$ and $\pi:A\rightarrow F(M)$ such that $\pi:A\rightarrow F(M)$ and $\pi \circ \tilde{F}=F$. We work on the background of the flaps $\mathcal{F}_i$ where the fibers of the map $\tilde{F}$ are connected.
    
    Let $B_r$ be the set of regular values in the background of the system $\mathcal{B}$.  By abuse of notation we identify each value $c_i$ with the corresponding point in the background $\mathcal{B}$.
    We first consider the case where $B_r\backslash \cup_{i=1}^{n}l_{c_i}^{\pi_i(\vec{\epsilon})}$ is connected.
    
    Consider the local system $\xi\rightarrow B_r$ whose fiber over $b$ is $H_1(\tilde{F}^{-1}(b),\mathbb{Z})\cong \mathbb{Z}_n$. Let $p:\tilde{B}\rightarrow B_r$ be the universal cover and let $\tilde{\xi}:=p^*(\xi)$. Since $\tilde{B}$ is simply connected, $\tilde{\xi}$ is trivial. Let $c_1,c_2$ be a $\mathbb{Z}$-basis of continuous sections of $\tilde{\xi}\rightarrow \tilde{B}$. Using $c_1,c_2$ we define action coordinates for $\tilde{b}\in \tilde{B}$ as
    \begin{equation}
        \label{eq.ACuniversal}
        G(\tilde{b}):=\left(\frac{1}{2\pi}\int_{c_1(\tilde{b})}\lambda,\ \frac{1}{2\pi}\int_{c_2(\tilde{b})}\lambda\right),
    \end{equation}
where $\lambda$ is a $1$-form such that $d\lambda =-\omega$. 

Now we use the map $G$ to define the map $f_{\vec{\epsilon}}$.
Fix a point $q_0\in B_r$. For each $p\in B_r$ let $\gamma(q_0,p)$ be a path that connects $q_0$ to $p$ in $B_r \backslash \cup_{i=1}^{n}l_{c_i}^{\pi_i(\vec{\epsilon})}$. Define $f_{\vec{\epsilon}}$ as 
\begin{equation}
    \label{definefin1corner}
    f_{\vec{\epsilon}}(p):=G(\gamma(q_0,p)).
\end{equation}
Since $J$ defines an effective global $S^1$-action we may assume $f_{\vec{\epsilon}}(x,y)=(x,f_{\vec{\epsilon}}^{(2)}(x,y))$. Properties $(1)$ and $(2)$ in the Theorem follow from the fact that the map $f_{\vec{\epsilon}}$ is defined using action coordinates.
Using Proposition \ref{p.monodromy} we obtain \eqref{eq.monodromyrelationnflaps}, as in \vungoc \ \cite[Theorem $3.8$]{vu2007moment}. 
 
Over the lines $l_{c_i}^{\pi_i(\vec{\epsilon})}$ where we make the cuts, if $n_i=1$ and $\pi_x(\operatorname{Im}(\mathcal{F}_i))\cap \pi_x(\operatorname{Im}(\mathcal{F}_j))=\emptyset$ for all $j\neq i$, the map can be continuously extended over the line $l^{\pi_i(\vec{\epsilon})}_{c_i}$. This is due to Equation \eqref{eq.monodromyrelationnflaps} and the fact that the height of the resulting polytope near $c_i$ equals the density function for the Duistermaat-Heckman measure, i.e., the volume of the reduced orbifold $J^{-1}(x)/S^1$, for a formula see Karshon \cite{karshon1999periodic} or Section \ref{s.duistermaatmeasuregeneral}.
If $n_i\neq 1$ then the map may not extend continuously, for example see Figure \ref{f.polytopeinvariantsflapwith2elliptic}. 

Using the local normal form for elliptic-elliptic values and the fact that the monodromy of the system around each $\gamma_{\mathcal{F}_i}$ is $\begin{bsmallmatrix}1 & n_i\\0 & 1\end{bsmallmatrix}$, where $n_i$ is the number of elliptic-elliptic values in the flap $\mathcal{F}_i$, we may use the same argument as \vungoc\ \cite[Theorem $3.8$]{vu2007moment} to conclude that the image of $f_{\vec{\epsilon}}$ is a rational polytope with flap-shaled holes in the middle. The discontinuity of the map that gives the flap-shaped holes is due to the fact that the fibers of $F$ inside of the flaps have two connected components, while on the outside the fibers are connected.  

Let us now consider the case where $B_r\backslash \cup_{i=1}^{n}l_{c_i}^{\pi_i(\vec{\epsilon})}$ is not connected. This can happen in multiple situations, specifically when there exist flaps $\mathcal{F}_i, \mathcal{F}_j$ such that $\pi_x(\operatorname{Im}(\mathcal{F}_i))\cap \pi_x(\operatorname{Im}(\mathcal{F}_j))\neq \emptyset$.
 
The idea is to use paths in $B_r$ that obey the following rule. Fix $q_0\in B_r$. Intuitively, if $\epsilon_i=-1$ then the path between $q_0$ and $p$ passes above $\gamma_{\mathcal{F}_i}$ if $\pi_x(p)>\pi_x(c_i)$ and it goes below $\gamma_{\mathcal{F}_i}$ otherwise. If $\epsilon_i=1$ the path passes below $\gamma_{\mathcal{F}_i}$ if $\pi_x(p)>\pi_x(c_i)$ and above $\gamma_{\mathcal{F}_i}$ otherwise. See Figure \ref{p.ruleofpaths} for an example. More precisely,
replace the points $c_i$ by nearby points $\tilde{c_i}$, and if necessary $\gamma_{\mathcal{F}_i}$ by $\tilde{\gamma}_{\mathcal{F}_i}$, such that $B\backslash \cup_{i=1}^{n}l_{\tilde{c_i}}^{\pi_i(\vec{\epsilon})}$is connected. If several $c_i$'s have different $x$-coordinate this replacement is done in a way that preserves the ordering with respect to the $x$-coordinate. Therefore, up to homotopy, there exists a unique path $\gamma(q_0,p)$ joining $q_0$ to $p$ inside $B_r$ avoiding all the $\tilde{\gamma}_{\mathcal{F_i}}, i=1,...,n$ such that the whole picture is isotopic to one in $B_r\backslash \cup_{i=1}^{n}l_{\tilde{c_i}}^{\pi_i(\vec{\epsilon})}$. The homotopy class of $\gamma(q_0,p)$ depends on the choice of ordering of the $x$-coordinates of the points $\tilde{c}_i$. However the value given by equation \eqref{definefin1corner} is well-defined. This is due to the fact that the monodromy is Abelian. The properties of the map follows analogously to the previous case. For more details see Pelayo \& Ratiu \& \vungoc\ \cite[Theorem $B$] {pelayo2017affine} and \vungoc\ \cite[Theorem $3.8$]{vu2007moment}.

The maps $g_i$ are defined using action coordinates on each flap $\mathcal{F}_i$. This is possible since the set of regular values on each flap $\mathcal{F}_i$ is simply connected.
\end{proof}
\begin{figure}[ht]
\centering
    \includegraphics[scale=0.7]{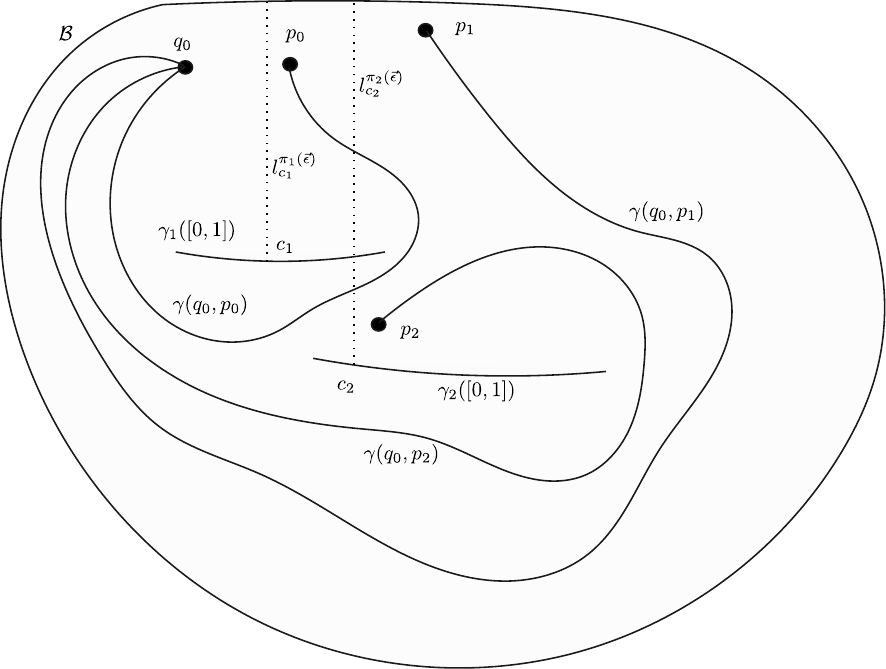}
    \caption{Rule that the paths used in Theorem \ref{t.polytopenflaps} need to follow in the case $n=2$ and $\vec{\epsilon}=(1,1)$.}
\label{p.ruleofpaths}
\end{figure}
\begin{definition}
    We call the image of $f_{\vec{\epsilon}}$ together with the images of $g_i$ from Theorem \ref{t.polytopenflaps} a \textit{representative of the affine invariant} of $(M,\omega,F)$. The \textit{affine invariant} is the collection of all the representatives for all the choices of $\vec{\epsilon}$. For an alternative definition, see Section \ref{s.grouporbitnflaps}.
\end{definition}

\subsection{Group orbit of the affine invariant in the case of one cut for each flap}
\label{s.grouporbitnflaps}
As in Section \ref{s.polytopenflaps} let $(M,\omega,F=(J,H))$ be a hypersemitoric system on a $4$-dimensional symplectic manifold $(M,\omega)$ such that the only critical values outside of the elliptic critical values in the boundary of $F(M)$ are caused by \textbf{standard} flaps $\mathcal{F}_i, i=1,...,n$. Let $c_i$ be the points defined in Section \ref{s.polytopenflaps}. Recall that by Corollary \ref{c.flapsfinitenflaps}, it follows that $n<\infty$. For $\vec{\epsilon}\in \{-1,1\}^{n}$, Theorem \ref{t.polytopenflaps} gives us a polytope $\Delta_{\vec{\epsilon}}:=\operatorname{Im}(f_{\vec{\epsilon}})$ which has holes in its interior, where $f_{\vec{\epsilon}}$ is the map given in Theorem \ref{t.polytopenflaps}. In this Section we investigate, analogous to \vungoc\ \cite[Section $4$]{vu2007moment}, the relations between these polytopes when we vary $\vec{\epsilon}$.

Let $\mathcal{L}$ be a vertical line in $\mathbb{R}^2$. Then $\mathcal{L}$ splits $\mathbb{R}^2$ into two half-spaces, $A_1$ to the left of $\mathcal{L}$ and $A_2$ to the right. Define the piecewise affine transformation $t^{k}_{\mathcal{L}}$ acting as the identity on $A_1$  and as the matrix $T_{k}:=\begin{bsmallmatrix}1 & 0\\ k & 1\end{bsmallmatrix}$ on $A_2$, where $k\in \mathbb{Z}$. Note that $T$ fixes $\mathcal{L}$. 

Consider now the vertical lines $\mathcal{L}_i, i=1,...,n$ through the points $c_i$. For any $\vec{k}:=(k_1,...,k_n)\in \mathbb{Z}^{n}$ we define the piecewise affine transformation of $\mathbb{R}^2$, $t^{\vec{k}}:=t^{k_1}_{\mathcal{L}_i}\circ ... \circ t^{k_n}_{\mathcal{L}_n}$.
Furthermore, let $G=\{0,1\}^n$, with the structure of the Abelian group $(\mathbb{Z}_2)^{n}$ and let $\vec{n}:=(n_1,...,n_n)\in \mathbb{Z}^n$, where $n_i\in \mathbb{N}$ is the number of elliptic-elliptic values in the flap $\mathcal{F}_i$. 

\begin{proposition}
    \label{p.groupofpolytopes1cutforeachflap}
    Let $G$ act transitively on the set
    \begin{equation*}
        \Delta:=\{\Delta_{\vec{\epsilon}}\mid \vec{\epsilon}\in \{-1,1\}^n \}
    \end{equation*}
    by the formula
    \begin{equation}
    \label{eq.Gactionnflaps}
        G\times \Delta \rightarrow \Delta,\quad (\vec{u},\Delta_{\vec{\epsilon}}) \mapsto \vec{u}\cdot \Delta_{\vec{\epsilon}}:=\Delta_{(-2\vec{u}+1)\cdot \vec{\epsilon}}
    \end{equation}
    where scalar multiplication and addition is component wise, and the dot $\cdot$ between $\vec{\epsilon}$ and $\vec{u}$ means component wise multiplication. Then the action satisfies:
    \begin{equation*}
    %\label{eq.taction}
        \vec{u}\cdot \Delta_{\vec{\epsilon}}=t^{\vec{u}\cdot \vec{\epsilon}\cdot \vec{n}}\Delta_{\vec{\epsilon}}
    \end{equation*}
    where the dot $\cdot$ between $\vec{\epsilon}, \vec{u}$ and $\vec{n}$ means point wise multiplication. 
    The action is free if and only if the $x$-coordinates of the elliptic-elliptic critical values %\ke{why the EE values and not the values of the cuts?} \ps{They are the same for a standard flap, at least ignoring the max and min} are pairwise distinct. 
\end{proposition}
\begin{proof}
    The proof is analogous to the proof of \vungoc\ \cite[Proposition $4.1$]{vu2007moment}, using Equation \eqref{eq.monodromyrelationnflaps}. 

    The action of $G$ commutes with $\mathcal{T}$ and therefore defines an action on the equivalence classes modulo $\mathcal{T}$. Therefore, it suffices to focus on generators of $G$.

     We use here the notations of Theorem \ref{t.polytopenflaps}. We focus on the case $n=1$ for simplicity, the case of general $n$ follows from a similar argument. Let $M_1:=\tilde{F}^{-1}(\mathcal{B}\backslash \gamma_1([0,1]))\cap J^{-1}(]-\infty,\pi_x(c_1)])$. Furthermore, let $M_2:=\tilde{F}^{-1}(\mathcal{B}\backslash \gamma_1([0,1]))\cap J^{-1}([\pi_x(c_1),+\infty[)$.

    Selecting an element of the class $\Delta_{\vec{\epsilon}}$ amounts to fixing the starting local basis of the action variables $f_1$ in $M_1$. Another representative of that class is obtained by composing $f_1$ by a transformation in $\mathcal{T}$. Therefore, in what follows we fix $f_1$ and by the notation $\Delta_{\vec{\epsilon}}$ we always mean a particular representative obtained from $f_1$ by the process of Theorem \ref{t.polytopenflaps}. 

    Consider the action of $G$ given by \eqref{eq.Gactionnflaps}. Consider the element $1\in G$ and let it act on the polytope $\Delta_{1}$ associated to the element $1$ obtaining  $1\cdot \Delta_1=\Delta_{-1}$. Let $f_2$ and $\tilde{f_2}$ be the local action variables in $M_2$ obtained for $\Delta_1$ and $\Delta_{-1}$ respectively. Let us fix $y>y_1$, where  $c_1=(x_1,y_1)$. Using Equation \eqref{eq.monodromyrelationnflaps} we have for $\Delta_{1}$
    \begin{equation*}
        \lim_{\begin{smallmatrix}(x,y)\rightarrow c \\ x<\pi_x(c_1) \end{smallmatrix}} df_{1}(x,y) = \begin{bmatrix}
            1 & 0 \\
            n_1 & 1 
        \end{bmatrix} \lim_{\begin{smallmatrix}(x,y)\rightarrow c \\ x>\pi_x(c_1) \end{smallmatrix}} df_{2}(x,y)
    \end{equation*}
    whereas for $\Delta_{-1}$ the formula reads 
    \begin{equation*}
        \lim_{\begin{smallmatrix}(x,y)\rightarrow c \\ x<\pi_x(c_1) \end{smallmatrix}} df_{1}(x,y) = \begin{bmatrix}
            1 & 0 \\
            0 & 1 
        \end{bmatrix} \lim_{\begin{smallmatrix}(x,y)\rightarrow c \\ x>\pi_x(c_1) \end{smallmatrix}} d\tilde{f}_{2}(x,y)
    \end{equation*}
    entailing 
    \begin{equation*}
        \lim_{\begin{smallmatrix}(x,y)\rightarrow c \\ x>\pi_x(c_1) \end{smallmatrix}} d\tilde{f}_{2}(x,y) = \begin{bmatrix}
            1 & 0 \\
            n_1 & 1 
        \end{bmatrix} \lim_{\begin{smallmatrix}(x,y)\rightarrow c \\ x>\pi_x(c_1) \end{smallmatrix}} df_{2}(x,y)
    \end{equation*}
    and therefore, since in $M_2$ the maps $f_2$ and $\tilde{f_2}$ must differ only by an element of $\mathcal{T}$ we obtain 
    \begin{equation*}
        \tilde{f_2}=\begin{bmatrix}
            1 & 0 \\
            n_1 & 1 
        \end{bmatrix}
        \circ f_2.
    \end{equation*}
The case for $\epsilon=-1$ follows analogously. 
    
\end{proof}

\subsection{The Duistermaat-Heckman measure associated with the \texorpdfstring{$S^1$}{S1}-action in the case of one cut for each flap}
\label{s.measure1cutflap}
As in Section \ref{s.polytopenflaps} let $(M,\omega,F=(J,H))$ be a hypersemitoric system on a $4$-dimensional symplectic manifold $(M,\omega)$ such that the only critical values outside of the elliptic critical values occurring at the boundary of $F(M)$ are caused by standard flaps $\mathcal{F}_i, i=1,...,n$.
%Let $\gamma_i:[0,1]\rightarrow \mathcal{F}_i$ be a curve such that $\gamma_i(t),\ t\in \ ]0,1[$ parametrizes the hyperbolic-regular values in $\mathcal{F}_i$ and $\gamma_i(0),\gamma_i(1)$ are the parabolic values in $\mathcal{F}_i$. 
The main goal of this section is to prove Corollary \ref{c.flapsfinitenflaps}, from where it follows that $n<\infty$.

Analogous to \vungoc\ \cite[Section $5$]{vu2007moment}, the straightened momentum images $\Delta_{\vec{\epsilon}}:=\operatorname{Im}(f_{\vec{\epsilon}})$ given by the maps $f_{\vec{\epsilon}}$ of Theorem \ref{t.polytopenflaps} are an efficient tool for recovering the various invariants associated with the momentum map of the hypersemitoric system $(M,\omega,F=(J,H))$, in particular of the effective $S^1$-action induced by $J$. 

Recall the Duistermaat-Heckman measure $\mu_J$ and its density function $\rho_J$ from Section \ref{s.duistermaatmeasuregeneral}. We will see that away from the projection of the image of the flaps we recover the function $\rho_J$ from $\Delta_{\vec{\epsilon}}$.

\begin{proposition}
\label{p.lengthnflaps}
    Given any $\vec{\epsilon}\in \{-1,1\}^n$ consider $\Delta_{\vec{\epsilon}}=\operatorname{Im}(f_{\vec{\epsilon}})$, where $f_{\vec{\epsilon}}$ is the map of Theorem \ref{t.polytopenflaps}. Then $\rho_J(x)$, for $x\in J(M)$ such that $x\notin \pi_x(\operatorname{Im}(\mathcal{F}_i)),$ for all $i=1,...,n$, is equal to the length of the vertical segment of $\Delta_{\vec{\epsilon}}$ at $x$, i.e, $\Delta_{\vec{\epsilon},x,max}-\Delta_{\vec{\epsilon},x,min}$, where $\Delta_{\vec{\epsilon},x,max}$ (resp.\ $\Delta_{\vec{\epsilon},x,min}$) is the largest (resp.\ smallest) value of $\Delta_{\vec\epsilon}$ on the vertical line induced by $x$.
\end{proposition}
\begin{proof}
    Let $f_{\vec{\epsilon}}$ be the homeomorphism given by Theorem \ref{t.polytopenflaps}.
    In points $(x,y)$ such that $x \notin \pi_x(\operatorname{Im}(\mathcal{F}_i)),$ for all $i=1,...,n$, the map is a set of smooth action variables. Furthermore, the fibers $F^{-1}(x,y)$ are connected. According to the Arnold-Liouville-Mineur Theorem \ref{t.AL}, the symplectic form is given by $\frac{1}{2\pi}dx\wedge d\theta_1+\frac{1}{2\pi}dy\wedge d\theta_2$ and $J^{-1}(x)$ minus two circles is given by $]\Delta_{\vec{\epsilon},x,min},\Delta_{\vec{\epsilon},x,max}[\ \times \mathbb{T}^2$. Thus \begin{equation*}\mu_J([a,b])=\int_{0}^{2\pi}\int_0^{2\pi}\int_a^b\int_{\Delta_{\vec{\epsilon},x,min}}^{\Delta_{\vec{\epsilon},x,max}}\frac{1}{4\pi^2} dy\ dx\ d\theta_1 \ d\theta_2=\int_{a}^{b}(
    \Delta_{\vec{\epsilon},x,max}-\Delta_{\vec{\epsilon},x,min})dx
    \end{equation*}
    and the result follows. 
\end{proof}
Let $\vec{\epsilon}\in \{-1,1\}^n$ and $f_{\vec{\epsilon}}$ be the map given by Theorem \ref{t.polytopeflap}. Denote by $\Sigma_0(F)$ the set of critical values of $F$ of rank $0$. The following theorem shows that away from the holes we can compute the derivative of the density function $\rho_J$ by looking at the slopes of $\Delta_{\vec{\epsilon}}$. Furthermore, it describes the change of slopes, induced by the cuts, in the affine invariant. For completeness we add Karshon's \cite{karshon1999periodic} formula in \eqref{eq.differencespjnflaps}.

\begin{theorem}
\label{t.measurenflaps}
    If $\alpha^{+}(x)$ (resp.\ $\alpha^{-}(x)$) denotes the slope of the top (resp.\ bottom) boundary of the set $\Delta_{\vec{\epsilon}}=f_{\vec{\epsilon}}(F(M))$, then the derivative of the Duistermaat-Heckman function, for $x$ such that $x\notin \pi_x(\operatorname{Im}(\mathcal{F}_i)), i=1,...,n$, is
    \begin{equation}
    \label{eq.derivatitepjnflaps}
        \rho_J'(x)=\alpha^{+}(x)-\alpha^{-}(x)
    \end{equation}
    and is locally constant on $J(M)\backslash \{\pi_x(f_{\vec{\epsilon}}(\Sigma_0(F)))\in \mathbb{R}$. 
    If $(x,y)\in \Sigma_0(F)$ is such that $\pi_x(x,y)=\pi_x(d_{i,1})$, for some $i=1,\cdots,n$, then 
    \begin{equation}
        \label{eq.differenceslopes}
        (\alpha^{+}(x+0)-\alpha^{+}(x-0))-(\alpha^{-}(x+0)-\alpha^{-}(x-0))=-\sum_{j}n_j-e^{+}-e^{-}
    \end{equation}
    where 
    \begin{itemize}
        \item the sum runs over the flaps $\mathcal{F}_j$ such that $\pi_x(d_{j,1})=\pi_x(d_{i,1})$, and $n_j$ is the number of elliptic-elliptic values in the image of the flap $\mathcal{F}_j$;
        \item $e^{+}$ (resp.\ $e^{-}$) is nonzero if and only if an elliptic top vertex (resp.\ a bottom vertex) projects down onto $x$. In this situation
    \begin{equation*}
        e^{\pm}=-\frac{1}{a^{\pm}b^{\pm}}\geq 0
    \end{equation*}
    where $a^{\pm},b^{\pm}$ are the isotropy weights for the $S^1$-action at the corresponding vertices.
    \end{itemize} 
    Moreover, if $(x,y)\in \Sigma_0(F)$ then
    \begin{equation}
    \label{eq.differencespjnflaps} 
        \rho_J'(x+0)-\rho_J'(x-0)=-k-e^{+}-e^{-}
    \end{equation}
    where $e^\pm$ is as above and 
    $k\in \mathbb{N}$ is nonzero if there exists $k$ elliptic-elliptic values $d_1,...,d_k$ in the interior of $F(M)$ such that $\pi_x(d_i)=x$ for $i=1,...,k$.
    %\ke{shouldn't we be looking at $c_i$ instead of $d_i$ and shouldn't $k$ count how many elliptic values are in the flap, that is, $\pi_x(d_i) \in \pi_x(\operatorname{Im}(\mathcal F))$ where $x \in \pi_x(\operatorname{Im}(\mathcal F))$?}
    %\ps{The second equation, where I think your comments are referring to come from the result of Karshon. We could change the content of the equation but instead of working the Duistermaat-Heckman density function we have to work with the difference of the slopes of the affine invariant in this case, which can be two different functions, when we are in the projection of a flap. Away from the projection of a flap they are the same. If we do this yes I agree with your comments, its basically the change of slope we got in the Theorem for the affine invariant. That is the reason we added the phrase "For completeness we add Karshon's formula", just because we were looking at the Duistermaat-Heckman measure and its density function.}
    %\item \ps{Take this part out since its described above?} $e^{+}$ (resp. $e^{-}$) is nonzero if and only if an elliptic top vertex (resp. a bottom vertex) projects down onto $x$. In this situation
    %\begin{equation*}
    %    e^{\pm}=-\frac{1}{a^{\pm}b^{\pm}}\geq 0
    %\end{equation*}
    %where $a^{\pm},b^{\pm}$ are the isotropy weights for the $S^1$-action at the corresponding vertices. 

\end{theorem}
\begin{proof}
     Equation \eqref{eq.derivatitepjnflaps} follows from Proposition \ref{p.lengthnflaps}. Equation \eqref{eq.differenceslopes} follows analogously to the proof of \cite[Theorem $5.3$]{vu2007moment} using Equation \eqref{eq.monodromyrelationnflaps}.
     Equation \eqref{eq.differencespjnflaps} follows from Karshon \cite[Lemma $2.12$]{karshon1999periodic}, where the Duistermaat-Heckman function is computed, see Section \ref{s.duistermaatmeasuregeneral}.
\end{proof}

Theorem \ref{t.measurenflaps} has consequences of topological nature:

\begin{corollary}
    \label{c.flapsfinitenflaps}
    Let $(M,\omega,F:=(J,H))$ be a hypersemitoric system such that the only critical values outside of the elliptic critical values occurring at the boundary of $F(M)$ are caused by standard flaps. Then the number of standard flaps is finite.
\end{corollary}
\begin{proof}
    The proof is analogous to the proof of \vungoc \ \cite[Corollary $5.10$]{vu2007moment}: from Theorem \ref{t.polytopenflaps}, to each system $(M,\omega,F=(J,H))$, one can associate a rational polytope with holes in the interior. Each cut in the image of a standard flap makes a change to the slopes of the upper or lower boundary of the polytope, see Equations \eqref{eq.differenceslopes} and  \eqref{eq.differencespjnflaps}. This phenomenon together with convexity in the region away from the holes gives the desired result. For more details see \vungoc \ \cite[Section $5$]{vu2007moment}.
\end{proof}
\begin{remark}
    Note that Corollary \ref{c.flapsfinitenflaps} only applies to standard flaps. The number of flaps with no elliptic-elliptic values need not be finite. 
\end{remark}

%\ke{Should there be some discussion or explanation of the discontinuity in the ``polytope'' that appears in this case, e.g., in Figure 5.4 top left?}
\begin{remark} When the number of elliptic-elliptic values present in the image of a flap is greater than one, the right sides of Equations \eqref{eq.differenceslopes} and \eqref{eq.differencespjnflaps} are different. This causes a discontinuity visible on the image of the affine invariant along the line where we made the cut, see Figure \ref{f.polytopeinvariantsflapwith2elliptic} top left.
\end{remark}

\subsection{Affine invariant via a cut for each elliptic-elliptic value}
\label{s.polytopeflap}
In order to keep the notation as simple as possible, in this section we only deal with a single flap $\mathcal{F}$ in a hypersemitoric system $(M,\omega,F=(J,H))$. We note that the adaptation to the case of multiple flaps is analogous and is done in Section \ref{s.generalhypersemitoric}, where an affine invariant for a general simple hypersemitoric system is defined.
Let $(M,\omega,F=(J,H))$ be a hypersemitoric system on a $4$-dimensional symplectic manifold $(M,\omega)$ such that the only critical values outside of the elliptic critical values occurring at the boundary of $F(M)$ are caused by a \textbf{standard} flap $\mathcal{F}$. 
%Let $\gamma:[0,1]\rightarrow \mathcal{F}$ be the curve such that $\gamma(0)$ and $\gamma(1)$ are the parabolic values and $\gamma(t)=(x_t,y_t),\ t\in\ ]0,1[$ parametrizes the hyperbolic-regular values. %Recall that by standard flap we mean a flap where the set of elliptic-elliptic values is non-empty and where there are no critical points in the interior of the flap. For a sketch of a standard flap in the case where the number of elliptic-elliptic values is one see Figure \ref{p.flap}.

Let $d_{i}, i=1,...,n$ be the elliptic-elliptic values present in $\operatorname{Im}(\mathcal{F})$. By Corollary \ref{c.flapsfinitenflaps} the number $n$ is finite. %Using the Duistermaat-Heckman measure, see Karshon \cite{karshon1999periodic} or Section \ref{s.duistermaatmeasuregeneral}, the ideas of \vungoc \ \cite[Section $5$]{vu2007moment} and Section \ref{s.measure1cutflap}.
Let $\pi_j:\{-1,1\}^{n}\rightarrow \{-1,1\}$ be the projection onto the $j$-th coordinate, $\pi_x:\mathbb{R}^2\rightarrow \mathbb{R}$ be $(x,y)\mapsto x$ and $\pi_y:\mathbb{R}^2\rightarrow \mathbb{R}$ be $(x,y)\mapsto y$. 
Furthermore, let $c_i\in \gamma_{\mathcal{F}}$ be the hyperbolic-regular value such that $\pi_x(c_i)=\pi_x(d_i)$ and $\vec{\epsilon}=(\epsilon_1,...,\epsilon_n)$ with $\epsilon_i\in \{\pm 1\}$ for $i=1,...,n$. For $\epsilon\in \{-1,1\}$ let $\mathcal{L}^{\epsilon}_{c_i}$ be the vertical half line starting at $c_i$ and going to $\pm \infty$ according to the sign of $\epsilon$, i.e, $\mathcal{L}_{c_i}^{\mathcal{\epsilon}}=\{(x_i,y),\epsilon y\geq \epsilon y_i\}$ where $c_i=(x_i,y_i)$ and $l^{\epsilon}_{c_i}:=F(M)\cap \mathcal{L}^{\epsilon}_{c_i}$. 
Finally, let $\mathcal{T}$ be the subgroup of $\AGL(2,\mathbb{Z})$ which leaves a vertical line, with orientation, invariant.
\begin{theorem}
\label{t.polytopeflap}
Using the notation from above, let $(M,\omega,F)$ be a hypersemitoric system where the only critical values apart from the elliptic values in the boundary of $F(M)$ occur at the image of a \textbf{standard} flap $\mathcal{F}$. Given any $\vec{\epsilon}\in \{-1,1\}^{n}$, there exists a continuous map $f_{\vec{\epsilon}}: F(M)\backslash(\gamma_{\mathcal{F}}\cup_{i=1}^{n} l^{\pi_i(\vec{\epsilon})}_{c_i})\rightarrow \mathbb{R}^2$, such that:
\begin{enumerate}
\item $f_{\vec{\epsilon}}$ is a diffeomorphism onto its image; 
\item $f_{\vec{\epsilon}}$ is affine;
\item $f_{\vec{\epsilon}}$ preserves $J$, i.e., is of the form $f_{\vec{\epsilon}}(x,y)=(x,f_{\vec{\epsilon}}^{(2)}(x,y))$;
\item The map $f_{\vec{\epsilon}}$ extends to a continuous map over $F(M)\backslash\gamma_{\mathcal{F}}$ and for any $c\in \text{int}\left(l^{\epsilon_i}_{c_i}\right)$
\begin{equation}
\label{eq.monodromyrelation}
\lim_{\begin{smallmatrix}(x,y)\rightarrow c \\ x<\pi_x(c_i) \end{smallmatrix}} df_{\vec{\epsilon}}(x,y) = M_{c_i} \lim_{\begin{smallmatrix}(x,y)\rightarrow c \\ x>\pi_x(c_i) \end{smallmatrix}} df_{\vec{\epsilon}}(x,y),
  \end{equation}
  where 
  \begin{equation*}
 M_{c_i}= \begin{bmatrix}
          1 & 0 \\
          k(c_i) & 1
         \end{bmatrix},
\end{equation*}
$k(c_i)=\epsilon_i$ and $\epsilon_i=\pi_i(\vec{\epsilon})$.
\item The image of $f_{\vec{\epsilon}}$ is a rational convex polytope with a ``hole'' in its interior. That is $\operatorname{Im}(f_{\vec{\epsilon}})=B\backslash U$ where $B$ is a rational convex polytope in $\mathbb{R}^2$ and $U\subset B$ is a closed set. 
\end{enumerate} 
Such an $f_{\vec{\epsilon}}$ is unique modulo left composition by $\mathcal{T}$.

Furthermore, there exists a map $g:
\operatorname{Im}(\mathcal{F})\backslash \gamma_{\mathcal{F}}\rightarrow \mathbb{R}^2$ that is affine, $J$-preserving and a diffeomorphism onto its image.
Additionally, it is related to $f_{\vec{\epsilon}}$ by the following formula: for $t \in\ ]0,1[$, $y_t=\pi_y(\gamma_{\mathcal{F}}(t))$ and $x_t=\pi_x(\gamma_{\mathcal{F}}(t))$
\begin{itemize}
\item if $y_t>\pi_y(c_1)$
\begin{equation}
    \label{eq.goodactionflap1}
    \lim_{y\rightarrow y_{t}^{+}}f_{\vec{\epsilon}}^{(2)}(x_t,y)=\lim_{y\rightarrow y_t^{-}}f_{\vec{\epsilon}}^{(2)}(x_t,y)+g(x_t,y)-g(x_t,\tilde{y}_t);
\end{equation}
\item if $\pi_y(c_1)>y_t$
\begin{equation}
    \label{eq.goodactionflap2}
    \lim_{y\rightarrow y_{t}^{-}}f_{\vec{\epsilon}}^{(2)}(x_t,y)=\lim_{y\rightarrow y_t^{+}}f_{\vec{\epsilon}}^{(2)}(x_t,y)+g(x_t,y)-g(x_t,\tilde{y}_t),
\end{equation}
where $(x_t,\tilde{y}_t)$ is the elliptic value in the boundary of $\operatorname{Im}(\mathcal{F})$ with $x$-coordinate $x_t$. 
\end{itemize}
The map $g$ does not depend on $\vec{\epsilon}$ and is unique modulo left composition by $\mathcal{T}$.

%For examples see Figure \ref{p.iafsflap}, and Figures \ref{f.actionangle2points2cuts11}, \ref{f.actionangle2points2cuts1-1}, \ref{f.flappypart}. 
\end{theorem}
\begin{proof}
The case $n=1$ follows from Theorem \ref{t.polytopenflaps}. We focus on the case $n=2$ for simplicity, the case $n>2$ is analogous. 

First we prove the existence of the map $f_{\vec{\epsilon}}$. We work on the background $\mathcal{B}$ of the system, i.e., consider the background of the flap. For any choice of $\epsilon$ the domain $B\backslash (l_{c_1}^{\pi_1(\vec{\epsilon})}\cup l_{c_2}^{\pi_2(\vec{\epsilon})})$ is not connected, and splits into two connected components. Let us denote these connected components by $U_1$ and $U_2$. Furthermore, $U_1$ and $U_2$ are simply connected. Therefore, we can take action coordinates on each $U_i, i=1,2$, and obtain $f_i:U_i\rightarrow \mathbb{R}^2$ satisfying conditions $(1)$ and $(2)$ of the Theorem. Since $J$ generates an effective $S^1$-action we can assume that $f_i(x,y)=(x,g_i(x,y))$ for $i=1,2$ and some smooth functions $g_i$. 

Recall that action coordinates can be interpreted as the measurement of the volume of a subset of the symplectic manifold. Therefore, there is a unique way (modulo left composition by an element of $\mathcal{T}$) to choose $f_1$ and $f_2$ to create a function $f_{\vec{\epsilon}}$ such that 
\begin{equation*}
    f_{\vec{\epsilon}}(x,y)=\begin{cases}
        f_1(x,y),\ (x,y)\in U_1,\\
        f_2(x,y),\ (x,y)\in U_2,
    \end{cases}
\end{equation*}
is continuous and $f_{\vec{\epsilon}}$ has the property that for fixed $x_0$, $\max f_{\vec{\epsilon}}(x_0,y)-\min f_{\vec{\epsilon}}(x_0,y)$ for $y\in J^{-1}(x_0)$ is the symplectic volume of the  $J^{-1}(x)/S^1$, which we denote by $\rho(x)$. Recall that for a proper $S^1$-action the function $\rho(x)$ has been computed, see for example Karshon \cite{karshon1999periodic} or Section \ref{s.duistermaatmeasuregeneral}. Using the formula of $\rho(x)$ we obtain equation \eqref{eq.monodromyrelation}. Using the formula of $\rho(x)$ and the fact that near the boundary of $F(M)$ the system is toric, we obtain property $(5)$ of the theorem. For more details on this see \vungoc\ \cite[Theorem $3.8$]{vu2007moment}. 

The map $g$ follows from taking action coordinates on the flap, which is possible since the set of regular values there is simply connected. Equations \eqref{eq.goodactionflap1} and \eqref{eq.goodactionflap2} follow from the fact that the height of the image of $f_{\vec{\epsilon}}$ at a fixed $x$ is $\rho(x)$. 
\end{proof}

\begin{definition}
    We call the image of $f_{\vec{\epsilon}}$ and the image of $g$ from Theorem \ref{t.polytopeflap} a \textit{representative of the affine invariant} of $(M,\omega,F)$. The \textit{affine invariant} is the collection of all the representatives for all the choices of $\vec{\epsilon}$. For an alternative definition, see Section \ref{s.orbitpolytope}.
\end{definition}

\subsection{Group orbit of the affine invariant in the case of one cut for each elliptic-elliptic value}
\label{s.orbitpolytope}
As in Section \ref{s.polytopeflap} let $(M,\omega,F=(J,H))$ be a hypersemitoric system on a $4$-dimensional symplectic manifold $(M,\omega)$ such that the only critical values outside of the elliptic critical values in the boundary of $F(M)$ are caused by a \textbf{standard} flap $\mathcal{F}$. Let $c_i\in \gamma_{\mathcal{F}}$ be defined as in Section \ref{s.polytopeflap}. For $\vec{\epsilon}\in \{-1,1\}^{n}$, Theorem \ref{t.polytopeflap} gives us a rational convex polytope $\Delta_{\vec{\epsilon}}:=\operatorname{Im}(f_{\vec{\epsilon}})$ which has a hole in its interior, where $f_{\vec{\epsilon}}$ is the map given in Theorem \ref{t.polytopeflap}. In this section we investigate, analogous to \vungoc\ \cite[Section $4$]{vu2007moment}, the relations between these polytopes when we change $\vec{\epsilon}$.

Let $\mathcal{L}$ be a vertical line in $\mathbb{R}^2$. Then $\mathcal{L}$ splits $\mathbb{R}^2$ into two half-spaces, $A_1$ to the left of $\mathcal{L}$ and $A_2$ to the right. Define the piecewise affine transformation $t^{\epsilon}_{\mathcal{L}}$ acting as the identity on $A_1$  and as the matrix $T_{\epsilon}:=\begin{bsmallmatrix}1 & 0\\ \epsilon & 1\end{bsmallmatrix}$ on $A_2$, where $\epsilon\in \{-1,1\}$. Note that $T$ fixes $\mathcal{L}$. 

Consider now the vertical lines $\mathcal{L}_i,\ i=1,...,n$, through the elliptic-elliptic values $c_i$. For $\vec{\epsilon}=(\epsilon_1,...,\epsilon_n)\in \{-1,1\}^n$ define the piecewise affine transformation $t^{\vec{\epsilon}}:=t^{\epsilon_1}_{\mathcal{L}_1}\circ \cdot \cdot \cdot \circ t^{\epsilon_n}_{\mathcal{L}_n}$. Furthermore, consider $G=\{0,1\}^n$, with the structure of the Abelian group $(\mathbb{Z}_2)^{n}$. 

\begin{proposition}
    \label{p.groupofpolytopes1cutforeachelliptic}
    Let $G$ act transitively on the set
    \begin{equation*}
        \Delta:=\{\Delta_{\vec{\epsilon}}\mid \vec{\epsilon}\in \{-1,1\}^n \}
    \end{equation*}
    by the formula
    \begin{equation}
    \label{eq.Gaction}
        G\times \Delta \rightarrow \Delta,\quad (\vec{u},\Delta_{\vec{\epsilon}}) \mapsto \vec{u}\cdot \Delta_{\vec{\epsilon}}:=\Delta_{(-2\vec{u}+1)\cdot \vec{\epsilon}}
    \end{equation}
    where scalar multiplication and addition is component wise, and the dot $\cdot$ between $\vec{\epsilon}$ and $\vec{u}$ means component wise multiplication. Then the action satisfies:
    \begin{equation*}
    %\label{eq.taction}
        \vec{u}\cdot \Delta_{\vec{\epsilon}}=t^{\vec{u}\cdot \vec{\epsilon}}\Delta_{\vec{\epsilon}}.
    \end{equation*}
    Furthermore, the action is free. 
\end{proposition}
\begin{proof}
Analogous to the proof of Proposition \ref{p.groupofpolytopes1cutforeachflap}.
\end{proof}

\subsection{Examples}
\label{s.examplespolytopeflap}
\subsubsection{The affine invariant for the modified Jaynes-Cummings model}Recall $\mathbb{R}^2\times S^2$ with coordinates $(u,v,x,y,z)$ and $J,H,G$ from Section \ref{s.jaynescummingsmodel}. The modified Jaynes-Cummings model is given by $(J,H+G):\mathbb{R}^2\times S^2\rightarrow \mathbb{R}^2$. 

Using the computations of Appendix \ref{s.actionJaynesCummings} we obtain the following:
\begin{example}
\label{e.holger}
Figure \ref{p.iafsflap} is a representative of the affine invariant given by Theorem \ref{t.polytopenflaps} and Theorem \ref{t.polytopeflap} of the modified Jaynes-Cummings model for $\vec{\epsilon}=+1$. This example was first presented by Dullin \cite{holgertalk}.
\end{example}
    
\begin{figure}[ht]
\centering
    \includegraphics[scale=1]{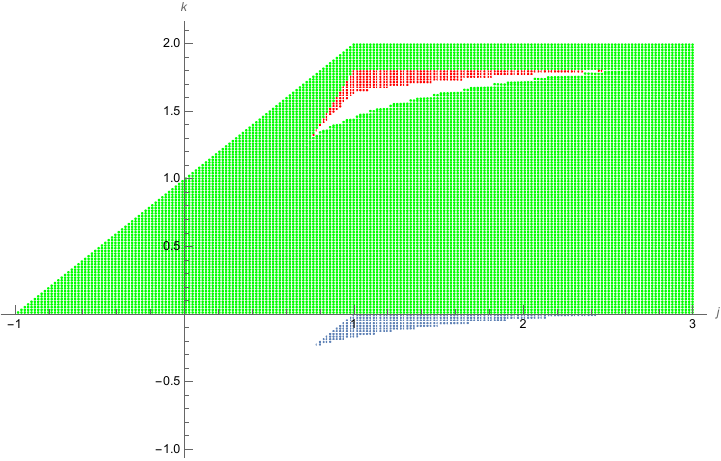}
    \caption{Classical actions on the joint spectrum for $\hbar=\frac{2}{101}$ and $\epsilon=1$. The green dots correspond to the classical actions on the regular values outside of the image of the flap. The red dots correspond to the classical actions on the regular values on the background of the flap, and the white region is the discontinuous jump caused by the hyperbolic-regular values. The blue dots correspond to the classical actions on the flap, which we chose to plot at the bottom in order to make the figure less confusing. This example was first presented by Dullin \cite{holgertalk}.}
\label{p.iafsflap}
\end{figure}

\subsubsection{The affine invariant for a system containing a flap with two elliptic-elliptic values}
\label{s.systemflapwith2elliptic}
Recall the results and notation of Section \ref{s.quantizationhirzebruch} and consider the Hirzebruch surface $W_{1,1,2}$. Consider the functions $J=\frac{1}{2}(|z_2|^2+|z_3|^2), R=\frac{1}{2}(|z_3|^2-|z_4|^2)$ and $X=\Re(\overline{z_1z_2}z_3\overline{z_4})$. Notice that the functions $J$ and $R$ in this section are \textbf{different} from the ones in Section \ref{s.quantizationhirzebruch}. The Hamiltonians of the system are $(J,H_t)$ where 
\begin{equation*}
    H_t=(1-t)R+\frac{t}{3}\left(\frac{9}{20}X+(2J-3)(R+2)\right)+2t|z_3|^2|z_4|^2.
\end{equation*}
For $t=0.44$ the system exhibits a flap with two elliptic-elliptic values. The bifurcation diagram of the system for $t=0.44$ together with its joint spectrum for $\hbar=0.5$ are represented in Figure \ref{f.bifucartiondiagramwithspecflapwith2ellipticelliptic}. For more details on the quantization of the system see Appendix \ref{s.quantizationflap2ellitptic}. Notice that for $t=0$ the system $(J,H_0)$ is \textbf{not} toric, since $(J,R)$ do not generate an effective $\mathbb{T}^2$-action.
\begin{figure}[ht]
\begin{center}
    \includegraphics[width=0.8\textwidth]{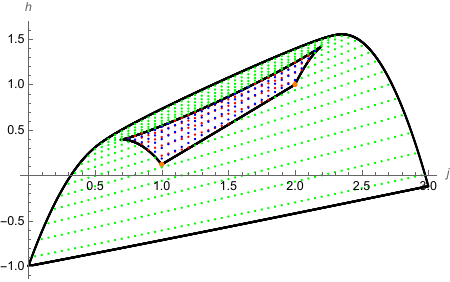}
    \caption{Bifurcation diagram of the system $(W_{1,1,2},\omega_{1,1,2},(J,H_{0.44}))$ together with its joint spectrum for $\hbar=0.05$. The orange and black dots form the bifurcation diagram, where the orange dot represent the elliptic-elliptic values on the image of the flap. Green dots corresponds to values of the joint spectrum outside of the image of the flap. Red dots correspond to values of the joint spectrum in the background of the flap. Blue points correspond to values of the joint spectrum on the flap.}
\label{f.bifucartiondiagramwithspecflapwith2ellipticelliptic}
\end{center}
\end{figure}
Using the computation of the classical actions done in Appendix \ref{s.actionangle2llipticflap} we obtain the following:
%\begin{figure}[h]
%\begin{center}
 %   \includegraphics[]{Jointspectrumflapwith2ellipticvalues.pdf}
 %   \caption{Joint spectrum of the quantization of $(W_{1,1,2},\omega_{1,1,2},(J,H_{0.44}))$ for $\hbar=0.05$. Green dots correspond to values outside of the flap. Red dots correspond to values on the background of the flap. Blue dots correspond to values on the flappy part of the flap. }
%\label{f.jointspectrumflapwith2ellipticvalues}
%\end{center}
%\end{figure}

\begin{example}
    A representative of the affine invariant given by Theorem \ref{t.polytopenflaps} for the system $(W_{1,1,2},\omega_{1,1,2},(J,H_{0.44}))$ and $\vec{\epsilon}=+1$ is given in Figure \ref{f.polytopeinvariantsflapwith2elliptic}
\end{example}

\begin{example}
    A representative of the affine invariant given by Theorem \ref{t.polytopeflap} for the system $(W_{1,1,2},\omega_{1,1,2},(J,H_{0.44}))$ and $\vec{\epsilon}=(+1,+1)$ is given in Figure \ref{f.polytopeinvariantsflapwith2elliptic}.
\end{example}

\begin{example}
    A representative of the affine invariant given by Theorem \ref{t.polytopeflap} for the system $(W_{1,1,2},\omega_{1,1,2},(J,H_{0.44}))$ and $\vec{\epsilon}=(+1,-1)$ is given in Figure \ref{f.polytopeinvariantsflapwith2elliptic}.
\end{example}

\begin{figure}[ht]
\begin{center}
    \includegraphics[]{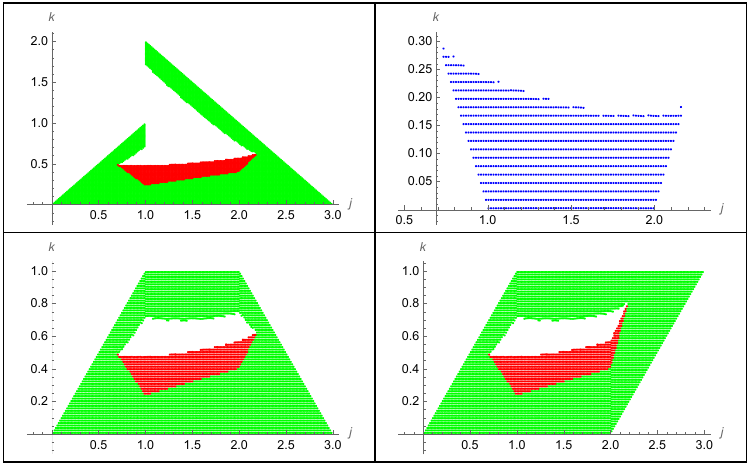}
    \caption{Different representatives of the affine invariant for the system $(W_{1,1,2},\omega_{1,1,2},(J,H_{0.44}))$ computed on the joint spectrum of the system. Green values correspond to the classical actions computed outside of the image of the flap. Red values correspond to the classical actions computed on the background of the flap. Blue values correspond to the classical actions computed on the flap. The joint spectrum for these computations is computed with $\hbar=0.015$. The picture on the left top corner is the result of applying Theorem \ref{t.polytopenflaps} for $\vec{\epsilon}=+1$. The picture on the left lower corner is the result of applying Theorem \ref{t.polytopeflap} for $\vec{\epsilon}=(+1,+1)$. The picture on the right lower corner is the result of applying Theorem \ref{t.polytopeflap} for $\vec{\epsilon}=(+1,-1)$.}
    \label{f.polytopeinvariantsflapwith2elliptic}
\end{center}
\end{figure}

%\begin{figure}[h]
%\begin{center}
 %   \includegraphics[scale=0.8]{flappypart.pdf}
  %  \caption{Classical actions, with $\hbar=0.05$, in the flappy part of the flap of the system $(W_{1,1,2},\omega_{1,1,2},(J,H_{0.44}))$. No choice of cut is needed here since the set of regular values is simply connected in the flappy part of the flap.}
  %  \label{f.flappypart}
%\end{center}
%\end{figure}
\section{Affine invariant for pleats}
\label{s.pleat}
Consider an  integrable system $(M,\omega,F=(J,H))$ on a $4$-dimensional symplectic manifold $(M,\omega)$ where $J$ generates an effective $S^1$-action. Furthermore, assume that the system has a pleat/swallowtail and does not admit other types of singularities other than elliptic-elliptic or elliptic-regular type occurring at the boundary of $F(M)$.
Denote by $P$ this pleat/swallowtail in $F(M)$, for a sketch see Figure \ref{p.pleat}.
\subsection{Definition of the affine invariant}
\label{s.polytopepleat}
The aim is to associate to $(M,\omega,F)$ an affine invariant. First let us fix some notation. Let $\gamma(t)=(x_t,y_t)$, $t\in [0,1]$, be the curve of singular values such that $\gamma(0)$ and $\gamma(1)$ are the parabolic values and $\gamma(t), \ t\in ]0,1[$ is a hyperbolic-regular value. Let $l_1$ and $l_2$ be the curves of elliptic-regular values that intersect transversely. Recall that the preimage of a regular value $F^{-1}(c)$ for $c\in F(M)\backslash P$ is a torus, the preimage of $c\in \gamma(]0,1[])$ is a bitorus, and the preimage of $c$ in the interior of $P$ is the disjoint union of two tori $T^{+}(c)\cup T^{-}(c)$, see Figure \ref{p.pleat}. We notice that at $l_1$ the torus $T^{-}(c)$ degenerates to a circle and then vanishes while the torus $T^{+}(c)$ continues. The reverse situation happens when crossing $l_2$. Furthermore, note that the swallowtail does not generate monodromy in the sense of Section \ref{s.monodromy}, since the set of regular values is simply connected. 
%\kedel{However there is the phenomenon of bidromy, see Efstathiou \& Sugny \cite{efstathiou2010integrable}. }

There are two options to define an affine invariant of $(M,\omega,F)$. Both arise of the idea of introducing action coordinates on the set of regular values, or in a subset of the regular values that is simply connected. However for a pleat there are the following two natural choices. Option $1$ is to obtain action coordinates with respect to the tori $T^{+}(c)$ in the pleat, while option $2$ is to obtain action coordinates with respect to the tori $T^{-}(c)$. This is formalized as follows. Recall that $\mathcal{T}$ denotes the subgroup of $\AGL(2,\mathbb{Z})$ that leaves a vertical line, with orientation, invariant.
\begin{theorem} 
\label{t.polytopepleat}
Let $(M,\omega,F)$ be the system defined above:
    \begin{enumerate}
        \item There exists a smooth map $f_1$, unique modulo left composition by an element in $\mathcal{T}$, from $F(M)\backslash (l_2\cup  \gamma([0,1]))$ to its image in $\mathbb{R}^2$, which is affine and preserves $J$, i.e., it is of the form $f_1(x,y)=(x,f_1^{(2)}(x,y))$.
        \item There exists a smooth map $f_2$, unique modulo left composition by an element in $\mathcal{T}$, from $F(M)\backslash (l_1\cup  \gamma([0,1]))$ to its image in $\mathbb{R}^2$, which is affine and preserves $J$, i.e., it is of the form $f_2(x,y)=(x,f_2^{(2)}(x,y))$.
    \end{enumerate}
    Note that $f_1,f_2$ cannot be continuously extended over $(l_2\cup \gamma([0,1])), (l_1\cup \gamma([0,1]))$ respectively. %Geometrically this is displayed by swallowtail shaped gaps in Figure \ref{f.pleat1} and Figure \ref{f.pleat2}. 
    
    Furthermore, let $\pi_i:\mathbb{R}^2\rightarrow \mathbb{R}$ be the projection onto the i-th coordinate coordinate for $i=x,y$. There exists a choice of $f_1$ and $f_2$, such that for $t\in \ ]0,1[ \ :$
    \begin{equation}
    \label{eq.pleatgoodaction}
        \lim_{y\rightarrow y_t^{-}}f_1^{(2)}(x_t,y)+\lim_{y\rightarrow y_t^{-}}f_2^{(2)}(x_t,y)=\lim_{y\rightarrow y_t^{+}}f_1^{(2)}(x_t,y)=\lim_{y\rightarrow y_t^{+}}f_2^{(2)}(x_t,y)
    \end{equation}
    if the $y$-projection of the hyperbolic-regular values in $P$ is larger than the $y$-projection of the elliptic-regular values in $P$. The reverse situation is analogous. 
    
\end{theorem}
\begin{proof}
     Recall that there are two options. Option $1$ is to take action coordinates with respect to the tori $T^{+}(c)$, obtaining the map $f_1$, while option $2$ is to take action coordinates with respect to the tori $T^{-}(c)$, obtaining the map $f_2$. This is possible in both cases due to Theorem \ref{t.AL} and the fact that the set is simply connected. The fact that the maps preserve the $S^1$-action follows from an argument analogous to the one in the proof of Theorem \ref{t.polytopeflap}. 
     
    Since a generator of the homology group of a torus $F^{-1}(c)$ with $c\in F(M)\backslash P$ splits into the sum of the two generators of $T^{+}(c)\cup T^{-}(c)$ when we cross the line of hyperbolic-regular values, the maps $f_1$ and $f_2$ cannot be extended continuously over $(l_2\cup \gamma([0,1])), (l_1\cup \gamma([0,1]))$ respectively. Furthermore, we can choose maps $f_1$ and $f_2$ in a way such that equation \eqref{eq.pleatgoodaction} is satisfied. See Efstathiou \& Sugny \cite[Section $2.3$]{efstathiou2010integrable} for more details.
\end{proof}

\begin{definition}
    We call the image of $f_1$ and $f_2$ from Theorem \ref{t.polytopepleat} the affine invariant for the system $(M,\omega,F)$. 
\end{definition}

\subsection{Example} 
\label{s.pleatexample}
In order to compute and visualize Theorem \ref{t.polytopepleat} we now give an example of an integrable system on a Hirzebruch surface with a pleat. 
The example is motivated by Le Floch \& Palmer \cite[Example $2.2.6$]{palmer2019semitoric}. Throughout this subsection we fix $\beta=n=1$ and $\alpha=1.02$. The ``strange'' choice of $\alpha$ is due to the fact that in the computations of the joint spectrum we will take $\hbar=\frac{1}{25}$ and we want to have a nontrivial Hilbert space, i.e.,  there need to exist $M_1$ and $M_2$ natural numbers that are solutions of $\hbar(M_1+\frac{3}{2})=\alpha+n\beta$ and $\hbar(M_2+1)=\beta$. See Subsection \ref{s.quantizationhirzebruch} for more details.  Let $J=\frac{1}{2}|z_2|^2, X:=\Re(\overline{z_1}z_3\overline{z_4})$, $R=\frac{1}{2}|z_3|^2$ and set
\begin{equation*}
    H_t:=(1-2t)R+tX+2t|z_1|^2|z_3|^2.
\end{equation*}
We want to consider the integrable system $(J,H_t)$ on the Hirzebruch surface $W_{1.02,1,1}$. The bifurcation diagram of the system  for $t=1$ together with its joint spectrum for $\hbar=\frac{1}{25}$ is plotted in Figure \ref{f.bifurcationdiagrampleatwithspectrum}. For mode details on the quantization of the system see Appendix \ref{s.quantizationpleat}.
\begin{figure}[ht]
\begin{center}
    \includegraphics[width=0.6\textwidth]{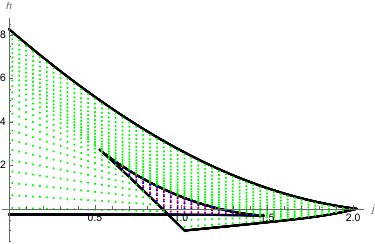}
    \caption{Bifurcation diagram for the system $(W_{1.02,1,1},\omega_{1.02,1,1},(J,H_1))$ together with the joint spectrum for the choice $\hbar = \frac{1}{25}$. Black dots represent the bifurcation diagram. Green dots correspond to values of the joint spectrum outside of the swallowtail. Red and Blue dots correspond to values of the joint spectrum inside the swallowtail depending on the choice of torus. }
    \label{f.bifurcationdiagrampleatwithspectrum}
\end{center}
\end{figure}

Using the computation of the classical actions of the system described in Appendix \ref{s.actionpleat} we obtain:
\begin{example}
    Figure \ref{f.panelswallowtail} shows the affine invariant for the hypersemitoric system $(W_{1.02,1,1},\omega_{1.02,1,1},(J,H_1))$ given by Theorem \ref{t.polytopepleat}.
\end{example}

%\begin{figure}[h]
%\begin{center}
 %   \includegraphics[]{jointspectrum01.pdf}
 %   \caption{The joint spectrum of $(J,H_t)$ when $\hbar=0.1$ and $t=1$. Green dots correspond to values outside the swallowtail. Red and Blue dots correspond to values inside the swallowtail depending on the choice of torus. }
 %   \label{f.spectrumpleat01}
%\end{center}
%\end{figure}

%Getting a higher resolution (i.e.\ more grid points) of the pictures corresponds to working with a smaller value of $\hbar$. Compare the more detailed Figure \ref{f.pleat4} to the less detailed Figure \ref{f.pleat3}.
%To work with a smaller value of $\hbar$, for example $\hbar=\frac{1}{25}$, we need to consider a different $\alpha$, for example $\alpha=1.02$. This is in order to have a non-trivial Hilbert space, i.e., to ensure that $M_1$ and $M_2$ are natural numbers when looking for solutions of $\hbar(M_1+\frac{3}{2})=\alpha+n\beta$ and $\hbar(M_2+1)=\beta$. See Subsection \ref{s.quantizationhirzebruch} for more details.

%\begin{corollary}
%Figure \ref{f.pleat4} shows a representative of the polytope invariant for the system $(W_{1.02,1,1},\omega_{1.02,1,1},(J,H_{1}))$ given by Theorem \ref{t.polytopepleat}.
%\end{corollary}

\begin{figure}[ht]
\begin{center}
   \includegraphics[width=1\textwidth]{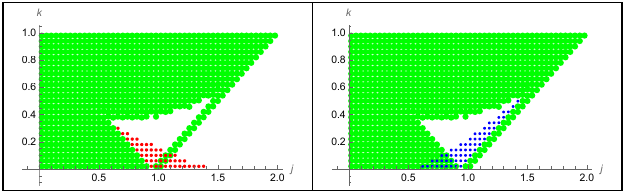}
    \caption{Classical actions for the system $(W_{1.02,1,1},\omega_{1.02,1,1},(J,H_1))$ computed on the joint spectrum of the system when $\hbar=\frac{1}{25}$. The picture on the left corresponds to the choice of tori in the swallowtail with smaller values of $|z_3|^2$. The green dots are values of the action on points outside the swallowtail and the red points correspond to the action on values inside the swallowtail. 
    The picture on the right corresponds to the choice of tori in the swallowtail with bigger values of $|z_3|^2$. The green dots are values of the action on points outside the swallowtail and the blue points correspond to the action on values inside the swallowtail.}
   \label{f.panelswallowtail}
\end{center}
\end{figure}

\section{Affine invariant in the presence of a line of curled tori}
\label{s.curledtori}
\subsection{Motivation}
The motivation for this section is the following situation: consider an integrable system $(M,\omega,F=(J,H))$ on a $4$-dimensional symplectic manifold such that $J$ generates an effective $S^1$-action. Furthermore, suppose that the isotropy weights at the fixed points of $J$ are of the form $(1,n)$ where $n\in \mathbb{Z}$, $|n|\geq 2$. Then the system may fail to be hypersemitoric due to the presence of degenerate points that are not parabolic. The generic situation is that in the bifurcation diagram there exist curves of hyperbolic-regular values that connect to a degenerate, nonparabolic value from which it is further connecting to the image of a so called generalized flap. However, it can also be the case that the curve of hyperbolic-regular values connects to a single degenerate, nonparabolic value.  We want to consider these nonhypersemitoric scenarios. In Section \ref{s.examplecurledtori} we give an example of a line of curled tori ending in a single degenerate value and compute the classical actions of the system. In Section \ref{s.curledtorimicroflap} we give an example of the generic situation, i.e, a line of hyperbolic-regular values ending in a degenerate value connecting to the image of a generalized flap, and compute the classical actions of the system. 

\begin{definition}
    \label{d.lineofcurledtori}
    Let $(M,\omega,J)$ be a Hamiltonian $S^1$-space. We say that the integrable system $(M,\omega,F=(J,H))$ is a \textit{system exhibiting a line of curled tori} if there exists a curve $c:[0,1]\rightarrow F(M)$ such that $c(]0,1[)$ consists of hyperbolic-regular values and $c(0)$ or $c(1)$ is in $\partial(F(M))$. Furthermore, each hyperbolic-regular value is required to have precisely one curled torus as preimage.
\end{definition}
For an illustration of systems exhibiting a line of curled tori see Figure \ref{f.bifurcationdiagramX} and Figure \ref{f.bifurcationdiagrammicroflap}.

\subsection{Line of curled tori ending in a single nonparabolic value}
\label{s.examplecurledtori}
Consider the Hirzebruch surface $W_{1,1,2}$, constructed in Section \ref{s.quantizationhirzebruch}, with the Hamiltonians $F:W_{1,1,2}\rightarrow \mathbb{R}^2$, where  $F=(J,X):=(\frac{1}{2}|z_2|^2,\Re(\overline{z_1}^2z_3\overline{z_4}))$. The fibers $(J,X)^{-1}(t,0)$ with $1<t<3$ are hyperbolic-regular, curled tori. $(J,X)^{-1}(1,0)$ and $(J,X)^{-1}(3,0)$ are degenerate fibers, with $(3,0)$ lying on the boundary of $F(M)$. All fibers of the system are connected. The bifurcation diagram of the system together with its joint spectrum for $\hbar=0.05$ is shown in Figure \ref{f.bifurcationdiagramX}. For more details on the quantization of the system see Appendix \ref{s.quantizationcurledtori}. 

\begin{figure}[ht]
\begin{center}
    \includegraphics[]{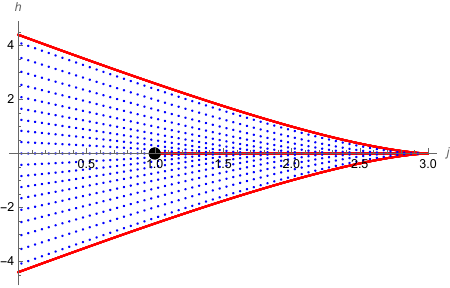}
    \caption{Bifurcation diagram for the system $(W_{1,1,2},\omega_{1,1,2},(J,X))$ together with its joint spectrum for $\hbar=0.05$. The bifurcation diagram is sketched in red. The degenerate, nonparabolic value in the interior of $F(M)$ is indicated in black. The values from the joint spectrum are plotted in blue.}
    \label{f.bifurcationdiagramX}
\end{center}
\end{figure}

Notice that the set of regular values of $(W_{1,1,2},\omega_{1,1,2},F)$ is simply connected, hence there is no topological monodromy. There is however fractional monodromy, see Section \ref{s.fractionalmonodromy} or
Efstathiou \& Martynchuk \cite{martynchuk2017parallel} for more details. 
Since the set of regular values is simply connected and the fibers of the regular values of the system $(W_{1,1,2},\omega_{1,1,2},F)$ are connected there are no choices of vertical cuts that need to be made, contrary to what was necessary in Section \ref{s.polytopenflaps}, Section \ref{s.polytopeflap} and Section \ref{s.polytopepleat}. Therefore, obtaining an affine invariant for this system is solely based on Theorem \ref{t.AL}. Let $\mathcal{T}$ denote the subgroup of $\AGL(2,\mathbb{Z})$ that leaves a vertical line, with orientation, invariant. Furthermore, let $c(t)=(x_t,y_t),\ t\in [0,1]$ be the line in $F(W_{1,1,2})$ such that $c(]0,1[)$ are hyperbolic-regular values and $c(0),c(1)$ are critical values of rank $0$, with $c(0)$ in the boundary of $F(M)$. 
\begin{lemma}
\label{l.curledtori}
Let $(W_{1,1,2},\omega_{1,1,2},F)$ be as defined above. Then there exists a unique diffeomorphism $f$, up to composition with an element in $\mathcal{T}$, from $F(M)\backslash c([0,1])$ to its image in $\mathbb{R}^2$ such that $f$ is affine and preserves $J$, i.e., $f(x,y)=(x,f^{(2)}(x,y))$.
Furthermore, $f$ cannot be continuously extended over $c([0,1[)$.
Let $(1,n)$ be the isotropy weights associated with the critical value $c(1)$ of rank $0$, which is in fact in this example $n=2$.
Then for $t\in \ ]0,1[$
%\begin{itemize}
 %   \item if $\pi_x(c(1))>\pi_x(c(0))$ we have 
%\begin{equation}
 %   \label{eq.dline1}
  %  \lim_{y\rightarrow y_t^{+}}f^{(2)}(x_t,y)=\lim_{y\rightarrow y_t^{-}}f^{(2)}(x,y)+(x_1-x_t)(\sum_{i=1}^{l}\frac{1}{k_i}),
%\end{equation}
 %   \item if $\pi_x(c(0))>\pi_x(c(1))$ we have 
\begin{equation}
    \label{eq.dline2}
    \lim_{y\rightarrow y_t^{+}}f^{(2)}(x_t,y)=\lim_{y\rightarrow y_t^{-}}f^{(2)}(x,y)+\frac{(x_t-x_1)}{n}.
\end{equation}
%\end{itemize}
\end{lemma}
\begin{proof}
   Recall that the set of regular values is simply connected and the fibers are connected, therefore, by Theorem \ref{t.AL} there are well defined action coordinates on the set of regular values. Analogously to the proof of Theorem \ref{t.polytopeflap}, since $J$ generates an effective $S^1$-action, the map can be chosen to be of the form $f(x,y)=(x,f^{(2)}(x,y))$.
    
    The degenerate value $c(1)$, image of a fixed point with isotropy weights $(1,n)=(1,2)$, implies the existence of fractional monodromy, see Section \ref{s.fractionalmonodromy}. This means that, when circling around the degenerate value, a generator of a torus $T(c)$, where $c$ is a regular value of $F$, splits into the sum of itself and $\frac{1}{n}$ of the generator induced by the $S^1$-action. More precisely, let $b$ in $H_1(T(c))$ be the generator induced by the $S^1$-action and $a$ be another generator completing the basis. Then, when circling around the degenerate value, by Theorem \ref{t.fractionalmonodromy}, $(b,Na)\mapsto (b,Na+kb)$, where $N$ and $k$ are the constants in Theorem \ref{t.fractionalmonodromy}, which formally means that $(b,a)\mapsto (b,a+\frac{b}{n})$ by Theorem \ref{t.weightsfractionalmonodromy}.

    In other words, for the action coordinates, and therefore for the map $f$, this means that they cannot be continuously extended over the line of curled tori. The discontinuity corresponds to a triangular indention in Figure \ref{f.polytopecurledtori} that is opening up along the line of curled tori.
\end{proof}

We call the image of the map $f$ from Lemma \ref{l.curledtori} the affine invariant of the system $(W_{1,1,2},\omega_{1,1,2},F)$.
    
Using the computation of the classical actions done in Appendix \ref{s.actioncurledtori} we obtain:

\begin{example}
    Figure \ref{f.polytopecurledtori} is a representative of the affine invariant for the system $(W_{1,1,2},\omega_{1,1,2},(J,X))$ given by Lemma \ref{l.curledtori}. The discontinuities mentioned in Lemma \ref{l.curledtori} by equation %equations \eqref{eq.dline1},%
    \eqref{eq.dline2} are given by the formula $\frac{j-1}{2}$.
\end{example}

\begin{figure}[ht]
\begin{center}
    \includegraphics[width=0.6\textwidth]{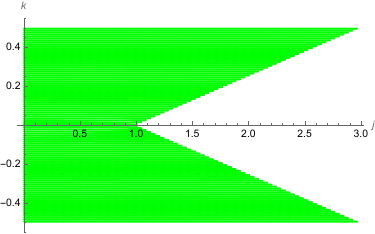}
    \caption{Classical actions of the system $(W_{1,1,2},\omega_{1,1,2},(J,X))$ evaluated on the joint spectrum for $\hbar=0.01$.}
    \label{f.polytopecurledtori}
\end{center}
\end{figure}

\begin{remark}
    Analogously to Section \ref{s.polytopenflaps} and Section \ref{s.polytopeflap} a vertical cut could be made at the value $c(1)$, obtaining maps $f_1$ and $f_2$ in each connected component. The map $f_2$ can be continuously extended to $f_1$ by applying the transformation $f_2-\frac{(J-1)}{2}$. Notice that this transformation is not affine. However, $2f_2$ can be continuously extended by considering the affine transformation $2f_2-(J-1)$. 
\end{remark}

\subsection{Line of curled tori connecting to the image of a generalized flap}
\label{s.curledtorimicroflap}
The example of this section is a perturbation of the example in Section \ref{s.examplecurledtori}. Consider the Hirzebruch surface $W_{1,1,2}$ and the Hamiltonians $J:=\frac{1}{2}|z_2|^2$, $R=\frac{1}{2}|z_3|^2$, $X=\mathfrak{R}(\overline{z_1}^2z_3\overline{z_4})$. The integrable system we consider is $(W_{1,1,2},\omega_{1,1,2},F=(J,H))$ where 
\begin{equation*}
    H:=0.75R+0.25X.
\end{equation*}
The bifurcation diagram of the system together with its joint spectrum for $\hbar=0.05$ is shown in Figure \ref{f.bifurcationdiagrammicroflap}.
The quantization of this system is analogous to what is done in Appendix \ref{s.quantizationcurledtori}. The system exhibits a so called generalized flap, see Figure \ref{f.microflapzoom}.
\begin{figure}[ht]
\begin{center}
    \includegraphics[]{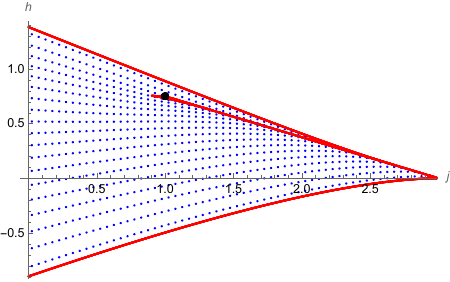}
    \caption{Bifurcation diagram for the system $(W_{1,1,2},\omega_{1,1,2},(J,H))$ together with its joint spectrum for $\hbar=0.05$. The bifurcation diagram is sketched in red. The degenerate, nonparabolic value in the interior of $F(M)$ is indicated in black. The values from the joint spectrum are plotted in blue.}
    \label{f.bifurcationdiagrammicroflap}
\end{center}
\end{figure}

\begin{figure}[ht]
\begin{center}
    \includegraphics[]{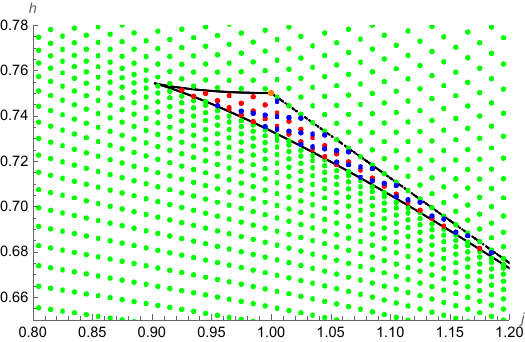}
    \caption{Bifurcation diagram of the system together with its joint spectrum, for $\hbar=0.01$, in the range $(j,h)\in [0.8,1.2]\times[0.65,0.78]$. The image of the generalized flap can be seen to appear. The bifurcation diagram is sketched in black. The orange point represents the degenerate value. Green points correspond to values in the joint spectrum outside of the image of the generalized flap. Red points correspond to values in the joint spectrum in the background of the generalized flap. Blue points correspond to values in the joint spectrum in the generalized flap.}
    \label{f.microflapzoom}
\end{center}
\end{figure}

Notice that the set of regular values on the background of the system is simply connected. Therefore, we can compute the classical actions on the background of the system without introducing any choice of cut, analogous to Lemma \ref{l.curledtori}. Furthermore, we can compute the classical actions on the generalized flap. Using an analogous method to the one in Appendix \ref{s.actioncurledtori} we obtain Figure \ref{f.gridmicroflap}. 

Furthermore, for a system exhibiting a line of curled tori we can apply the previous ideas and obtain the following:
\begin{proposition}
\label{p.generalizedcurledtori}
    Let $(M,\omega,F=(J,H))$ be a system exhibiting a line of curled tori. Let $C_0$ denote the set of critical values of rank $0$ of the system. If $F(M)\backslash C_0$ is simply connected then by making a suitable choice of action coordinates one can associate an affine invariant with the system. Furthermore, no choice of cuts is necessary.
\end{proposition}
\begin{proof}
    Analogous to the first paragraph of the proof of Lemma \ref{l.curledtori}.
\end{proof}
\begin{figure}[ht]
\begin{center}
    \includegraphics[width=1\textwidth]{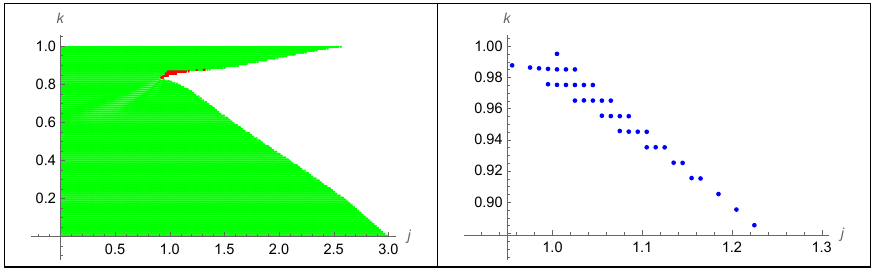}
    \caption{Classical actions of the system $(W_{1,1,2},\omega_{1,1,2},(J,H))$ computed on the joint spectrum for $\hbar=0.01$. On the left the classical actions are computed in the background of the system. Green points correspond to values outside of the image of the generalized flap. Red points correspond to values on the background of the generalized flap. On the right the classical actions are computed on the generalized flap.}
    \label{f.gridmicroflap}
\end{center}
\end{figure}

\section{The affine invariant of an hypersemitoric system}
\label{s.generalhypersemitoric}

%\ke{Logically, this Section should follow directly after Sections 5 and 6. Current Section 7 which discusses non-hypersemitoric systems should follow the discussion here, as showing the direction for future extensions.}\ps{I'm okay with changing it, but I think the layout of the pictures looks weird when we do this.}

\subsection{Introduction}
In this section, we define the affine invariant of a simple hypersemitoric system. First, in Section \ref{s.subsectiongeneralinvariant}, we associate to a simple hypersemitoric system without hyperbolic-regular lines ending in hyperbolic-elliptic values, see Definition \ref{d.hrlineendinginhe}, an affine invariant. The idea is to work ``layer wise'', that is we describe how to apply a procedure to obtain a polytope from considering the background of the system, after appropriate choices of tori are made, and then the procedure can be applied iteratively. We do an approach analogous to that of Section \ref{s.polytopeflap}, to obtain a polytope where the boundary is continuous. Note that different approaches can be made, we chose the one that delivers polytopes with a continuous boundary. Then in Section \ref{s.affinehyperbolicregularlines} we use the results of Section \ref{s.subsectiongeneralinvariant} to define an affine invariant for any simple hypersemitoric system.
%\ps{
%\begin{remark} 
%In the previous sections we did not mention the case of a hypersemitoric system containing a line of hyperbolic-regular values ending in hyperbolic-elliptic values in the boundary of the image of the momentum map. This specific situation will be dealt with in Section \ref{s.affinehyperbolicregularlines}. 
%\end{remark}}

\subsection{The affine invariant for simple hypersemitoric systems without hyper{\-}bolic-regular lines ending in hyperbolic-elliptic values}
\label{s.subsectiongeneralinvariant}
Before defining an affine invariant for any simple hypersemitoric system we first consider a more restricted type of simple hypersemitoric systems. The general case will then be dealt with in Section \ref{s.affinehyperbolicregularlines}. We start by introducing the following definition:
\begin{definition}
\label{d.hrlineendinginhe}
    Let $(M,\omega,F=(J,H))$ be a hypersemitoric system. A \textit{hyperbolic-regular line with hyperbolic-elliptic (briefly H-E) values} $\gamma:[0,1]\rightarrow F(M)$ is a line such that $\gamma(s)$ is a hyperbolic-regular value for all $s\in \ ]0,1[$ and $\gamma(0),\gamma(1)$ are hyperbolic-elliptic values.
\end{definition}
\begin{remark}
\label{r.imageofhyperbolicellitptic}
    If $\gamma$ is a hyperbolic-regular line with H-E values of a hypersemitoric system $(M,\omega,F)$, then $\gamma(0),\gamma(1)\in \partial(F(M))$, see Palmer \& Hohloch \cite[Lemma $2.1$ \& Proposition $4.1$]{hohloch2021extending}. Furthermore, those authors prove that $\gamma(0),\gamma(1)$ lie in the image of a fixed surface of the $S^1$-action induced by $J$. Therefore, they lie in the image of the extremal level sets of $J$. Since lines of hyperbolic-regular values do not have vertical tangencies, see Hohloch \& Palmer \cite[Lemma $4.2$]{hohloch2021extending}, $\gamma(0)$ and $\gamma(1)$ are in the image of different extremal level sets of $J$. One value is in the image of the minimum of $J$ and the other value is in the image of the maximum of $J$.
\end{remark}
For the remainder of this section we assume $(M,\omega,F=(J,H))$ to be a compact simple hypersemitoric system without hyperbolic-regular lines with H-E values. Furthermore, let $A$ be its unfolded momentum domain, i.e., there exists a map $\tilde{F}:M\rightarrow A$ and a projection $\pi:A\rightarrow \mathbb{R}^2$ such that the regular level sets of $\tilde{F}$ are the connected components of the level sets of $F$ and $F=\pi \circ \tilde{F}$. 
\begin{definition}
    A flap, respectively boundary flap, is called \textbf{initial} if its image is not contained in the image of another flap, resp.\ boundary flap. A pleat is called \textbf{initial} if it is not contained in another pleat. 
\end{definition}
Let us introduce some notation:
\begin{itemize}
\item Let $\mathcal{P}_1,...,\mathcal{P}_p$, $p\in \mathbb{N}$, denote the pleats of $F(M)$. Recall that for each pleat $\mathcal{P}_k$ there exist two choices of torus $\mathbb{T}^{\pm}_k$ that allow us to define action coordinates (see Section \ref{s.polytopepleat}). For each $\alpha \in \{+,-\}^{p}$ let $F_{\alpha}$ be the subset of $A$ where in each pleat $\mathcal{P}_k$ we consider the tori $\mathbb{T}_k^{\pi^p_j(\alpha)}$, where $\pi^p_j:\{+,-\}^{p}\rightarrow \{+,-\}$ is the projection onto the $j$-coordinate. The background $B_{\alpha}$ of $F_{\alpha}$ is defined as taking the background of all the flaps, respectively boundary flaps, appearing in $F_{\alpha}$.
\item Let $C_{0}$ be the set consisting of focus-focus values, elliptic-elliptic values occurring in pleats, and elliptic-elliptic values in the images of flaps or boundary flaps whose $x$-projection is in the interior of the image of $J$.
Let $P_0\in M$ denote the set of critical \textit{points} of rank $0$ of $F$. For each $p \in P_0$ let $m_p,n_p$ denote its isotropy weights with respect to $J$. For each $c\in C_0$ let \begin{equation*}
m(c):=\sum\limits_{\{p\in P_0\ \text{with}\ F(p)=c\}}\frac{1}{m_pn_p}.
\end{equation*}
\item  Let $\tilde{C}_{1,\alpha}$ be the rank $1$ critical values occurring in $B_{\alpha}$, $C_{1,\alpha}:=\pi(\tilde{C}_{1,\alpha})$ in $F(M)$, and $C_{0,\alpha}$ be the points of $C_{0}$ occurring in $F_{\alpha}$ that are focus-focus values in the background $B_{\alpha}$, or critical values of rank $0$ in the image of a flap.
\item
For each critical value of rank zero $c\in C_{0}$ occurring in the image of a flap, respectively boundary flap $\mathcal{F}$ in $F_{\alpha}$, that is not a focus-focus value in $B_{\alpha}$, redefine $c$ to be the hyperbolic-regular value in $\gamma_{\tilde{\mathcal{F}}}$ with the same $x$-projection, where $\tilde{\mathcal{F}}$ is the initial flap in $F_{\alpha}$, respectively initial boundary flap, containing $\mathcal{F}$. For a value $c\in C_{0}$ contained in a pleat but not in $C_{0,\alpha}$ redefine $c$ to be the hyperbolic-regular value with the same $x$-projection in the corresponding initial pleat.
The new points have the same value of $m(c)$ as the corresponding point it substituted.   
\item Let $\pi_i:\{-1,1\}^{|C_{0}|}\rightarrow \{-1,1\}$ be the projection onto the $i$-th coordinate. 
\item Let $\mathcal{T}$ be the subgroup of $\AGL(2,\mathbb{Z})$ that leaves a vertical line, with orientation, invariant.
\item  Let $\mathcal{F}_{u_1},\dots,\mathcal{F}_{u_n}$, with $u_{i}\in \mathbb{N}$, for $1\leq i\leq n$,
%$\{\mathcal{F}_{\beta}\},\  \beta\in B$ be the set of 
be the unbounded flaps/boundary flaps of the system. For each unbounded flap/boundary flap $\mathcal{F}_{u_i}$, define the set $U_{u_i}:=\pi(\mathcal{F}_{u_i})\backslash \pi(\mathcal{B}_{u_i})$, where $\mathcal{B}_{u_i}$ is the background of each flap/boundary flap $\mathcal{F}_{u_i}$. Finally, define the set $U_{\mathcal{F}}:=\bigcup_{i=1}^{n}U_{u_i}$.
\end{itemize}
\begin{theorem}
\label{t.polytopegeneral}
    Using the notation from above let $(M,\omega,F=(J,H))$ be a compact simple hypersemitoric system without hyperbolic-regular lines with H-E values. Fix an $\alpha\in \{+,-\}^{p}$. 
    
    For each $\vec{\epsilon}\in \{-1,1\}^{|C_{0}|}$ there exists a unique continuous map, up to composition with an element of $\mathcal{T}$,
    \begin{equation*}
    f_{\vec{\epsilon},\alpha}:F(M)\backslash(C_{1,\alpha}\cup \ U_{\mathcal{F}})\rightarrow \mathbb{R}^2
    \end{equation*}
    such that 
    \begin{itemize}
        \item $f_{\vec{\epsilon},\alpha}|_{F(M)\backslash\left(C_{1,\alpha}\cup \ U_{\mathcal{F}} \  \cup\left(\bigcup_{i=1}^{|C_{0}|}l_{c_i}^{\pi_i(\vec{\epsilon})}\right)\right)}$ is smooth;
        \item $f_{\vec{\epsilon},\alpha}|_{F(M)\backslash\left(C_{1,\alpha}\cup \ U_{\mathcal{F}} \  \cup\left(\bigcup_{i=1}^{|C_{0}|}l_{c_i}^{\pi_i(\vec{\epsilon})}\right)\right)}$ is affine;
        \item $f_{\vec{\epsilon},\alpha}$ preserves $J$, i.e., is of the form $f_{\vec{\epsilon},\alpha}(x,y)=(x,f^{(2)}_{\vec{\epsilon},\alpha}(x,y))$; 
        \item %If $c_i$ is admissible  for $\vec{\epsilon}$,
        For any $c\in \text{int}\left(l_{c_i}^{\pi_i(\vec{\epsilon})}\right)$, with $c_i\in C_{0}$:
        \begin{equation}
        \label{eq.derivativerelationgeneral}
            \lim_{\begin{smallmatrix}(x,y)\rightarrow c \\ x<\pi_x(c_i) \end{smallmatrix}} df_{\vec{\epsilon},\alpha}(x,y) = M_{c_i} \lim_{\begin{smallmatrix}(x,y)\rightarrow c\\ x>\pi_x(c_i) \end{smallmatrix}} df_{\vec{\epsilon},\alpha}(x,y),
            \end{equation}
             where 
             \begin{equation*}
                 M_{c_i}= \begin{bmatrix}
                      1 & 0 \\
                  k(c_i) & 1
                \end{bmatrix}
            \end{equation*}
and 
$k(c_i)= \sum_{j}\pi_j(\vec{\epsilon})m(c_j)$,
where the sum runs over all $c_j\in C_{0}$ satisfying $c_i\in l_{c_j}^{\pi_j(\vec{\epsilon})}$.
    \end{itemize}
%\item  \kedel{If $C_{0,\alpha}=\emptyset$, there exists a unique map $f_{\alpha}:F(M)\backslash (C_{1,\alpha}\cup \ U_{\mathcal{F},\alpha})\rightarrow \mathbb{R}^2$, up to composition with an element of $\mathcal{T}$, such that $f_{\alpha}$ is smooth, affine and preserves $J$, i.e., is of the form $f_{\alpha}(x,y)=(x,f^{(2)}_{\alpha}(x,y))$. }\ps{Isn't there a problem here when there are elliptic-elliptic values in the boundary of pleats? Kill this part, we can't ensure continuity sometimes without making cuts...}
%\end{enumerate}
\end{theorem}
\begin{proof}
    The proof expands the ideas of the proofs of Theorem \ref{t.polytopeflap}, Theorem \ref{t.polytopepleat} and Theorem $3.8$ of \vungoc\ \cite{vu2007moment}.

    \textit{Step 1:} If the system is semitoric, then the result follows from \vungoc\ \cite[Theorem $3.8$]{vu2007moment}.

    \textit{Step 2:} If no pleats are present in the system, then the results follow completely analogous to the proof of Theorem \ref{t.polytopenflaps}, if $\pi_x(C_{0})$ is single valued, and Theorem \ref{t.polytopeflap} otherwise.

    \textit{Step 3:} If there are pleats but these do not contain critical values of rank $0$ in its interior, then the proof is analogous to the proof of Theorem \ref{t.polytopeflap} and Theorem \ref{t.polytopepleat}.

    \textit{Step 4:} If there are pleats present in the system containing critical values of rank $0$ we proceed analogously to the proof of Theorem \ref{t.polytopeflap}. Consider the set $A:=B_{\alpha}\backslash (\cup_{c_i\in C_0}l^{\pi_i(\vec{\epsilon})}_{c_i} \cup C_{1,\alpha})$. The set $A$ has finitely many simply connected components. In each connected component we take action coordinates as follows:
    \begin{itemize}
    \item We are able to fit them together in order for the resulting function $f_{\vec{\epsilon},\alpha}$ to be continuous,
    \item the height of the resulting polytope is given by the volume of $J^{-1}(x)/S^1$. One can visualize it in the affine invariant as the difference between the highest and the lowest point outside of the image of the pleats intersected with the complement of the boundary of the momentum map; in the example of Figure \ref{f.panelswallowtail}, it is the difference between the highest green and lowest green value in a vertical line (and not the difference between green and red values), except for the intersection of the pleat with the boundary of the momentum map.
    \item Equation \eqref{eq.derivativerelationgeneral}
    holds. For $c$ such that $c\in l_{c_j}^{\pi_j(\vec{\epsilon})}$ only for $c_j\in C_{0}$ in the same connected component that are either 
    a focus-focus value in $B_{\alpha}$ or the hyperbolic-regular value $c_j$ in $\gamma_{\mathcal{F}}$ for a flap $\mathcal{F}$ induced by the unique rank $0$ critical value on the image of the flap $\mathcal{F}$ in $F_{\alpha}$, Equation \eqref{eq.derivativerelationgeneral} follows from Proposition \ref{p.monodromy}, analogous to what is done in Theorem \ref{t.polytopenflaps} and \cite[Theorem $3.8$]{vu2007moment}.
    \end{itemize}
    The above conditions determine $f_{\vec{\epsilon},\alpha}$ due to the formula of the Duistermaat-Heckmann measure, see Section \ref{s.duistermaatmeasuregeneral}.
\end{proof}

\begin{definition}
    Given a compact simple hypersemitoric system $(M,\omega,F=(J,H))$ without hyperbolic-regular lines with H-E values a representative of the affine invariant is defined by the collective images of iteratively applying Theorem \ref{t.polytopegeneral} to each layer of background/flap/boundary flap/pleat found in the system for a choice of $\vec{\epsilon}$ in each layer. If no critical values are present in the layer no choice of $\vec{\epsilon}$ is required. The analogous of Section \ref{s.grouporbitnflaps} and Section \ref{s.orbitpolytope} follows, i.e., the different choices of $\vec{\epsilon}$ can be interpreted as a group orbit. 
\end{definition}

\subsection{Elliptic-regular lines with boundary values and their properties}
Recall that $\gamma_{HR}$ is a hyperbolic-regular line with H-E values of the system $(M,\omega,F)$ if $\gamma_{HR}:[0,1]\rightarrow F(M), \ s\mapsto \gamma_{HR}(s)$ is such that $\gamma_{HR}(s)$ is an hyperbolic-regular value for all $s\in \ ]0,1[$ and $\gamma_{HR}(0),\gamma_{HR}(1)$ are hyperbolic-elliptic values. Moreover, recall $\{\gamma_{HR}(0),\gamma_{HR}(1)\}\in \partial(F(M))$, see Remark \ref{r.imageofhyperbolicellitptic} for more details. Analogously:
%we define an \textit{elliptic-regular line $\gamma_{ER}\subset F(M)$ with boundary values}, with the extra requirements that its preimage is not a subset of a fixed surface of the $S^1$-action induced by $J$ and its endpoints are in $\partial(F(M))$:
\begin{definition}
    Let $(M,\omega,F=(J,H))$ be a hypersemitoric system. An \textit{elliptic-regular line $\gamma_{ER}:[0,1]\rightarrow F(M)$ with boundary values} is a line such that $\gamma_{ER}(s)$ is an elliptic-regular value for all $s\in \ ]0,1[$, $\pi_x(\gamma(0)),\pi_x(\gamma(1))\in \partial(J(M))$, where $\pi_x:(x,y)\mapsto x$,  and $F^{-1}(\gamma_{ER})$ is not a subset of a fixed surface of the $S^1$-action induced by $J$. 
\end{definition}
\begin{remark}
\label{r.differentendpointselliptic}
    An elliptic-regular line with boundary values does not have vertical tangencies, see Lemma \ref{l.nonverticaltangenciesellitpic}. Therefore, for an elliptic-regular line with boundary values $\gamma_{ER}:[0,1]\rightarrow F(M)$, $\gamma(0)$ and $\gamma(1)$ have different $J$-values.
\end{remark}
Note the following: 
\begin{lemma}
\label{l.linepairs}
    Let $(M,\omega,F=(J,H))$ be a simple hypersemitoric system and $\gamma_{HR}\subset F(M)$ a hyperbolic-regular line with H-E values. Then there exists another hyperbolic-regular line with H-E values $\tilde{\gamma}_{HR}$ or an elliptic-regular line with boundary values $\tilde{\gamma}_{ER}$  in $F(M)$.
\end{lemma}
\begin{proof}
    The idea is to use the Morse inequalities on a reduced level set $J^{-1}(j)/S^1$. First we recall the following fact: a hyperbolic-regular value $(j_0,h_0)$, where $j_0$ is a regular value of $J$, of a hypersemitoric system comes from either a hyperbolic-regular line with H-E values, a pleat, a flap or boundary flap, see Hohloch \& Palmer \cite[Corollary $4.3$]{hohloch2021extending}. In the case of a pleat, a flap, or boundary flap, the hyperbolic-regular value comes paired with an elliptic-regular value with the same $j_0$ and lying in the image of the same pleat, flap, or boundary flap.
    Now let $j$ be a regular value of $J$, with the property that all parabolic values $c$ of the system are such that $\pi_x(c)\neq j$, where $\pi_x(j,h):=j$. Since $J^{-1}(j)/S^1 $ is a $2$-dimensional closed connected orientable surface its Euler characteristic satisfies $\chi(J^{-1}(j)/S^1)=2-2g$.  Using the Morse inequalities for the reduced Hamiltonian $H^{j}_{red}:J^{-1}(j)/S^1\rightarrow \mathbb{R}$ we obtain 
    \begin{align*}
        & \#\{\text{elliptic-points of $H^{j}_{red}$}\}- \#\{\text{ hyperbolic-points of $H^{j}_{red}$}\}=2-2g.
    \end{align*}
    Using the above equality, the fact that $(M,\omega,F=(J,H))$ is a simple hypersemitoric system, and the observation that in a pleat, flap, or boundary flap, the critical points come in pairs, we conclude that associated with a hyperbolic-regular point of an hyperbolic-regular line with H-E values, there exists another hyperbolic-regular or elliptic-regular point. Since the set of regular values $j$ of $J$ with the property that $\pi_x(c)\neq j$ for all parabolic values $c$ of the system $(M,\omega,F)$ is dense over $\pi_x(\gamma)$ we conclude the desired result.
\end{proof}
Below we give examples of systems having hyperbolic-regular lines with H-E values in their momentum map images:

\begin{example}
\label{e.heartsphere}
    Consider the standard height function on the $2$-sphere $(S^2,\omega_{S^2},h)$. In particular $h$ induces an $S^1$-action. Furthermore, let $(S^2_{hs},\omega_{hs},h_{hs})$ be the heart-shaped sphere with $4$ elliptic points $p_1,p_2,p_3,p_4$ and $2$ hyperbolic points $q_1,q_2$ satisfying: $h_{hs}(p_1)<h_{hs}(q_1)<h_{hs}(p_2)<h_{hs}(q_2)<h_{hs}(p_3)<h_{hs}(p_4)$, see Figure \ref{f.heightfunctions}. Then
    \begin{align*}
    (M,\omega,F):=(S^2\times S^2_{hs},\omega_{S^2}\oplus \omega_{hs}, h\times h_{hs})
    \end{align*}
    is a simple hypersemitoric system which has $2$ hyperbolic-regular lines with H-E values $\gamma_{q_1},\gamma_{q_2}$, where $\gamma_{q_i}:=F(S^2\times \{q_i\})$ for $i=1,2$, and $4$ elliptic-regular lines with boundary values $\gamma_{p_1},\gamma_{p_2},\gamma_{p_3},\gamma_{p_4}$, where $\gamma_{p_i}:=F(S^2\times \{p_i\})$ for $1\leq i\leq 4$. Furthermore, the set $F(M) \backslash \{\gamma_{q_1},\gamma_{q_2},\gamma_{p_2},\gamma_{p_3}\}$ has $5$ connected components, see Figure \ref{f.momentumaps}.
\end{example}

\begin{example}
\label{e.torus}
    Consider the embedding of the torus $\mathbb{T}^2$ into $\mathbb{R}^3$ given by 
    \begin{align*}
    (\theta_1,\theta_2)\mapsto ((2+\cos(\theta_1))\cos(\theta_2), (2+\cos(\theta_1))\sin(\theta_2),\sin(\theta_1))
    \end{align*}
    and the symplectic form $\omega_{\mathbb{T}^2}$ on $\mathbb{T}^2$ induced by the volume form $dx\wedge dy\wedge dz$ on $\mathbb{R}^3$. We consider the system 
    \begin{align*}
    (M,\omega,F):=(S^2\times \mathbb{T}^2, \omega_{S^2}\oplus \omega_{\mathbb{T}^2},h\times h_T)
    \end{align*}
    where $h_T(\theta_1,\theta_2):=(2+\cos(\theta_1))\sin(\theta_2)$, see Figure \ref{f.heightfunctions}, and $h$ is the standard height function on $S^2$. This is a simple hypersemitoric system. The function $h_T$ has two hyperbolic points $q_1,q_2$ with $h_T(q_1)<h_T(q_2)$. Furthermore, the function $h_T$ has $2$ elliptic points $p_1,p_2$, given by the maximum and minimum of the function, respectively. The system $(M,\omega,F)$ has two hyperbolic-regular lines with H-E values: $\gamma_{q_1}:=F(S^2\times \{q_1\})$ and $\gamma_{q_2}:=F(S^2\times \{q_2\})$. The set $F(M)\backslash \{\gamma_1,\gamma_2\}$ has $3$ connected components, see Figure \ref{f.momentumaps}.
\end{example}

\begin{figure}[ht]
\begin{center}
    \includegraphics[scale=0.3]{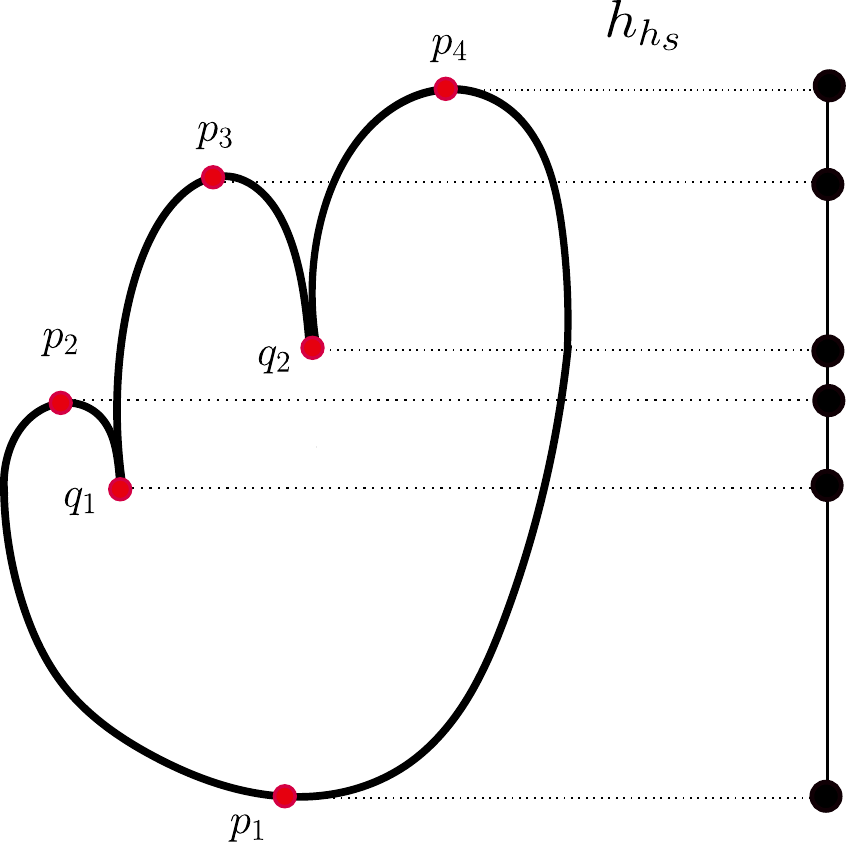}
\hspace{2cm}
    \includegraphics[scale=0.35]{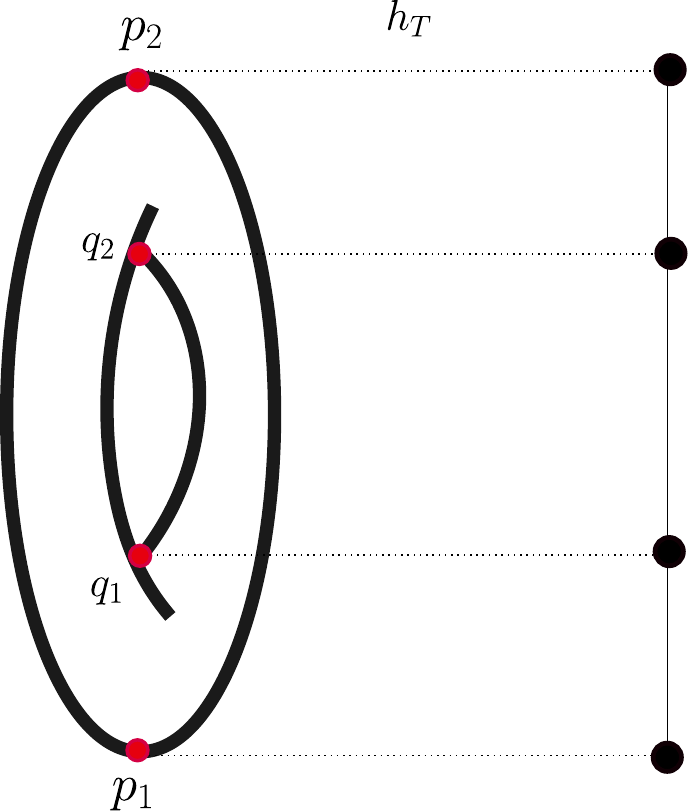}
    \caption{On the right the height function $h_{hs}$ on the heart-shaped sphere $S^2$. On the left the height function $h_T$ on the two torus $\mathbb{T}^2$.}
    \label{f.heightfunctions}
\end{center}
\end{figure}

\begin{figure}[ht]
\begin{center}
    \includegraphics[scale=0.3]{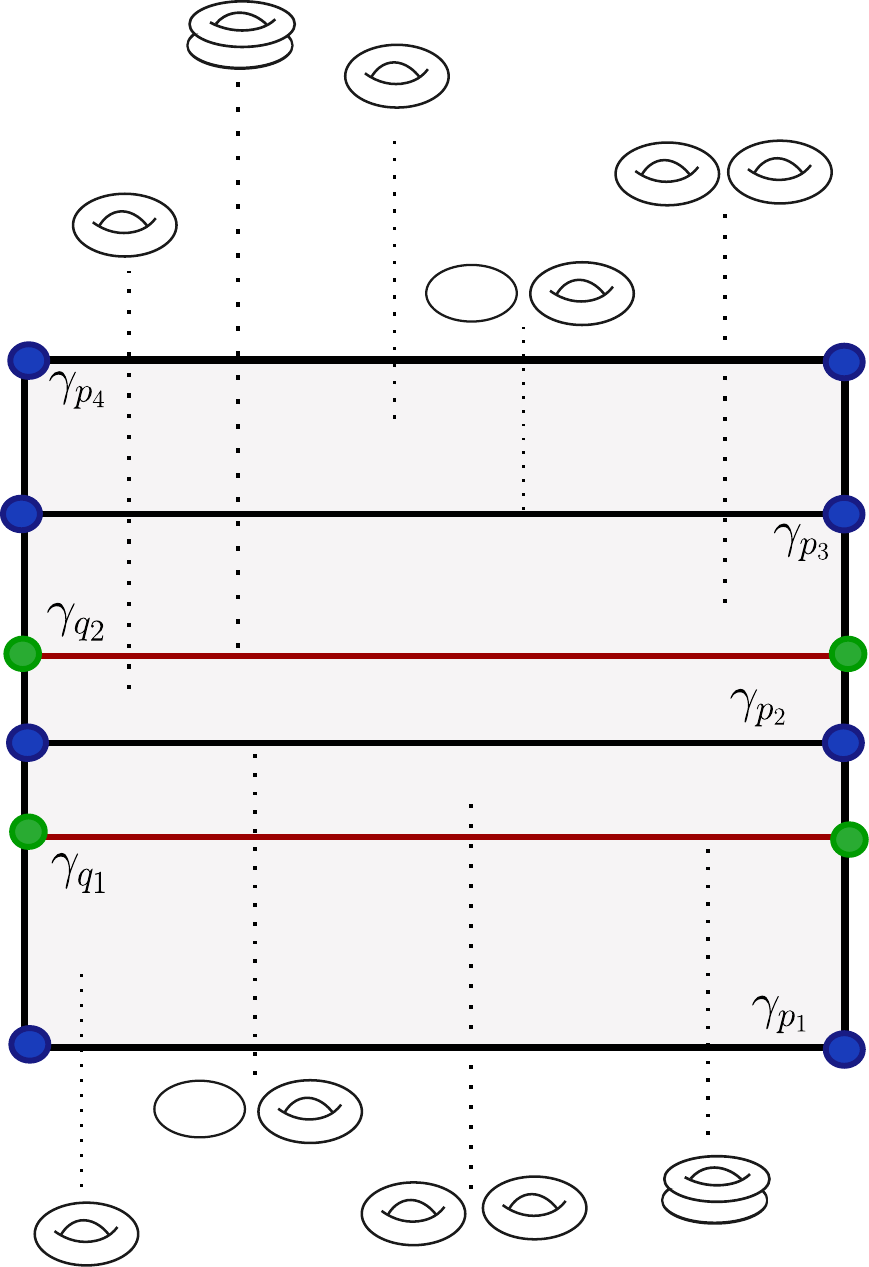}
\hspace{2cm}
    \includegraphics[scale=0.31]{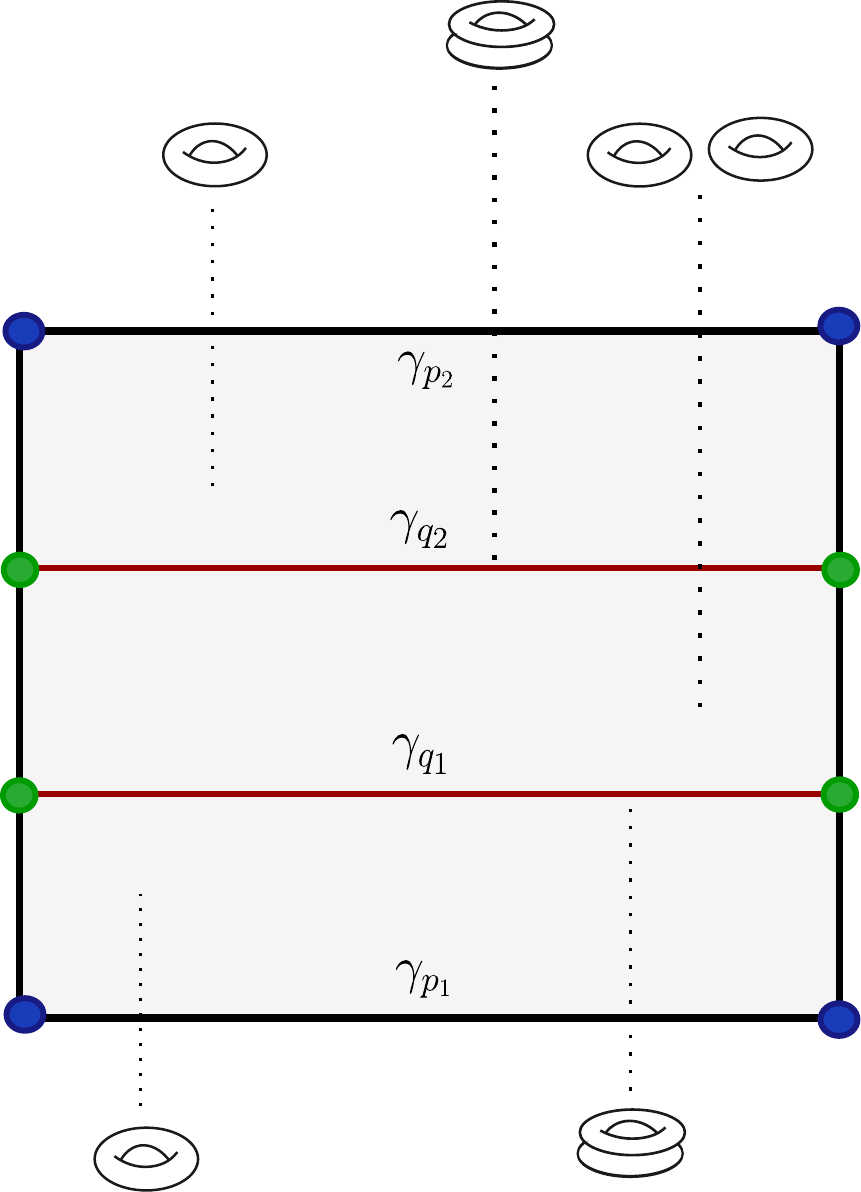}
    \caption{On the left is the momentum map image of Example \ref{e.heartsphere}. On the right is the momentum map image of Example \ref{e.torus}. Hyperbolic-regular values are drawn in red, hyperbolic-elliptic values are drawn in green, elliptic-elliptic values are drawn in blue, and elliptic-regular values are drawn in black.}
    \label{f.momentumaps}
\end{center}
\end{figure}

\subsection{Affine invariant for a simple hypersemitoric system}
\label{s.affinehyperbolicregularlines}
We are now ready to define an affine invariant for any simple hypersemitoric system. Let $(M,\omega,F=(J,H))$ be a simple hypersemitoric system. Let $\{\gamma_1,\cdots,\gamma_k\}$ be the set of hyperbolic-regular lines with H-E values or elliptic-regular lines with boundary values that satisfy $\gamma_i(]0,1[)\subset \text{int}(F(M))$ for $1\leq i\leq k$. Now consider the set $F(M)\backslash \{\gamma_1,\cdots ,\gamma_k\}$. Recall that the hyperbolic-elliptic values lie in the boundary of $F(M)$, namely in the image of the minimum and maximum of $J$. Moreover, for each hyperbolic-regular line with H-E values or elliptic-regular line with boundary values, its endpoints have different $J$-values, see Remark \ref{r.imageofhyperbolicellitptic} and Remark \ref{r.differentendpointselliptic}. Therefore, the set $\tilde{F}(M):=F(M)\backslash \{\gamma_1,\cdots ,\gamma_k\}$ is disconnected, if $\{\gamma_1,\cdots, \gamma_k\}\neq \emptyset$. Let $U$ be a connected component of the set $\tilde{F}(M)$. Since the fibers over $U$ may be disconnected, the choice of tori over $U$ to each we want to apply Theorem \ref{t.polytopegeneral} is not unique. Applying Theorem \ref{t.polytopegeneral} for all possible continuous choices of tori over the set of regular values of the connected components of $\tilde{F}(M)$ we obtain:

\begin{definition}
    Given a compact hypersemitoric system $(M,\omega,F=(J,H))$ a representative of the affine invariant is defined by the collective images of iteratively applying Theorem \ref{t.polytopegeneral} to each layer of background/flap/boundary flap/pleat  over all continuous possible choices of tori over the set of regular values of the connected components of $\tilde{F}(M)$, with a choice of $\vec{\epsilon}$ in each layer. If no critical values are present in the layer no choice of $\vec{\epsilon}$ is required. The analogous of Section \ref{s.grouporbitnflaps} and Section \ref{s.orbitpolytope} follows, i.e., the different choices of $\vec{\epsilon}$ can be interpreted as a group orbit. 
\end{definition}
\begin{remark}
\item In order to consider an invariant for a noncompact simple hypersemitoric system $(M,\omega,F=(J,H))$ one possible way is to work on the compact subsets $J^{-1}(]-\infty,j])$ for $j\in \mathbb{R}$, and let $j\rightarrow \infty$. 
\end{remark}
\subsection{Examples}
\label{s.examplesgeneralhypersemitoric}
Now we illustrate Theorem \ref{t.polytopegeneral} with three examples. The first example is a system with two focus-focus values on a flap. The second example is a system with two focus-focus values outside of a flap. And the third example is a flap with two elliptic-elliptic values inside of another flap.

Let $W_{1,1,2}$ be the Hirzebruch surface defined in Section \ref{s.quantizationhirzebruch}. Let
\begin{equation*}
\begin{cases}
J(z_1,z_2,z_3,z_4):=\frac{1}{2}(|z_2|^2+|z_3|^2),\\ R(z_1,z_2,z_3,z_4):=\frac{1}{2}|z_3|^2,\\
X(z_1,z_2,z_3,z_4):=\Re(z_1z_2\overline{z_3}z_4).
\end{cases}
\end{equation*}
Then for $a,b,c\in \mathbb{R}$ we consider the integrable system $(W_{1,1,2},\omega_{1,1,2},F_{a,b,c}=(J,H_{a,b,c}))$ where 
\begin{equation*}
    H_{a,b,c}:=aR+\frac{X}{\sqrt{2}}+bR^2+cR^3.
\end{equation*}
The quantization of this system is straightforward and analogous to what is done in Section \ref{s.quantizationhirzebruch} and Appendixes \ref{s.quantizationflap2ellitptic}, \ref{s.quantizationpleat}, \ref{s.quantizationcurledtori}. We have 
\begin{equation*}
    \hat{J}=\hbar(z_2\frac{\partial}{\partial z_2}+z_3\frac{\partial}{\partial z_3}+1), \quad \hat{R}=\hbar\left(z_3\frac{\partial}{\partial z_3}+\frac{1}{2}\right), \quad \hat{(R^n)}=(\hat{R})^n
\end{equation*}
for $n\geq1$, and 
\begin{equation*}
    \hat{X}=2\hbar^2\left(\frac{\partial}{\partial z_1}\frac{\partial}{\partial z_2}z_3\frac{\partial}{\partial z_4}+z_1z_2\frac{\partial}{\partial z_3}z_4\right).
\end{equation*}
For $(a,b,c)=(20,-35,17)$ the system exhibits a flap with two focus-focus values in its interior, see Figure \ref{f.bifurcationffinsideflap}. For $(a,b,c)=(\frac{20}{\sqrt{2}},\frac{-35}{\sqrt{2}},\frac{17}{\sqrt{2}})$ the system exhibits a flap and two focus-focus values outside of the flap, see Figure \ref{f.bifurcationffoutsideflap}. For $(a,b,c)=(100,-200,110)$ the system exhibits a flap with two elliptic-elliptic values inside of another flap, see Figure \ref{f.bifurcationflaponflap}.  
\begin{figure}[ht]
\begin{center}
    \includegraphics[scale=0.6]{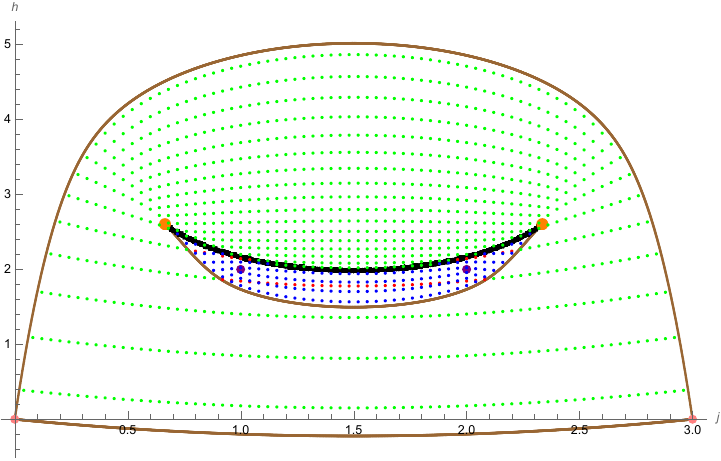}
    \caption{Bifurcation diagram for the system $(W_{1,1,2},\omega_{1,1,2},F_{a,b,c})$ with $(a,b,c)=(20,-35,17)$ together with the joint spectrum for $\hbar=\frac{1}{25}$. The purple points correspond to  focus-focus values.  The black points correspond to hyperbolic-regular values. The orange points correspond to parabolic values. The brown points correspond to elliptic-regular values. The pink points correspond to elliptic-elliptic values. Green points correspond to values of the joint spectrum outside of the image of the flap. Red points correspond to values of the joint spectrum on the background of the flap. Notice that the red dots continue the line of green dots when we pass "underneath" the flap. Blue points correspond to values of the joint spectrum on the flap. }
    \label{f.bifurcationffinsideflap}
\end{center}
\end{figure}
\begin{figure}[ht]
\begin{center}
    \includegraphics[scale=0.6]{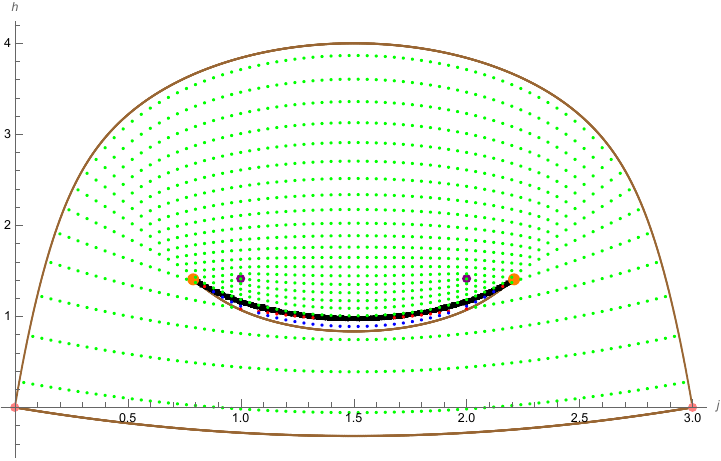}
    \caption{Bifurcation diagram for the system $(W_{1,1,2},\omega_{1,1,2},F_{a,b,c})$ with $(a,b,c)=(\frac{20}{\sqrt{2}},\frac{-35}{\sqrt{2}},\frac{17}{\sqrt{2}})$ together with the joint spectrum for $\hbar=\frac{1}{25}$. The purple points correspond to  focus-focus values.  The black points correspond to hyperbolic-regular values. The orange points correspond to parabolic values. The brown points correspond to elliptic-regular values. The pink points correspond to elliptic-elliptic values. The black points correspond to hyperbolic-regular values. Green points correspond to values of the joint spectrum outside of the image of the flap. Red points correspond to values of the joint spectrum on the background of the flap. Blue points correspond to values of the joint spectrum on the flap.}
    \label{f.bifurcationffoutsideflap}
\end{center}
\end{figure}

\begin{figure}[ht]
\begin{center}
    \includegraphics[scale=0.6]{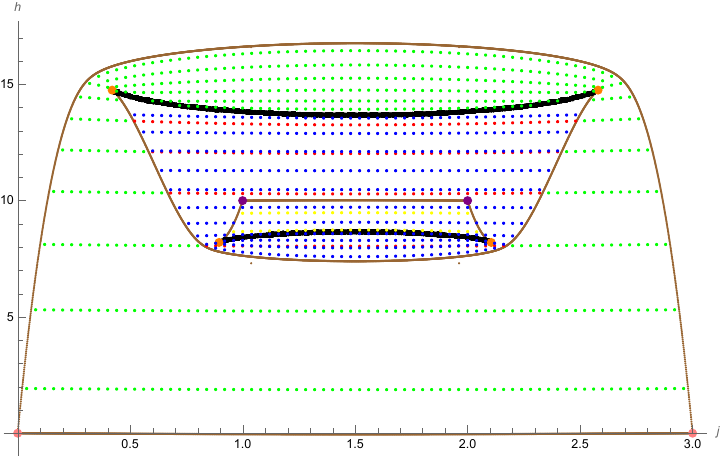}
    \caption{Bifurcation diagram for the system $(W_{1,1,2},\omega_{1,1,2},F_{a,b,c})$ with $(a,b,c)=(100,-200,110)$ together with its joint spectrum for $\hbar=\frac{1}{25}$. The purple points correspond to elliptic-elliptic values. The black points correspond to hyperbolic-regular values. The orange points correspond to parabolic values. The brown points correspond to elliptic-regular values. The pink points correspond to elliptic-elliptic values in the boundary. Green points correspond to values of the joint spectrum outside of the image of the initial flap. Red points correspond to values of the joint spectrum on the background of the initial flap. Blue points correspond to values of the joint spectrum on the initial flap. Yellow points correspond to values of the joint spectrum on the smaller flap.}
    \label{f.bifurcationflaponflap}
\end{center}
\end{figure}

Applying analogous methods to the ones in Appendixes \ref{s.actionJaynesCummings}, \ref{s.actionangle2llipticflap}, \ref{s.actionpleat} and \ref{s.actioncurledtori} to compute the actions of systems, we obtain the following results:

\begin{example}
    Four representatives of the affine invariant for the hypersemitoric system $(W_{1,1,2},\omega_{1,1,2},F_{a,b,c})$ with $(a,b,c)=(20,-35,17)$ given by Theorem \ref{t.polytopegeneral} are shown in Figure \ref{f.fullinvariantffinsideflap}.
\end{example}

\begin{example}
    Four representatives of the affine invariant for the hypersemitoric system $(W_{1,1,2},\omega_{1,1,2},F_{a,b,c})$ with $(a,b,c)=\frac{1}{\sqrt{2}}(20,-35,17)$ given by Theorem \ref{t.polytopegeneral} are shown in Figure \ref{f.fullinvarianffoutsideflap}.
\end{example}

\begin{example}
    Four representatives of the affine invariant for the hypersemitoric system $(W_{1,1,2},\omega_{1,1,2},F_{a,b,c})$ with $(a,b,c)=(100,-200,110)$ given by Theorem \ref{t.polytopegeneral} are shown in Figure \ref{f.fullinvariantflaponflap}.
\end{example}
\begin{figure}[ht]
\begin{center}
    \includegraphics[scale=1.4]{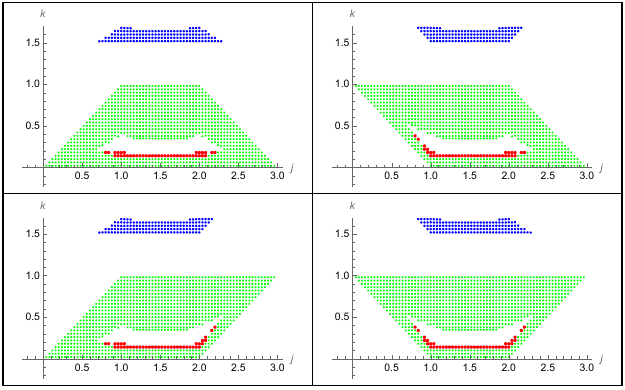}
    \caption{Representatives of the affine invariant given by Theorem \ref{t.polytopegeneral} applied to the joint spectrum of the system $(W_{1,1,2},\omega_{1,1,2},F_{a,b,c})$ with $(a,b,c)=(20,-35,17)$ and $\hbar=\frac{1}{25}$. The green points correspond to the classical actions computed outside of the image of the flap. The red points correspond to the classical actions computed on the background of the flap. The blue points correspond to the classical actions computed on the flap, which we chose to plot away from its original position to make the picture clearer. The first layer of the system is its background. The second layer of the system is the flap. On the top left the choice of $\vec{\epsilon}$ for the first layer is $(+1,+1)$ and the choice of $\vec{\epsilon}$ for the second layer is $(+1,+1)$.  On the top right the choice of $\vec{\epsilon}$ for the first layer is $(-1,+1)$ and the choice of $\vec{\epsilon}$ for the second layer is $(-1,-1)$.  On the bottom left the choice of $\vec{\epsilon}$ for the first layer is $(+1,-1)$ and the choice of $\vec{\epsilon}$ for the second layer is $(+1,-1)$.  On the bottom right the choice of $\vec{\epsilon}$ for the first layer is $(-1,-1)$ and the choice of $\vec{\epsilon}$ for the second layer is $(-1,+1)$.}
    \label{f.fullinvariantffinsideflap}
\end{center}
\end{figure}

\begin{figure}[ht]
\begin{center}
    \includegraphics[scale=1.4]{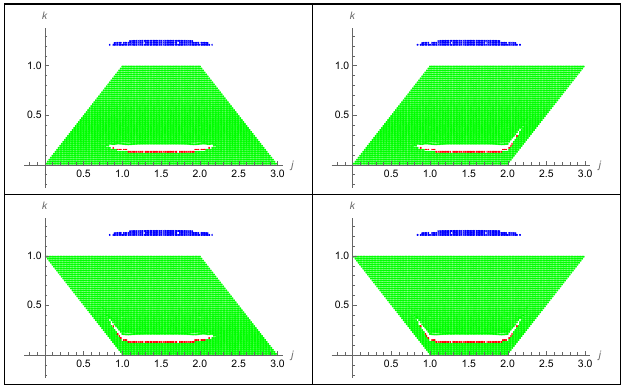}
    \caption{Representatives of the affine invariant given by Theorem \ref{t.polytopegeneral} applied to the joint spectrum of the system $(W_{1,1,2},\omega_{1,1,2},F_{a,b,c})$ with $(a,b,c)=(\frac{20}{\sqrt{2}},-\frac{35}{\sqrt{2}},\frac{17}{\sqrt{2}})$ and $\hbar=\frac{1}{50}$. The green points correspond to the classical actions computed outside of the image of the flap. The red points correspond to the classical actions computed on the background of the flap. The blue points correspond to the classical actions computed on the flap, which we chose to plot away from its original position to make the picture clearer. The first layer of the system is its background. The second layer is the flap. No choice of $\vec{\epsilon}$ is necessary for the second layer. On the top left the choice of $\vec{\epsilon}$ for the first layer is $(+1,+1)$. On the top right the choice of $\vec{\epsilon}$ for the first layer is $(+1,-1)$. On the bottom left the choice of $\vec{\epsilon}$ for the first layer is $(-1,+1)$. On the bottom right the choice of $\vec{\epsilon}$ for the first layer is $(-1,-1)$. }
    \label{f.fullinvarianffoutsideflap}
\end{center}
\end{figure}

\begin{figure}[ht]
\begin{center}
    \includegraphics[scale=1.4]{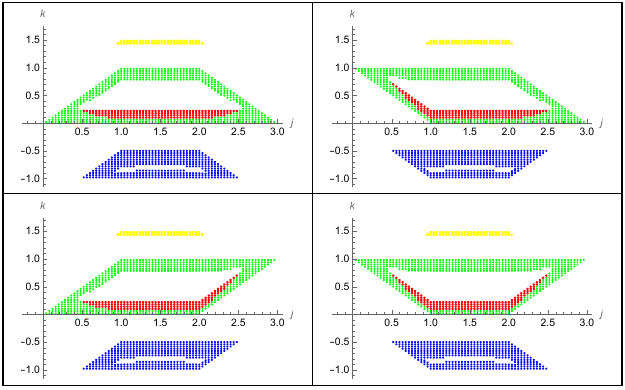}
    \caption{Representatives of the affine invariant given by Theorem \ref{t.polytopegeneral} applied to the joint spectrum of the system $(W_{1,1,2},\omega_{1,1,2},F_{a,b,c})$ with $(a,b,c)=(100,200,-110)$ and $\hbar=\frac{1}{25}$. The green points correspond to the classical actions computed outside of the image of the initial flap. The red points correspond to the classical actions computed on the background of the initial flap. The blue points correspond to the classical actions on the initial flap, which we chose to plot away from its original position to make the picture clearer. The yellow points correspond to the classical actions computed on the smaller flap, which we chose to plot away from its original position to make the picture clearer. The first layer of the system is the layer that contains the background of the initial flap. The second layer of the system is the initial flap. The third layer of the system is the smaller flap. No choice of $\vec{\epsilon}$ is necessary for the third layer. On the top left the choice of $\vec{\epsilon}$ for the first layer is $(+1,+1)$ and the choice of $\vec{\epsilon}$ for the second layer is $(+1,+1)$. On the top right the choice of $\vec{\epsilon}$ for the first layer is $(-1,+1)$ and the choice of $\vec{\epsilon}$ for the second layer is $(-1,-1)$. On the bottom left the choice of $\vec{\epsilon}$ for the first layer is $(+1,-1)$ and the choice of $\vec{\epsilon}$ for the second layer is $(+1,-1)$. On the bottom right the choice of $\vec{\epsilon}$ for the first layer is $(-1,-1)$ and the choice of $\vec{\epsilon}$ for the second layer is $(-1,+1)$. }
    \label{f.fullinvariantflaponflap}
\end{center}
\end{figure}
\clearpage
\appendix
\section{Classical actions}
\label{s.classicalactions}
In this Appendix we give detailed instructions on how to compute the classical actions of several systems introduced throughout the main body of the article. 
\subsection{Classical actions for the modified Jaynes-Cummings model}
\label{s.actionJaynesCummings}
Recall the system $(M,\omega,(J,H+G))$ from Section \ref{s.jaynescummingsmodel}. $J$ generates an $S^1$-action for which the function $z$ is invariant. In addition the functions
\begin{align*}
    L(u,v,x,y,z)=ux+vy,\quad T(u,v,x,y,z)=vx-uy
\end{align*} 
are invariant by the $S^1$-action. They satisfy the relation
\begin{align*}
    L^2+T^2=2(J-z)(1-z^2)
\end{align*}
with $-1\leq z\leq \min\{J,1\}$. For fixed $J=j$ these relations define a surface of revolution $Q_j$. In terms of the invariant functions $L,T,z$ we have
\begin{align*}
    H+G=\frac{1}{2}L+\gamma z^2.
\end{align*}

For a regular value $(J,H+G)=(j,h)$ we consider the corresponding fiber which is a smooth two-dimensional torus $T^2_{j,h}$ or the union of two such tori.
The reduced fiber $T^2_{j,h} /  S^1$ is a circle $S_{j,h}$ or the union of two circles.

Introduce cylindrical coordinates $(z,\theta)$ on $S^2$, with $x = (1-z^2)^{1/2} \cos\theta$ and $y= (1-z^2)^{1/2} \sin\theta$, and symplectic polar coordinates $(I,\phi)$ on $\mathbb{R}^2$ with $u = \sqrt{2I} \cos\phi$, $v = \sqrt{2I} \sin\phi$. 
Then the symplectic form becomes
\[ \omega = d\phi \wedge dI \oplus d\theta \wedge dz
= -d(I d\phi + zd\theta) =: -d\alpha. \]
Moreover,
\[ H+G= \frac12 (2I(1-z^2))^{1/2} \cos(\phi-\theta) + \gamma z^2, \]
and
\[ J = I + z. \]

To obtain the classical actions we integrate $\alpha$ along a homology cycle.
Notice that on a representative $\gamma_J$ of the homology cycle generated by the flow of $X_J$ we have that $I$, $z$ are constant while $\phi$, $\theta$ increase by $2\pi$ and thus the corresponding action is 
\[ \frac{1}{2\pi} \int_{\gamma_J} I d\phi + zd\theta = I + z = J. \]

To construct a second homology cycle $\gamma$ for our regular value $(j,h)$ we proceed as follows.
We first identify a point on the reduced space $Q_j=J^{-1}(j)/S^1$ with $T=0$ such that
\[ h = \frac12 L + \gamma z^2, \]
by solving the equation
\begin{equation}
\label{eq.reduced}
 4 (h-\gamma z^2)^2 = 2(j-z)(1-z^2), 
\end{equation}
with the restriction 
\[ -1 \le z \le \min\{j,1\}. \]
For each solution we can compute the corresponding value of $L$ from $L=2(h-\gamma z^2)$. 
Moreover denote by $(H+G)_{red}$ the reduced Hamiltonian of $H+G$ on $Q_j$ and by $\phi_t^{(H+G)_{red}}$ its Hamiltonian flow. We compute $\dot T:=\frac{d}{dt}(T\circ \phi_t^{(H+G)_{red}})$ at each of the solutions of Equation \eqref{eq.reduced} and we keep only the solutions with $\dot T > 0$. Note that along each closed orbit there are two solutions, one with $\dot T > 0$ and one with $\dot T < 0$.

Notice that
\begin{align*}
L &= (2(j-z)(1-z^2))^{1/2} \cos(\phi-\theta), \\
T &= (2(j-z)(1-z^2))^{1/2} \sin(\phi-\theta).
\end{align*}
Since $T=0$ we find $\phi-\theta=k\pi$, $k\in \mathbb{Z}$. If for the solution with $\dot{T}>0$ we have $L >0$ then we choose $\phi=\theta=0$. Otherwise if $L<0$ we choose $\phi=\pi$, $\theta=0$. Moreover, $I = j-z$. Therefore, we have initial conditions on the fiber $T^2_{j,h}$ expressed as $(I,\phi,z,\theta)$ and we use these as initial conditions to integrate the equations of motion in the phase space $S^2 \times \mathbb{R}^2$. We continue the integration until $\sin(\phi(t)-\theta(t))$ changes sign from negative to positive, corresponding to $T=0$, $\dot T > 0$, that is, the orbit in the reduced space has made one revolution. The time it takes for this to happen is the first return time denoted by $\tau$.

Even though the orbit traced in time $\tau$ in the reduced space is closed, the orbit in the original phase space in general does not close: after time $\tau$ the orbit has reached a point $(I,\tilde{\phi},z,\tilde{\theta})$.
To obtain a closed curve representing the homology cycle $\gamma$ we consider the orbit of $X_J$ from $(I,\tilde{\phi},z,\tilde{\theta})$ for time $-\tilde{\theta}$ to the original initial condition $(I,\phi,z,0)$. 
Then $\vartheta = \tilde{\theta}$ gives the rotation number.

Therefore, we have expressed the loop $\gamma$ as a concatenation of $\gamma_1$ and $\gamma_2$ where $\gamma_1$ is a solution curve of $X_{H+G}$ and $\gamma_2$ is a solution curve of $X_J$. 
The second classical action is, therefore, given by
\[ K = \frac{1}{2\pi} \int_\gamma \alpha = \frac{1}{2\pi} \int_{\gamma_1} \alpha + \frac{1}{2\pi} \int_{\gamma_2} \alpha.  \]
The second integral is given by
\[ \int_{\gamma_2} \alpha = -\vartheta J. \]
The first integral can be computed numerically by integrating $\alpha$ together with the equations of motion as we compute $\gamma_1$.
Therefore, we can compute $K$ combining these two results.

\subsection{Classical actions for the system with a flap with two elliptic-elliptic values}
\label{s.actionangle2llipticflap}
In order to compute the classical actions for the system $(W_{1,1,2},\omega_{1,1,2},(J,H_{0.44}))$ defined in Section \ref{s.systemflapwith2elliptic} we use an approach similar to Section \ref{s.actionJaynesCummings}. Recall that $W_{1,1,2}$ is obtained by $\mathbb{T}^2$ reduction from $\mathbb{C}^4$. Consider the complex coordinates $z_k=p_k+iq_k$, the symplectic polar coordinates $z_k=\sqrt{2I_k}\exp(i\phi_k)$, $k=1,...,4$ and the symplectic form 
\begin{equation*}
    \omega=\sum_{i=1}^{4}d\phi_i\wedge dI_i=-d\left(\sum_{i=1}^{4}I_id\phi_i\right)=:-d(\alpha).
\end{equation*}
The momentum map of the $\mathbb{T}^2$ reduction is $N=(N_1,N_2)$, where $N_1=I_1+I_2+2I_3$ and $N_2=I_3+I_4$. 
Recall that the momentum of the $S^1$-action of $J$ is $I_2+I_3$. The invariant functions for the $\mathbb{T}^3$ action of $(N_1,N_2,J)$ are 
\begin{equation*}
    \begin{cases}
        R=I_3-I_4,\\
        N_1=I_1+I_2+2I_3,\\
        N_2=I_3+I_4,\\
        J=I_2+I_3,\\
        X=4\sqrt{I_1I_2I_3I_4}\cos(-\phi_1-\phi_2+\phi_3-\phi_4),\\
        Y=4\sqrt{I_1I_2I_3I_4}\sin(-\phi_1-\phi_2+\phi_3-\phi_4).
    \end{cases}
\end{equation*}
They satisfy the relation 
\begin{equation*}
    X^2+Y^2=\frac{(N_2+R)}{2}\frac{(N_2-R)}
  {2}\left(J -\frac{(N_2 + R)}{2}\right) \left(N_1-J-\frac{(N_2 + R)}{2}\right).
\end{equation*}
Consider a regular value $(j,h)$.
First we identify a point in the reduced space $J^{-1}(j)/S^1$ with $Y=0$ such that, for $t=0.44$,
\begin{equation*}
    h=(1-t)R+\frac{s}{3}\left( \frac{9}{20}X+(2j-3)(R+2)\right)+8tI_3I_4
\end{equation*}
by solving the equation
\begin{align*}
   \biggl( (-1+2j-R)&(-1+R)(1+R)(-5+2j+R) \biggl) \\
   & -\frac{400 \biggl( 3h-4jt+6R^2t+R(-3+6t-2jt)           \biggl)^2}{81t^2}=0
\end{align*}
for $-1\leq R\leq \min\{3-j,1\}$. Then we compute $\dot{Y}:=\frac{d}{ds}(Y\circ \phi_{s}^{H_{red}})$ at each of the solutions and we keep only the solutions with $\dot{Y}>0$. If for the solution with $\dot{Y}>0$ we have $X>0$ we choose $\phi_1=\phi_2=\phi_3=\phi_4=0$. Otherwise if $X<0$ we choose $\phi_1=\phi_2=\phi_4=0$ and $\phi_3=\pi$. Using the fact that
\begin{equation*}
    \begin{cases}
        I_1=3-j-\frac{1+R}{2},\\
        I_2=j-\frac{(1+R)}{2},\\
        I_3=\frac{(1+R)}{2},\\
        I_4=\frac{(1-R)}{2}
    \end{cases}
\end{equation*}
we have initial conditions $(\phi_1,\phi_2,\phi_3,\phi_4,I_1,I_2,I_3,I_4)$ on the fiber $(J,H_{0.44})^{-1}(j,h)$ and we use these initial conditions to express the equations of motion.

We integrate the equations of motion until $\sin(-\phi_1(t)-\phi_2(t)+\phi_3(t)-\phi_4(t))$ changes sign from negative to positive, that is the orbit in the reduced space has made one revolution. Let us denote the time it takes for this to happen by $\tau$.

Recall that even though the orbit in the reduced space has made one revolution, the orbit in the original space in general does not close after time $\tau$. It has reached a point $(\tilde{\phi_1},\tilde{\phi_2},\tilde{\phi_3},\tilde{\phi_4},I_1,I_2,I_3,I_4)$. 

To obtain a closed curve representing a homology cycle we consider $3$ additional orbits. The orbit of $X_{N_1}$ for time $-\tilde{\phi_1}$. The orbit of $X_{N_2}$ for time $-\tilde{\phi_4}$. And the orbit of $X_{J}$ for time $-(\tilde{\phi_2}-\tilde{\phi_1})$.

Let $\gamma_1$ be the solution curve of $X_{H_{0.44}}$ doing one revolution in the reduced space for the given initial value conditions. Then the second classical action is given as 
\begin{equation*}
    \frac{1}{2\pi}\biggl(\int_{\gamma_1}\alpha-3\tilde{\phi_1}-\tilde{\phi_4}-j(\tilde{\phi_2}-\tilde{\phi_1})\biggl).
\end{equation*}
Since $J$ generates an effective $S^1$-action, the first action is $J$.

\subsection{Classical actions for the system with a pleat}
\label{s.actionpleat}
Recall the system given by $(W_{1.02,1,1},\omega_{1.02,1,1},(J,H_1))$ defined in Section \ref{s.pleatexample}. In this appendix we compute the classical actions for this system. We present a different approach from the one in Appendix \ref{s.actionJaynesCummings} and Appendix \ref{s.actionangle2llipticflap}.
Recall that
\begin{equation*}
    J=\frac{1}{2}|z_2|^2, \quad N_1=\frac{1}{2}(|z_1|^2+|z_2|^2+|z_3|^2), \quad N_2=\frac{1}{2}(|z_3|^2+|z_4|^2).
\end{equation*}
Using the $S^1$-actions generated by $J,N_1$ and $N_2$ in $\mathbb{R}^8$, the reduced space $J^{-1}(j)/S^1$ can be identified with the space described by the equations
\begin{equation*}
    \begin{cases}
        \pi_5^2+\pi_6^2=\pi_1\pi_3\pi_4, \\
        \pi_1+\pi_2+\pi_3=4.04,\\
        \pi_3+\pi_4=2,
    \end{cases}
\end{equation*}
where $\pi_k$, for $k=1,...,6$ are the following functions in $\mathbb{R}^8$:
\begin{equation*}
    \begin{cases}
    \pi_1:=x_1^2+y_1^2,\\
    \pi_2:=x_2^2+y_2^2,\\
    \pi_3:=x_3^2+y_3^2,\\
    \pi_4:=x_4^2+y_4^2,\\
    \pi_5:=\Re(\overline{z_1}z_3\overline{z_4}),\\
    \pi_6:=\Im(\overline{z_1}z_3\overline{z_4}),
    \end{cases}
\end{equation*}
and $z_j=x_j+iy_j$. For $t=1$ we obtain $H:=H_1=-\frac{\pi_3}{2}+\pi_5+2\pi_1\pi_3$. We compute that 
\begin{equation*}
    \dot{\pi_3}:=\{\pi_3,H\}=-2\pi_6.
\end{equation*}
Using these results we can compute the first return time $T(j,h)$,
\begin{equation*}
T(j,h)=\int_{0}^{T(j,h)}dt=2\int_{\pi_3^{-}}^{\pi_3^{+}}\frac{d\pi_3}   {\dot{\pi_3}}=\int_{\pi_3^{-}}^{\pi_3^{+}}\frac{d\pi_3}{\sqrt{S(j,h)}}
\end{equation*}
where $S(j,h)=(4.04-2j-\pi_3)\pi_3(2-\pi_3)-(h+\frac{\pi_3}{2}-2(4.04-2j-\pi_3)\pi_3)^2$ is a polynomial in $\pi_3$ of $4$th degree and $\pi_3^{\pm}$ are real roots of $S(j,h)$. Notice that for some values of $(j,h)$ there exist $4$ real roots of $S(j,h)$ which corresponds to the existence of disconnected components in the preimages of regular values in the swallowtail.

Recall that the combined $S^1$-actions given by $J,N_1$ and $N_2$ are
\begin{equation*}
    (t,\theta_1,\theta_2)\subset \mathbb{T}^3\mapsto (z_1e^{i\theta_1},z_2e^{i(\theta_1+t)},z_3e^{i(\theta_1+\theta_2)},z_4e^{i\theta_2}) 
\end{equation*}
where $t$ corresponds to the $S^1$-action  induced by $J$, $\theta_1$ to the $S^1$-action induced by $N_1$ and $\theta_2$ to the $S^1$-action induced by $N_2$.
The function $\eta_J(z_1,z_2,z_3,z_4)=\arg(z_2)-\arg(z_1)$ measures the action of $J$ relative to the actions of $N_1$ and $N_2$. We compute 
\begin{equation*}
    \dot{\eta}_J:=\{\eta_J,H\}=\frac{\pi_5}{\pi_1}+4\pi_3,
\end{equation*}
hence:
\begin{align*}
\theta_J(j,h)&=\int_{0}^{\theta_J(j,h)}d\eta_J=\int_0^{T(j,h)}\dot{\eta_J}dt
\\ &=\int_{\pi_3^{-}}^{\pi_3^{+}}\left(4\pi_3+\frac{h+\frac{\pi_3}{2}-2(4.04-2j-\pi_3)\pi_3}{(4.04-2j-\pi_3)}\right)\frac{d\pi_3}{\sqrt{S(j,h)}}.
\end{align*}
To obtain $\theta_{N_1}$ we consider the function $\eta_{N_1}(z_1,z_2,z_3,z_4)=\arg(z_1)$. And since $\dot{\eta_{N_1}}=-\dot{\eta_{J}}$ we find $\theta_{N_1}(j,h)=-\theta_{J}(j,h)$.

Finally to obtain $\theta_{N_2}$ we consider $\eta_{N_2}(z_1,z_2,z_3,z_4)=\arg(z_3)-\arg(z_1)$ and we have
\begin{equation*}
    \dot{\eta}_{N_2}:=\{\eta_{N_2},H\}=1-\frac{(\pi_1-\pi_3)\pi_5}{\pi_1\pi_3}-4\pi_1+4\pi_3,
\end{equation*}
hence
\begin{align*}
    &\theta_{N_2}(j,h) \\
    \nonumber &=\int_{\pi_3^{-}}^{\pi_3^{+}}\left(1-4(4.04-2j-\pi_3)+4\pi_3-\frac{(h+\frac{\pi_3}{2}-2(4.04-2j-\pi_3)\pi_3)}{(4.04-2j-\pi_3)\pi_3}\right)\frac{d\pi_3}{\sqrt{S(j,h)}}.
\end{align*}
Using these results we compute the classical actions, which depend on the choice of tori on the disconnected components of the swallowtail. The classical actions are given by the formula
\begin{equation*}
    \label{eq.actionanglepleat}
    \left(j,\quad \frac{1}{2\pi}\left(\int_{\gamma_1}\alpha-2.02\ \theta_{N_1}(j,h)-\theta_{N_2}(j,h)-j\theta_J(j,h)\right)\right)
\end{equation*}
where $\gamma_1$ is a orbit of $X_{H}$ for the initial value conditions given by $(j,h)$ and period $T(j,h)$.

\begin{remark} 
When doing the computations in Mathematica we noticed that applying a method analogous to the methods of Section \ref{s.actionJaynesCummings} and Section \ref{s.actionangle2llipticflap} was faster than using the methods described above. 
\end{remark}

\subsection{Classical actions for the system with curled tori}.
\label{s.actioncurledtori}
Recall the system $(W_{1,1,2},\omega_{1,1,2},(J,X))$ defined in Section \ref{s.examplecurledtori}.
In this section we compute the classical actions of the system $(W_{1,1,2},\omega_{1,1,2},(J,X))$, using a method analogous to the one found in Appendix \ref{s.actionJaynesCummings} and Appendix \ref{s.actionangle2llipticflap}.
Recall that $W_{1,1,2}$ is obtained by a $\mathbb{T}^2$ reduction of $\mathbb{C}^4$. We consider the symplectic polar coordinates $z_k=\sqrt{2I_k}\exp(\phi_k)$, $k=1,...,4$, for the symplectic form 
\begin{equation*}
    \omega=\sum_{i=1}^{4}d\phi_i\wedge dI_i=-d\left(\sum_{i=1}^{4}I_id\phi_i\right)=:-d(\alpha).
\end{equation*}
The action of the reduction has momentum map $N=(N_1,N_2)$, where $N_1=I_1+I_2+2I_3$ and $N_2=I_3+I_4$. The invariant functions of the $\mathbb{T}^3$ action of $(N_1,N_2,J)$ are
\begin{equation*}
    \begin{cases}
        R=I_3,\\
        N_1=I_1+I_2+2I_3,\\
        N_2=I_3+I_4,\\
        J=I_2,\\
        X=4\sqrt{I_1^2I_3I_4}\cos(-2\phi_1+\phi_3-\phi_4),\\
        Y=4\sqrt{I_1^2I_3I_4}\sin(-2\phi_1+\phi_3-\phi_4).
    \end{cases}
\end{equation*}
They satisfy the relation
\begin{equation*}
    X^2+Y^2=16(N_2-R)R(J-N_1+2R)^2.
\end{equation*}
Consider a regular value $(j,h)$. First we identify a point in the reduced space $J^{-1}(j)/S^1$ such that 
$h=X$ by solving the equation
\begin{equation*}
    16(j-3+2R)\biggl( (-j+3)+2(j-6)R+8R^2\biggl)=0
\end{equation*}
for $0\leq R\leq \min\{1,\frac{(3-j)}{2}\}$. Then we compute $\dot{Y}:=\frac{d}{dt}(Y\circ \phi_t^{X_{red}})$ at each of the solutions and keep only the solutions with $\dot{Y}>0$. If for the solution with $\dot{Y}>0$ we have $X>0$ we choose $\phi_i=0$ for $i=1,...,4$. Otherwise if $X<0$ we choose $\phi_i=0$ for $i=1,2,4$ and $\phi_3=\pi$. Using the fact that 
\begin{equation*}
    \begin{cases}
        I_1=3-j-2R,\\
        I_2=j,\\
        I_3=R,\\
        I_4=1-R,
    \end{cases}
\end{equation*}
we have initial conditions on the fiber $(J,X)^{-1}(j,h)$ and we use these initial conditions to express the equations of motion. 

We integrate the equations of motion until $\sin(-2\phi_1(t)+\phi_3(t)-\phi_4(t))$ changes sign from negative to positive, that is the orbit in the reduced space has made one revolution. Denote this first return time by $\tau$. 

Recall that even though the orbit in the reduced space has made one revolution, in the original space, in general, the orbit does not close after time $\tau$. It has reached a point $(\tilde{\phi_1},\tilde{\phi_2},\tilde{\phi_3},\tilde{\phi_4},I_1,I_2,I_3,I_4)$. To obtain a closed orbit we need to consider $3$ additional orbits, namely the orbit of $X_{N_1}$ for time $-\tilde{\phi_1}$, the orbit of $X_{N_2}$ for time $-\tilde{\phi_4}$ and the orbit of $X_{J}$ for time $-(\tilde{\phi_2}-\tilde{\phi_1})$.

Let $\gamma_1$ be the solution curve of $X_{X}$ doing one revolution in the reduced space. Then the second action is given by 
\begin{equation*}
    \frac{1}{2\pi}\biggl(\int_{\gamma_1}\alpha -3\tilde{\phi_1}-\tilde{\phi_4}-j(\tilde{\phi_2}-\tilde{\phi_1}) \biggl).
\end{equation*}
Since $J$ generates an effective $S^1$-action, the first action is simply given by $J$.

\section{Quantization of the systems}
\label{s.quantizationofthesystems}
In this Appendix we show how to quantize several systems introduced throughout the main body of the article. 
\subsection{Quantization of the system with a flap with two elliptic-elliptic values}
\label{s.quantizationflap2ellitptic}
Recall the system $(W_{1,1,2},\omega_{1,1,2},(J,H_{t}))$ defined in section \ref{s.systemflapwith2elliptic}. 
The quantization $\mathcal{H}_{1,1,2}$ of $W_{1,1,2}$, when it is nontrivial, is generated by the monomials $z_1^{\alpha_1}z_2^{\alpha_2}z_3^{\alpha_3}z_4^{\alpha_4}$ such that 
\begin{equation*}
    \begin{cases}
        \hbar(\alpha_1+\alpha_2+2\alpha_3+2)=3,\\
        \hbar(\alpha_3+\alpha_4+1)=1.
    \end{cases}
\end{equation*}
Furthermore, the quantization of $J$ is $\hat{J}=\hbar(z_2\frac{\partial}{\partial z_2}+z_3\frac{\partial}{\partial z_3}+1)$. The quantization of $R$ is $\hat{R}=\hbar(z_3\frac{\partial}{\partial z_3}-z_4\frac{\partial}{\partial z_4})$. The quantization of $X$ is $\hat{X}=2\hbar^2(z_1z_2\frac{\partial}{\partial z_3}z_4+\frac{\partial}{\partial z_1}\frac{\partial}{\partial z_2}z_3\frac{\partial}{\partial z_4})$. Therefore, the quantization of $H_t$ is 
\begin{align*}
    \hat{H_t} =(1-t)\hat{R}+\frac{t}{3}\left(\frac{9}{20}\hat{X}+(2\hat{J}-3)(\hat{R}+2)\right)+8t\hbar^2\left(z_3\frac{\partial}{\partial z_3}+\frac{1}{2}\right) \left(z_4\frac{\partial}{\partial z_4}+\frac{1}{2}\right),
\end{align*}
with all operators mentioned above restricted to $\mathcal{H}_{1,1,2}$. The operators $\hat{J}$ and $\hat{H}_t$ commute. We do the analogous of Section \ref{s.quantizationhirzebruch} to obtain the joint spectrum of the system $(J,H_t)$. Let
\begin{equation*}
e_{k}^a:=\frac{z_1^{\alpha_1(k)}z_2^{\alpha_2(k)}z_3^kz_4^{\alpha_4(k)}}{\sqrt{\alpha(k)!}}=\frac{z_1^{\alpha_1(k)}z_2^{\alpha_2(k)}z_3^kz_4^{\alpha_4(k)}}{\sqrt{(\alpha_1(k)!)(\alpha_2(k)!)(\alpha_3(k)!)(\alpha_4(k)!)}}
\end{equation*}
for a fixed $a\in \text{spec}(\hat{J})$, where $\alpha_1(k),\alpha_2(k),\alpha_4(k)$ solve the following equations
\begin{equation*}
    \begin{cases}
        \hbar(\alpha_1(k)+\alpha_2(k)+2k)=3,\\
        \hbar(k+\alpha_4(k)+1)=1,\\
        \hbar(\alpha_2(k)+k+1)=a.
    \end{cases}
\end{equation*}
For fixed $\hbar$, such that the previous equations are solved, the $e_a^k$ form a orthonormal basis of $\mathcal{E}_a:=\ker(\hat{J}-a)$ where 
\begin{equation*}
    \begin{cases}
        k=0,...,\lfloor a/\hbar\rfloor -1, \quad a\leq 1,\\
        k=0,...,\lfloor 1/\hbar \rfloor -1, \quad 1\leq a\leq 2,\\
        k=0,...,\lfloor (3-a)/\hbar \rfloor -1, \quad 2\leq a\leq 3.
    \end{cases}
\end{equation*}
Furthermore,
\begin{align*}
    \hat{X}(e_k^a)&=2\hbar^2\sqrt{(\alpha_1(k)+1)(\alpha_2(k)+1)(\alpha_4(k)+1)k}\ e_{k-1}^a
    \\
    & \quad +  2\hbar^2\sqrt{\alpha_1(k)\alpha_2(k)\alpha_4(k)(k+1)}\ e_{k+1}^{a}
\end{align*}
and
\begin{equation*}
    \hat{R}(e^{a}_k)=\hbar(k-\alpha_4(k))e_k^{a}, \quad \hat{J}(e_k^{a})=\hbar(a_2(k)+k+1)e_k^{a}.
\end{equation*}

\subsection{Quantization of the system with a pleat}
\label{s.quantizationpleat}
Recall the hypersemitoric system $(W_{1.02,1,1},\omega_{1.02,1,1},(J,H_t))$ of Section \ref{s.pleatexample} and the quantization procedure of Section \ref{s.quantizationhirzebruch}. We have
\begin{equation*}
    \hat{J}=\hbar\left(z_2\frac{\partial}{\partial z_2}+\frac{1}{2}\right).
\end{equation*}

Let us quantize the Hamiltonian $H_t$. Rewrite
\begin{equation*}
    X=\frac{1}{2}\left(\overline{z_1}z_3\overline{z_4}+z_1\overline{z_3}z_4\right).
\end{equation*}

Then replace $z_k$ by $\sqrt{2\hbar}\frac{\partial}{\partial z_k}$ and $\overline{z_k}$ by $\sqrt{2\hbar}z_k$ to obtain 
\begin{align*}
    \hat{H_t} & =(1-2t)\hbar\left(z_3\frac{\partial}{\partial z_3}+\frac{1}{2}\right)+t\sqrt{2}\hbar^{\frac{3}{2}}\left(z_1\frac{\partial}{\partial z_3}z4 +\frac{\partial}{\partial z_1}z_3\frac{\partial}{\partial z_4}\right) \\
    &\quad +8t\hbar^2\left(z_1\frac{\partial}{\partial z_1}+\frac{1}{2}\right)\left(z_3\frac{\partial}{\partial z_3}+\frac{1}{2}\right)
\end{align*}
in the Bargmann representation. Furthermore, the operators $\hat{J}$ and $\hat{H_t}$ commute. In order to understand the joint spectrum of $(J,H_t)$ we need to find the eigenvalues of $H_t$ when restricted to $\mathcal{E}_{a_2}:=\ker(\hat{J}-a_2)$ for $a_2\in \text{spec}(\hat{J})$. Therefore, we need to understand how $\hat{H}_t$ acts on the basis elements $z_1^{\alpha_1}z_2^{\alpha_2}z_3^{\alpha_3}z_4^{\alpha}$. For a fixed $a_2\in \text{spec}(\hat{J})$ let $\alpha_2=\frac{a_2}{\hbar}-\frac{1}{2}$ and 
\begin{equation*}
e^{a_2}_k:=\frac{z_1^{\alpha_1(k)}z_2^{\alpha_2}z_3^{k}z_4^{\alpha(k)}}{\sqrt{\alpha(k)!}}=\frac{z_1^{\alpha_1(k)}z_2^{\alpha_2}z_3^{k}z_4^{\alpha(k)}}{\sqrt{(\alpha_1(k)!)(\alpha_2!)(k!)(\alpha_4(k)!)}}
\end{equation*}
where $0\leq k \leq \lfloor \frac{1}{\hbar}\rfloor -1$ and $\alpha_1(k),\alpha_4(k)$ are such that the following equations are satisfied:
\begin{equation*}
    \begin{cases}
        \hbar(\alpha_1(k)+\alpha_2+k+\frac{3}{2})=2.02,\\
        \hbar(k+\alpha_4(k)+1)=1.
    \end{cases}
\end{equation*}
Evaluating $\hat{H_t}$ on $e^{a_2}_k$ we have:
\begin{align*}
    \hat{H_t}(e^{a_2}_k)&=\left(\frac{1}{2}\hbar(1+2k)(1-2t+4\hbar t)+\hbar^2t\alpha_1(k)4(1+2k)\right)e^{a_2}_k \\ & \quad + \left( \hbar^{3/2} \sqrt{2}t\sqrt{(\alpha_1(k)+1)(\alpha_4(k)+1)k}\right)e^{a_2}_{k-1}\\
    & \quad  +\left(\hbar^{3/2}\sqrt{2}t\sqrt{\alpha_1(k)\alpha_4(k)(k+1)}\right)e^{a_2}_{k+1}.
\end{align*}

\subsection{Quantization of the system with curled tori}
\label{s.quantizationcurledtori}
Recall the integrable system $(W_{1,1,2},\omega_{1,1,2},(J,X))$ defined in Section \ref{s.examplecurledtori} and the quantization procedure of Section \ref{s.quantizationhirzebruch}. The quantization of $J=\frac{1}{2}|z_2|^2$ is given by
\begin{equation*}
    \hat{J}=\hbar\left(z_2\frac{\partial}{\partial z_2}+\frac{1}{2}\right).
\end{equation*}
Using techniques analogous to the ones in Section \ref{s.quantizationpleat}, mainly $\overline{z_i}\mapsto \sqrt{2\hbar}z_i$ and $z_i\mapsto \sqrt{2\hbar}\frac{\partial}{\partial z_i}$, the quantization of $X$ is given by 
\begin{equation*}
    \hat{X}:=2\hbar^2\left(z_1^2\frac{\partial}{\partial z_3}z_4+\frac{\partial ^2}{\partial^2 z_1 }z_3\frac{\partial}{\partial z_4}\right).
\end{equation*}
Note that the operators $\hat{J},\hat{X}$ commute. Recall that the Hilbert space of the quantization of $W_{1,1,2}$, when it is nontrivial, is generated by the monomials $\frac{z_1^{\alpha_1}z_2^{\alpha_2}z_3^{\alpha_3}z_4^{\alpha_4}}{\sqrt{\alpha!}}$ such that 
\begin{equation*}
    \begin{cases}
        \hbar(\alpha_1+\alpha_2+2\alpha_3+2)=3,\\
        \hbar(\alpha_3+\alpha_4+1)=1.
    \end{cases}
\end{equation*}
As in Section \ref{s.quantizationpleat} let us consider, for a fixed value of $a_2\in \text{spec}(\hat{J})$, the basis elements $e_k^{a_2}$ of $\mathcal{E}_{a_2}$. A computation gives 
\begin{align*}
    \hat{X}(e^{a_2}_k)= \ &2\hbar^2\left(\sqrt{\alpha_4(k)}\sqrt{k+1}\sqrt{\alpha_1(k)}\sqrt{\alpha_1(k)-1}\right)e^{a_2}_{k+1}
    \\
    & +2\hbar^2\left(\sqrt{k}\sqrt{\alpha_4(k)+1}\sqrt{\alpha_1(k)+1}\sqrt{\alpha_1(k)+2}\right)e_{k-1}^{\alpha_2},
\end{align*}
where $\alpha_1(k),\alpha_4(k)$ solve the following equations:
\begin{equation*}
    \begin{cases}
        \hbar(\alpha_1(k)+\frac{a_2}{h}-\frac{1}{2}+2k+2)=3,\\
        \hbar(k+\alpha_4(k)+1)=1.
    \end{cases}
\end{equation*}

\vspace{20mm}
\printbibliography
\vspace{10mm}

\noindent
Konstantinos Efstathiou \\
\\
Duke Kunshan University \\
Division of Natural and Applied Sciences and \\
Zu Chongzhi Center for Applied Mathematics and Computational Sciences \\
8 Duke Avenue\\
215316 Kunshan, China\\
\\
{\em E\--mail}: \texttt{k DOT efstathiou AT dukekunshan DOT edu DOT cn} \\
\\
\noindent 
  Sonja Hohloch \& Pedro Santos\\
  \\
  University of Antwerp\\
  Department of Mathematics\\
  Middelheimlaan 1\\
  B-2020 Antwerpen, Belgium\\
  \\
  {\em E\--mail}: \texttt{sonja DOT hohloch AT uantwerpen DOT be} \\
  {\em E\--mail}: \texttt{pedro DOT santos AT uantwerpen DOT be}

\end{document}